\documentclass[12pt, reqno]{amsart}
\usepackage{amssymb,amsthm,amsmath}
\usepackage[numbers,sort&compress]{natbib}
\usepackage{amssymb,amsmath}
\usepackage{amsfonts}
\usepackage{mathrsfs}
\usepackage{latexsym}
\usepackage{amssymb,xcolor}
\usepackage{amsthm}
\usepackage{indentfirst}
\date{
February 9, 2020}

\let\oldsection\section
\renewcommand\section{\setcounter{equation}{0}\oldsection}

\newtheorem{corollary}{Corollary}[section]
\newtheorem{theorem}{Theorem}[section]
\newtheorem{lemma}{Lemma}[section]
\newtheorem{proposition}{Proposition}[section]
\newtheorem{definition}{Definition}[section]

\newtheorem{remark}{Remark}[section]

\allowdisplaybreaks
\begin{document}

\title[Entropy-bounded solutions compressible Navier--Stokes]{Entropy-bounded solutions to the one-dimensional heat
conductive compressible Navier--Stokes equations with far field vacuum}

\author{Jinkai~Li}
\address[Jinkai~Li]{South China Research Center for Applied Mathematics and Interdisciplinary Studies, South China Normal University, Guangzhou 510631, China}
\email{jklimath@m.scnu.edu.cn; jklimath@gmail.com}

\author{Zhouping~Xin}
\address[Zhouping~Xin]{The Institute of Mathematical Sciences,
The Chinese University of Hong Kong, Hong Kong, China}
\email{zpxin@ims.cuhk.edu.hk}



\begin{abstract}
In the presence of vacuum, the physical entropy for polytropic gases behaves singularly and it is thus a challenge to study
its dynamics. It is shown in this paper that the boundedness of the entropy can be propagated up to any finite time
provided that the initial vacuum presents only at far fields with sufficiently slow decay of the initial density. More precisely, for the Cauchy problem of the onedimensional heat conductive compressible Navier--Stokes
equations, the global well-posedness of strong solutions and uniform boundedness of
the corresponding entropy are established, as long as the initial density vanishes only at far fields with a rate no more than $O(\frac{1}{x^2})$.
The main tools of proving the uniform boundedness of the entropy are some singularly weighted energy estimates carefully designed for the heat conductive
compressible Navier--Stokes equations
and an elaborate De Giorgi type iteration technique for some classes of
degenerate parabolic equations. The De Giorgi type iterations are carried out to
different equations in establishing the lower and upper bounds of the entropy.
\end{abstract}

\maketitle


\section{Introduction}
\label{secIntro}
\subsection{The compressible Navier--Stokes equations}
The one dimensional heat conductive compressible Navier--Stokes equations for the
polytropic gases are:
\begin{eqnarray}
  &&\rho_t+(\rho u)_x=0,\label{erho}\\
  &&\rho(u_t+uu_x)-\mu u_{xx}+p_x=0,\label{eu}\\
  &&\rho(e_t+ue_x)+pu_x-\kappa\theta_{xx}=\mu |u_x|^2, \label{ee}
\end{eqnarray}
where the density $\rho\geq0$, the velocity $u\in\mathbb R$, and the absolute temperature
$\theta\geq0$ are the unknowns, and the specific internal energy
$e$ and the pressure $p$ are expressed as
$$
e=c_v\theta, \quad p=R\rho\theta,
$$
with $R$ and $c_v$ being positive constants, $\mu$ and $\kappa$ are the viscous
and heat conductive coefficients, respectively, which are assumed to be positive constants.

In terms of $\vartheta$, the energy equation becomes
\begin{equation}
  c_v\rho(\theta_t+u\theta_x)+pu_x-\kappa\theta_{xx}=\mu |u_x|^2. \label{etheta}
\end{equation}
The entropy $s$ is defined by the Gibb's equation $\theta Ds=De+pD(\frac1\rho)$. The following equations of state hold:
\begin{equation}
  p=Ae^{\frac{s}{c_v}}\rho^\gamma,\quad s=c_v\left(\log\frac{R}{A}+\log\theta-(\gamma-1)\log\rho\right), \label{entropy}
\end{equation}
where $\gamma-1=\frac{R}{c_v}$ and $\gamma>1$. The entropy $s$ satisfies
\begin{equation}
  \label{es}
  \rho(s_t+us_x)-\kappa\left(\frac{\theta_x}{\theta}\right)_x
  =\frac1\theta\left(\mu |u_x|^2+\kappa\frac{|\theta_x|^2}{\theta}\right),
\end{equation}
at the place where both $\rho$ and $\theta$ are positive.

The compressible Navier--Stokes equations have been studied extensively.
In the absence of vacuum, that is,
the density is uniformly positive, local
well-posedness of classic or strong solutions
was first proved by Nash in \cite{NASH62} long time ago, and later by many mathematicians, see, e.g.,
\cite{ITAYA71,VOLHUD72,TANI77,VALLI82,LUKAS84}. However, the global existence
of classic or strong solutions with arbitrary
large initial data is not known generally. Only the
one-dimensional theory is quite satisfactory:
global well-posedness of strong solutions was proved by Kazhikhov--Shelukhin \cite{KAZHIKOV77} and Kazhikhov
\cite{KAZHIKOV82}; global well-posedness in the framework of weak
solutions can be also proved, see, e.g., \cite{ZLOAMO97,ZLOAMO98,CHEHOFTRI00,JIAZLO04}; large time behavior
of solutions with large initial data was recently proved in \cite{LILIANG16}.
Compared with the one-dimensional case, the multidimensional case is much more
complicated, and up to now, essentially only for the cases that the initial data
is around some non-vacuum equilibrium, the global well-posedness is well understood. The results along this
direction were first obtained by Matsumura--Nishida
\cite{MATNIS80,MATNIS83}, and later developed by many mathematicians, see, e.g.,   \cite{PONCE85,VALZAJ86,DECK92,HOFF97,KOBSHI99,DANCHI01,CMZ,CHIDAN15,DANXU17,FZZ18}.

One major difference between the one-dimensional and multidimensional cases
for the compressible Navier--Stokes equations is
the possible formation of vacuum.
As shown by Hoff-Smoller \cite{HS}, for the 1D compressible Navier--Stokes equations,
if there is no vacuum initially, then no vacuum will be formed later in finite time,
while such a result is still open for the multidimensional case. The possible
formation of vacuum is one of the main challenges.

In the presence of vacuum, the study of the compressible Navier--Stokes equations becomes
much more difficult than the non-vacuum case due to the degeneracy of the system.
Global existence of weak solutions to the isentropic fluids with possible vacuum
was first initiated by Lions \cite{LIONS98}, and
later improved by Feireisl--Novotn\'y--Petzeltov\'a \cite{FEIREISL01}
and further by Jiang--Zhang \cite{JIAZHA03}.
For the full case, global existence of variational weak solutions was proved
by Feireisl \cite{FEIREISL04B} for special equations of state. Local well-posedness of strong solutions was proved in
\cite{SALSTR93,CHOKIM04,CHOKIM06-1,CHOKIM06-2}. Global existence of strong solutions, of small energy but
allowing large oscillations and vacuum, was first proved by Huang--Li--Xin \cite{HLX12} for the isentropic case, and
generalized later by the authors in \cite{HUANGLI11,WENZHU17,LIGLOBALSMALL} for the full case.

There are some substantial differences in the mathematical theories for the compressible Navier--Stokes equations between the vacuum
and non-vacuum cases. First, in the absence of vacuum, the well-posedness holds in both the homogeneous and
inhomogeneous spaces, but it is not necessarily true if the vacuum appears. In fact, if the density is compactly supported,
then the well-posedness holds in the homogeneous spaces, see, e.g., \cite{CHOKIM04,CHOKIM06-1,CHOKIM06-2,HLX12,HUANGLI11,WENZHU17}, but not in the inhomogeneous spaces, see Li--Wang--Xin \cite{LWX}, while if the density tends to
zero sufficiently slowly at the far field, then the well-posedness holds in both the homogeneous and inhomogeneous spaces, see the recent work
by the authors \cite{LIXIN17}. Second, the solution spaces guarantee the uniform boundedness of the entropy for the non-vacuum case, but may fail
for the vacuum case. In fact, it follows from the blowup results of Xin \cite{XIN98} and Xin--Yan \cite{XINYAN13} that the corresponding entropy in \cite{HUANGLI11,WENZHU17} must be
unbounded, if initially
there is an isolated mass group surrounded by the vacuum region.


Due to the lack of the expression of the entropy in the vacuum region and the high
singularity and degeneracy of the entropy equation close to the vacuum region, in spite of its importance, the mathematical analysis of the entropy for the viscous compressible fluids in the presence of vacuum was rarely carried out before.
In this paper, we continue our studies, initiated in \cite{LIXIN17},
on the uniform boundedness of the entropy for the full compressible Navier--Stokes equations
in the presence of vacuum.
Different from the non heat conductive case in \cite{LIXIN17},
for the heat conductive case, one may only need to deal with the
the far field vacuum, as the heat conductivity will make the temperature strictly
positive everywhere after the initial time, which implies that the entropy becomes unbounded instantaneously
if the interior vacuum occurs initially. However, positive heat conductivity leads to both increase and decrease of the
entropy and thus creates substantial difficulties in the analysis compared with \cite{LIXIN17}.


The results of this paper are stated and proved in the Lagrangian coordinates, see Section \ref{ssecmainresult};
however, since the solutions being established are Lipschitz continuous, all results can be transformed accordingly in the Euler coordinates.

\subsection{Main results and key ideas of the analysis}
\label{ssecmainresult}
Let $y$ be the Lagrangian coordinate and define the coordinate transform
between $y$ and the Euler coordinate $x$ as
$x=\eta(y,t)$ with $\eta(y,t)$ satisfying
\begin{equation*}\label{flowmap}
  \left\{
  \begin{array}{l}
  \partial_t\eta(y,t)=u(\eta(y,t),t),\\
  \eta(y,0)=y.
  \end{array}
  \right.
\end{equation*}
Denote
\begin{equation*}
  \varrho(y,t):=\rho(\eta(y,t),t),\quad v(y,t):=u(\eta(y,t),t), \quad \vartheta(y,t):=\theta(\eta(y,t),t), \label{newunknown}
\end{equation*}
and
\begin{equation*}
 J:= J(y,t)=\eta_y(y,t).
\end{equation*}
Then,
\begin{equation*}
J_t=v_y,\quad J|_{t=0}\equiv1, \quad J\varrho=\varrho_0.\label{LCNSJ}
\end{equation*}
Thus, in the Lagrangian coordinates, the system (\ref{erho}), (\ref{eu}), and (\ref{etheta}) becomes
\begin{eqnarray}
  J_t=v_y,\label{EqJ}\\
  \varrho_0v_t-\mu\left(\frac{v_y}{J}\right)_y+R\left(\frac{\varrho_0}{J}
  \theta\right)_y=0,\label{EqV}\\
  c_v\varrho_0\vartheta_t-\kappa\left(\frac{\vartheta_y}{J}\right)_y
  +R\frac{\varrho_0}{J}\vartheta v_y=\mu\frac{|v_y|^2}{J}.\label{EqTheta}
\end{eqnarray}
The initial date will be taken as
\begin{equation}
  \label{IC}
  (J,v,\vartheta)|_{t=0}=(J_0,v_0,\vartheta_0),
\end{equation}
where $J_0$ has uniform positive lower and upper bounds.

It should be emphasized that here $J$ is deliberately chosen to replace $\varrho$ as one of the unknowns of the system,
which is one of the main technical differences between the current paper and the classic works
\cite{KAZHIKOV82,KAZHIKOV77}. Note that, by the definition of $J$, the initial $J_0$ should
be identically one; however, for the aim of extending a local
solution $(J,v,\vartheta)$ to be a global one, one needs the local well-posedness
of solutions to the system (\ref{EqJ})--(\ref{EqTheta}) with initial $J_0$
not being identically one. 

In the
Lagrangian coordinates, the entropy can be expressed as
\begin{eqnarray}
 s &=&c_v\left(\log\frac RA+(\gamma-1)\log J+\log\vartheta-(\gamma-1)\log\varrho_0\right).\label{Entropy}
\end{eqnarray}
The effective viscous flux $G$, defined as
\begin{equation}
\label{ExG}
G:=\mu\frac{v_y}{J}-R\frac{\varrho_0\vartheta}{J},
\end{equation}
is useful for proving the global existence of solutions, which satisfies
\begin{equation}
  G_t-\frac\mu J\left(\frac{G_y}{\varrho_0}\right)_y=-\frac{\kappa(\gamma-1)}{J}
  \left(\frac{\vartheta_y}{J}\right)_y-\gamma\frac{v_y}{J}G. \label{EqG}
\end{equation}

The following conventions will be used throughout this paper.
For $1\leq q\leq\infty$ and positive integer $m$, $L^q=L^q(\mathbb R)$ and
$W^{1,q}=W^{m,q}(\mathbb R)$ denote the standard Lebesgue and Sobolev spaces,
respectively, and $H^m=W^{m,2}$. For simplicity, $L^q$ and
$H^m$ denote also their $N$ product spaces $(L^q)^N$ and $(H^m)^N$, respectively.
$\|u\|_q$ is the $L^q$ norm of $u$, and $\|(f_1,f_2,\cdots,f_n)\|_X$ is the sum
$\sum_{i=1}^N\|f_i\|_X$ or the equivalent norm $\left(\sum_{i=1}^N\|f_i\|_X^2
\right)^{\frac12}$.

The definition of the solutions being considered in this paper is given as follows:

\begin{definition}\label{Def}
Given a positive time $\mathcal T$ and assume that
\begin{equation}\tag{H0}
  \left\{
    \begin{aligned}
      &0<\varrho_0\in W^{1,\infty}(\mathbb R),\quad\underline J\leq J_0\in L^\infty(\mathbb R),\quad
      \vartheta_0\geq0,\\
      &\sqrt{\varrho_0}v_0, \sqrt{\varrho_0}v_0^2, \sqrt{\varrho_0}\vartheta_0, \sqrt{\varrho_0}J_0',v_0', \varrho_0^{\frac32}
      \vartheta_0'\in L^2(\mathbb R),
    \end{aligned}
  \right.
\end{equation}
where $\underline J$ is a positive constant. A triple $(J, v, \vartheta)$ is called a solution to the
system (\ref{EqJ})--(\ref{EqTheta}), subject to (\ref{IC}), in $\mathbb R\times(0,T)$, if it has the regularities
\begin{eqnarray*}
&0<J\in L^\infty(\mathbb R\times(0,\mathcal T)),\quad 0\leq\vartheta\in L^\infty(\mathbb R\times(0,\mathcal T)), \\
&J_t,\sqrt{\varrho_0}J_y, \sqrt{\varrho_0}v,\sqrt{\varrho_0}v^2, v_y, \sqrt{\varrho_0}\vartheta, \varrho_0^{\frac32}\vartheta_y\in L^\infty(0,\mathcal T; L^2(\mathbb R)), \\
&\sqrt{\varrho_0}J_{yt}, vv_y,\sqrt{\varrho_0}v_t,\sqrt{\varrho_0}v_{yy}, \vartheta_y,\varrho_0(\frac{\vartheta_y}{J})_y,  \varrho_0^2\vartheta_t\in L^2(0,\mathcal T; L^2(\mathbb R)),
\end{eqnarray*}
satisfies (\ref{EqJ})--(\ref{EqTheta}) a.e.\,in $\mathbb R\times(0,\mathcal T)$, and fulfills the initial condition (\ref{IC}).
\end{definition}

\begin{remark}
It can be checked easily that $(J,v,\vartheta)$ in Definition \ref{Def} has the regularities
\begin{eqnarray*}
  J\in C([0,\mathcal T]; H^1((-R,R))),\quad J_t\in L^2(0,\mathcal T; H^1((-R,R))),\\
  v,\vartheta\in C([0,\mathcal T]; H^1((-R,R)))\cap L^2(0,\mathcal T; H^2((-R,R))),\\
  v_t,\vartheta_t\in L^2(0,\mathcal T; L^2((-R,R))),
\end{eqnarray*}
for any $R>0$ and, in particular, $(J, v, \vartheta)|_{t=0}$ is well-defined.
\end{remark}

The main results of this paper are summarized in the following theorems, whose precise statements will be given in the subsequent sections,
and the major ideas of the proofs are sketched here.

First, the following well-posedness results hold.

\begin{theorem}
\label{ThmGloUni-v'}
(i) Assume that (H0) holds.
Then there is a local solution $(J, v, \vartheta)$ to the system (\ref{EqJ})--(\ref{EqTheta}) with initial data (\ref{IC}).

(ii) Under the additional assumption that
\begin{equation}
  \left(\frac{1}{\sqrt{\varrho_0}}\right)'\in L^\infty(\mathbb R), \quad \varrho_0\in L^1(\mathbb R),
  \quad\sqrt{\varrho_0}\vartheta_0'\in L^2(\mathbb R)
  \tag{H1}
\end{equation}
the solution $(J, v, \vartheta)$ established in (i) is unique and exists globally in time.
\end{theorem}

The local existence part of Theorem \ref{ThmGloUni-v'} can be proven in the standard way.
For the global existence, one may try to follow the arguments for the
non-vacuum case in \cite{KAZHIKOV77}. Unfortunately, it does not work directly here. Indeed, one of the key
observations used in \cite{KAZHIKOV77} is the following inequality (see (3.11) there)
\begin{equation}
m_\varrho(t)\geq C\left[1+\int_0^tM_\vartheta(\tau)d\tau\right]^{-1},\label{KSIID}
\end{equation}
where $m_\varrho$ and $M_\vartheta$ are the lower bound of $\varrho$ and upper bound of $\vartheta$, respectively, which is employed
to obtain the $L^\infty(0,T; L^2)$ type a priori estimates (see (4.7) in \cite{KAZHIKOV77}) and consequently the high
order estimates. However, (\ref{KSIID}) fails in the presence of vacuum where $m_\varrho\equiv0$ and
$M_\vartheta$ is finite.

The key step of proving the global existence here is to get the a priori $L^\infty(0,T; L^2)$
estimate of $(\sqrt{\varrho_0} v^2,\sqrt{\varrho_0}\vartheta)$ and upper bound of
$J$. These are achieved by the $L^2$ type energy estimate for $E:=\frac{v^2}{2}+c_v\vartheta$ and
the observation that $J=B (J_0+\frac R\mu\int_0^t\frac{\varrho_0\vartheta}{B} d\tau)$ for some function $B$ having positive lower and upper bounds (see Proposition \ref{PropIdJ}, below), which, in particular, implies
\begin{equation*}
\|J\|_\infty\leq C\left(1+\int_0^t\|\varrho_0\vartheta\|_\infty d\tau\right).
\end{equation*}
It is noted that this inequality holds for both the vacuum and non-vacuum cases, and it reduces to (\ref{KSIID}) for the non-vacuum case.

Now, we turn to the major issue of this paper: the uniform boundedness of the entropy. For the lower bound, we need the following key assumption:
\begin{equation}
  \left(\frac{1}{\varrho_0}\right)''\in L^\infty(\mathbb R). \tag{H2}
\end{equation}

\begin{theorem}
  \label{ThmSlowBd'}
Under the assumptions (H0)--(H2), the entropy of the solution in Theorem \ref{ThmGloUni-v'} is uniformly bounded from below, up to any finite time, as long as it holds initially.
\end{theorem}

Note that the entropy $s$ satisfies
\begin{equation}\label{EQENTROPY}
c_v\varrho_0s_t-\kappa\left(\frac{s_y}{J}\right)_y=\kappa R\left(\frac{\varrho_0'}{J\varrho_0}-\frac{J_y}{J^2}\right)_y+\frac{c_v}{J\vartheta}\left(\mu|v_y|^2
+\frac{\kappa}{\vartheta}|\vartheta_y|^2\right).
\end{equation}
So in the non-heat conductive case, $\kappa=0$, the entropy can only increase in time and thus is bounded from below
trivially, while the upper bound of the entropy is achieved by carrying a certain class of singular type energy estimates
in \cite{LIXIN17}. However, in the general case $\kappa>0$, the term $\kappa
R\big(\tfrac{\varrho_0'}{J\varrho_0}-\tfrac{J_y}{J^2}\big)_y$ may cause both the increasing and decreasing of $s$ and gives
some major technical difficulties to get the uniform bounds on $s$. In particular, though the idea of estimating the
entropy by singularly weighted energy estimates may still be useful here, yet it is not enough to yield the uniform bounds
for the entropy. Some additional ideas are needed for the heat conductive case. Indeed, here are some new key
observations:

For the uniform lower bound of $s$, it suffices to estimate a new quantity $S:=\log\vartheta-(\gamma-1)\log\varrho_0$, which can be shown to satisfy
\begin{equation}
  \label{EQS}
  c_v\varrho_0S_t-\kappa\left(\frac{S_y}{J}\right)_y=F_{gd}+F_{ok}+F_{bd},
\end{equation}
where $F_{gd}=\varrho_0f_{gd}$ and $F_{ok}=\varrho_0f_{ok}$ for some
$f_{gd}\in L^\infty(0,T; L^2)$ and $f_{ok}\in L^\infty(0,T;L^\infty)$,
while $F_{bd}$ is given by
\begin{equation}
  \label{FBD}
F_{bd}=\frac{\mu}{J\vartheta}\Big(v_y-\frac{R}{2\mu}\varrho_0\vartheta\Big)^2+\kappa
\frac{|\vartheta_y|^2}{J\vartheta^2}.
\end{equation}
The uniform lower bound of $S$ is achieved by applying some modified De Giorgi type iterations to (\ref{EQS}).
Note that $F_{bd}$ is nonnegative and thus causes no difficulty in proving the uniform lower bound of $S$.
The contributions due to the source term $F_{ok}=\varrho_0f_{ok}$ are dealt with by introducing
an auxiliary function $\tilde S:=S+Mt$, with a sufficient large $M$, which satisfies a similar equation as $S$, but with the term corresponding to
$F_{ok}$ having desired sign. To deal with the source term $F_{gd}$, one notes that $\frac{F}{\varrho_0}\in L^\infty(0,T; L^2)$ is sufficient to get the lower bound of the solution to the model equation
$\varrho_0V_t-V_{yy}=F$, by applying a modified De Giorgi type iteration. Thus, since $\frac{F_{gd}}{\varrho_0}\in
L^\infty(0,T; L^2)$, the contributions due to the term $F_{gd}$ can also be handled.

Technically, due to the degeneracy of equation (\ref{EQS}), different from the classic De Giorgi iteration
for uniform parabolic equations, the testing function used in our iteration
is $\frac{(S-\ell)_-}{\varrho_0}$ instead of $(S-\ell)_-$.
In other words, our energy estimates needed in the
De Giorgi iteration should be of singular type, to which our idea of singular energy estimates
in \cite{LIXIN17} will be useful here.
Moreover, due to the unboundedness of the domain and the lack of integrability of $S$, some suitable cut-off and delicate approximations will be used to justify rigorously the arguments, see Proposition \ref{PropEstL2S} in Section \ref{SecLowEty}.

For the upper bound of the entropy, we need also the following compatibility condition:
\begin{equation}
  \varrho_0^{\frac{1-\gamma}{2}}v_0, \varrho_0^{1-\frac\gamma2}\vartheta_0, \varrho_0^{-\frac\gamma2}G_0\in L^2(\mathbb R).\tag{HS}
\end{equation}

\begin{theorem}
  \label{ThmSuppBd'}
Under the conditions (H0)--(H2) and (HS), the entropy of the unique solution in Theorem \ref{ThmGloUni-v'} is uniformly bounded from above, up to any finite time, as long as it holds initially.
\end{theorem}

As $J$ is uniformly positive, a necessary and sufficient condition for the uniform boundedness of the entropy is that $\vartheta$ tends to zero at the same rate as $\varrho_0^{\gamma-1}$ at the far field, which unfortunately is not guaranteed by the solution spaces used in \cite{LIONS98,FEIREISL01,JIAZHA03,FEIREISL04B,SALSTR93,CHOKIM04,CHOKIM06-1,
CHOKIM06-2,HLX12,WENZHU17,LWX}. Indeed, the solutions established in these papers have the $L^2$ integrability of $\sqrt{\varrho_0}\vartheta$, but not of $\vartheta$ itself, which allows $\vartheta$ not to decay to zero or even to
grow to infinity at the far field.

%
Due to the singular term $\frac{c_v}{J\vartheta}\left(\mu|v_y|^2
+\frac{\kappa}{\vartheta}|\vartheta_y|^2\right)$ in (\ref{EQENTROPY}), performing the same type of De Giorgi iteration to (\ref{EQENTROPY}) as before will not lead to the desired upper bound for the entropy.
In fact, for this case, instead of working on the entropy equation ((\ref{EQENTROPY})or (\ref{EQS})) directly,
we will apply a modified De Giorgi iteration to the temperature equation, with some elaborate singular type
energy estimates. The main steps can be sketched as follows. Note that the entropy has uniform upper bound iff
$$
\vartheta_\ell:=\vartheta-\ell\varrho_0^{\gamma-1}e^{Mt}\leq0,\quad\text{or equivalently}\quad (\vartheta_\ell)_+=0,
$$
for some positive numbers $\ell$ and $M$. $\vartheta_\ell$ satisfies
$$
c_v\varrho_0\partial_t\vartheta_\ell-\kappa\partial_y\left(\frac{\partial_y\vartheta_\ell}{J}\right) =v_yG+``\text{other terms}".
$$
Testing the above equation with $\varrho_0^{1-2\gamma}(\vartheta_\ell)_+$ yields
\begin{eqnarray*}
&&\frac{c_v}{2}\frac{d}{dt}\|\varrho_0^{1-\gamma}(\vartheta_\ell)_+\|_2^2+\kappa\|\varrho_0^{\frac12-\gamma}
\partial_y(\vartheta_\ell)_+\|_2^2 \\
&\leq &C\int_\mathbb R(|\varrho_0^{-\frac\gamma2}G|^2+|\varrho_0^{1-\frac\gamma2}\vartheta|^2)\varrho_0^{1-\gamma}\vartheta_\ell dy+``\text{other terms}",
\end{eqnarray*}
see Proposition \ref{PropWL2EstThetaell}, below.
It should be noted here that the choice of
the singularly weighted test function $\varrho_0^{1-2\gamma}(\vartheta_\ell)_+$ is crucial.
The above inequality indicates the necessity of carrying out the energy estimates
for $\varrho_0^{1-\frac\gamma2}\vartheta$ and $\varrho_0^{-\frac\gamma2}G$; these estimates, thanks to
the assumption (H1), can be achieved by testing
(\ref{EqV}), (\ref{EqTheta}), and (\ref{EqG}) with $\varrho_0^{-\gamma}v$, $\varrho_0^{1-\gamma}\vartheta$,
and $J\varrho_0^{-\gamma}G$, respectively, see Propositions \ref{PropWestvtheta} and \ref{PropWestG}, below.
With these estimates in hand, one can proceed the iteration to get finally $(\vartheta_\ell)_+\equiv0$ for some positive
$\ell$, which yields the desired upper bound of the entropy.

Some remarks are in order.

\begin{remark}
(i) Conditions $(\frac{1}{\sqrt{\varrho_0}})',(\frac{1}{\varrho_0})''\in L^\infty(\mathbb R)$ in (H1)--(H2) are essentially slow decay assumptions on $\varrho_0$ at the far field. In fact, for
$\varrho_0(y)=\frac{K_\varrho}{\langle y\rangle^{\ell_\varrho}},$ with $\langle y\rangle=(1+y^2)^{\frac12}$ and positive constants $K_\varrho$ and $\ell_\varrho$, it holds that
\begin{eqnarray*}
  \left(\frac{1}{\sqrt{\varrho_0}}\right)'\in L^\infty\Leftrightarrow0\leq\ell_\varrho\leq2\quad\text{and}\quad
  \left(\frac{1}{\varrho_0}\right)''\in L^\infty\Leftrightarrow0\leq\ell_\varrho\leq2.
\end{eqnarray*}

(ii) All results in the above theorems still hold true if replacing the assumptions $(\frac{1}{\sqrt{\varrho_0}})',(\frac{1}{\varrho_0})''\in L^\infty(\mathbb R)$ in (H1) and (H2) by the following weaker one:
$$
\frac{{\underline K}_\varrho}{\langle y\rangle^{\bar{\ell}_\varrho}}\leq\varrho_0(y)\leq\frac{{\bar K}_\varrho}{\langle y\rangle^{\underline{\ell}_\varrho}}, \quad\forall y\in\mathbb R,
$$
for some constants $0<\underline K_\varrho\leq\bar K_\varrho$ and $0\leq\underline\ell_\varrho\leq\bar\ell_\varrho\leq2$.
\end{remark}

\begin{remark}
\label{RemarkInitialSet}
Let $K_\varrho$ and $\frac1\gamma<\ell_\varrho\leq2$ be positive constants. Choose
$$
\varrho_0(y)=\frac{K_\varrho}{\langle y\rangle^{\ell_\varrho}},\quad J_0\equiv1,\quad v_0\in C_c^\infty(\mathbb R), \quad s_0\in W^{1,\infty}(\mathbb R), \quad \vartheta_0=\frac ARe^{\frac{s_0}{c_v}}\varrho_0^{\gamma-1}.
$$
Then, one can verify easily that (H0)--(H2) and (HS) hold.
Therefore, the set of the initial data that fulfills the conditions in the above theorems is not empty.
\end{remark}

\begin{remark}
Both the assumptions that there is no interior vacuum and that the initial density decays slowly at the far field are necessary conditions for guaranteeing the uniform boundedness of the entropy. In fact, if either there is an interior
point vacuum or the density decays to vacuum sufficient fast at the far field, then the entropy will become unbounded immediately after the initial time, see Li-Xin \cite{LIXININPREPARE}.
\end{remark}

\begin{remark}
It should be emphasized that though we deal with only the one dimensional case here, the main ideas of combining
singularly weighted energy estimates with some deliberately modified De Giorgi iterations can be used to derive the
uniform boundedness of the entropy for the multi-dimensional case at least locally in time. Indeed, by adapting
these ideas
with some more involved and complicated calculations, one can obtain that the boundedness of the entropy can be propagated
by the multi-dimensional compressible Navier-Stokes system up to the maximal existing time of the strong solution under
similar conditions on the initial density. However,
the global in time existence of strong solutions for general initial data is still unknown.
\end{remark}

%

The rest of this paper is arranged as follows: in Section \ref{Secloc-nv}, we
consider the system with the initial density
uniformly away from zero, prove the local existence of
solutions, and carry out some a priori estimates independent of the positive lower bound
of the initial density; Section \ref{SecLoc-v} is devoted to the proof of the local existence of solutions in the presence of far field vacuum; while the global existence and uniqueness of solutions are shown in Section
\ref{SecGlo-v}; and finally in Section \ref{SecLowEty} and Section \ref{SecUppEty}, we establish the
uniform lower and upper bounds of the entropy, respectively, by performing the singular type energy estimates and using some suitably modified De Giorgi type iterations.

Throughout this paper, $C$ will denote a genetic positive constant, which may vary from line to
line. For simplicity of presentations, the quantities, on which the constant $C$ depends, will be emphasized only in the statements, but not in the proofs, of the theorems, propositions, and corollaries.

\section{Local existence and a priori estimates in the absence of vacuum}
\label{Secloc-nv}
Let $\underline\varrho, \bar\varrho, \underline J,$ and $\bar J$ be positive constants. Assume that
\begin{equation}\label{ass-nv}
  \left\{
    \begin{aligned}
      &0<\underline\varrho\leq\varrho_0(y)\leq\bar\varrho<\infty,
      \quad 0<\underline J\leq J_0(y)\leq\bar J<\infty,\quad\forall y\in\mathbb R, \qquad \\
      &\varrho_0'\in L^\infty(\mathbb
      R),\quad J_0'\in L^2(\mathbb R),\quad v_0\in H^1(\mathbb R),\quad 0\leq\vartheta_0\in H^1(\mathbb R).
    \end{aligned}
  \right.
\end{equation}

The following local existence result holds.

\begin{proposition}
\label{PropLocNovacuum}
Under the assumption (\ref{ass-nv}), there is a positive time $T_0$ depending only on $\underline\varrho,
\bar\varrho, \underline J, \bar J, \|\varrho_0'\|_\infty, \|J_0'\|_2, \|v_0\|_{H^1}$, and $\|\theta_0\|_{H^1}$, such that the problem (\ref{EqJ})--(\ref{IC}) with the following far field condition
\begin{equation}
  \label{BC}
(v,\vartheta)\rightarrow0,\quad\mbox{as }y\rightarrow\infty,
\end{equation}
has a unique
solution $(J, v, \theta)$, on $\mathbb R\times(0,T_0)$, satisfying
\begin{eqnarray*}
  &&\frac{\underline J}{2}\leq J\leq2\bar J,\quad\mbox{on }\mathbb R\times[0,T_0],\quad J-J_0\in C([0,T_0]; H^1), \\
  &&v\in C([0,T_0]; H^1)\cap L^2(0,T_0; H^2),\quad 0\leq\vartheta\in C([0,T_0]; H^1)\cap L^2(0,T_0; H^2),\\
  &&J_t\in L^\infty(0,T_0; L^2)\cap L^2(0,T_0; H^1),\quad v_t\in L^2(0,T_0; L^2),\quad \vartheta_t\in L^2(0,T_0; L^2).
\end{eqnarray*}
\end{proposition}

\begin{proof}
This can be proved in the standard way by using the fixed point argument based on the following
linearized system
\begin{eqnarray}
  &&J_t=V_y, \label{L1}\\
  &&\varrho_0v_t-\mu\left(\frac{v_y}{J}\right)_y=-R\left(\frac{\varrho_0}{J}
  \Theta\right)_y, \label{L2}\\
  &&c_v\varrho_0\vartheta_t-\kappa\left(\frac{\vartheta_y}{J}\right)_y
  =\mu\frac{V_y^2}{J}-R\frac{\varrho_0}{J}\Theta V_y, \label{L3}
\end{eqnarray}
subject to (\ref{IC}) and (\ref{BC}), for given $(V,\Theta)$. Indeed, the classic theory for uniformly parabolic equations yields a
unique global solution $(v,\vartheta)$ to the system (\ref{L2})--(\ref{L3}), subject to (\ref{IC}) and (\ref{BC}). Thus, one
can define a solution mapping $(V,\Theta)\rightarrow(v,\vartheta)$. Then, by carrying out the energy estimates, similar to (actually easier than)
those we will derive in the rest of this section, one can see that this solution mapping
fulfills all the conditions of the Banach's contracting fixed point theorem, and thus has a unique fixed point in the corresponding Banach space, which yields the unique solution to the system (\ref{EqJ})--(\ref{EqTheta}), subject to (\ref{IC}) and (\ref{BC}).
\end{proof}

By applying Proposition \ref{PropLocNovacuum} iteratively, one can extend the local solution $(J, v, \vartheta)$ uniquely to the maximal time $T_{\text{max}}$ of existence, which is characterized as
\begin{equation}
  \label{Tm}
  \limsup_{T\rightarrow T_{\text{max}}^-}
  \left((\inf_{y\in\mathbb R}J)^{-1}+\sup_{y\in\mathbb R}J+\|J_y\|_2+\|v\|_{H^1}+\|\vartheta\|_{H^1}\right)=\infty.
\end{equation}
In the rest of this section, it is always assumed that the unique solution $(\rho, v,\theta)$ has already been extended uniquely to the
maximal time of existence $T_{\text{max}}$.

One aim of this section is to show $T_{\text{max}}$ is independent of $\underline\varrho$. To this end, we set
\begin{equation}
  \label{T_*}
  T_*:=\max\left\{T\in(0,T_{\text{max}})~\Big|~\frac{\underline J}{3}\leq J\leq 3\bar J \mbox{ on }\mathbb R\times[0,T]\right\}.
\end{equation}
In the rest of this section, we will focus on the solutions in the time interval $(0,T_*)$, so that $J$ has the positive lower and upper bounds stated in (\ref{T_*}).

\subsection{A priori $L^2$ estimates}

\begin{proposition}
\label{PropEstL2-NV}
There is a positive time $T_{\text{ode}}$ depending only on
$c_v, R, \mu, \kappa, \|\varrho_0\|_{W^{1,\infty}}$, $\underline J,$ and $\bar J$, such that
\begin{eqnarray*}
  \sup_{0\leq t\leq T_{\text{ode}}^*}\|(\sqrt{\varrho_0}v,\sqrt{\varrho_0}E)\|_2^2
  +\int_0^{T_{\text{ode}}^*}(\|\varrho_0\vartheta\|_\infty^2+
  \left\|\left(v_y,vv_y,\vartheta_y\right)\right\|_2^2)dt \leq \mathcal E_0,
\end{eqnarray*}
where $E=\frac{v^2}{2}+c_v\vartheta$, $T_{\text{ode}}^*:=\min\{T_*, T_{\text{ode}},1\}$, and $\mathcal E_0$ is a positive constant depending only on $c_v, R, \mu, \kappa, \underline J, \bar J, \|\varrho_0\|_{W^{1,\infty}}, \|\sqrt{\varrho_0}v_0\|_2,$ and $\|\sqrt{\varrho_0} E_0\|_2$.
\end{proposition}

\begin{proof}
It follows from (\ref{EqV}) and the Cauchy inequality that
\begin{equation}
  \label{2.2}
  \frac{d}{dt}\|\sqrt{\varrho_0}v\|_2^2+\mu\left\|\frac{v_y}{\sqrt J}
  \right\|_2^2\leq \frac{R^2}{\mu}
  \left\|\frac{\varrho_0}{\sqrt J}\vartheta\right\|_2^2.
\end{equation}
Set $E=\frac{v^2}{2}+c_v\vartheta$. Then,
\begin{equation}
  \label{EqE}
  \varrho_0E_t-\kappa\left(\frac{\vartheta_y}{J}\right)_y=\left(\left(\mu\frac{v_y}{J}
  -R\frac{\varrho_0}{J}\vartheta\right)v\right)_y.
\end{equation}
Testing (\ref{EqE}) with $E$ yields
\begin{equation*}
  \frac12\frac{d}{dt}\|\sqrt{\varrho_0}E\|_2^2+\kappa\int_{\mathbb R}
  \frac{\vartheta_y}{J}E_ydy =-\int_{\mathbb R}\left(\mu\frac{v_y}{J}-R
  \frac{\varrho_0}{J}\vartheta\right) vE_ydy. \label{2.3-1}
\end{equation*}
Direct estimates show that
\begin{eqnarray*}
  \int_{\mathbb R}
  \frac{\vartheta_y}{J}E_ydy&\geq&
  \frac{3c_v}{4}\left\|\frac{\vartheta_y}{\sqrt J}\right\|_2^2-\frac{1}{c_v}
  \left\|\frac{vv_y}{\sqrt J}\right\|_2^2,\\
  -\int_{\mathbb R}\left(\mu\frac{v_y}{J}-R
  \frac{\varrho_0}{J}\vartheta\right) vE_ydy
  &\leq&\frac{c_v\kappa}{4}\left\|\frac{\vartheta_y}{\sqrt{J}}\right\|_2^2
  +C\left(\left\|\frac{vv_y}{\sqrt J}\right\|_2^2+\int_{\mathbb R} \frac{\varrho_0^2}{J}\vartheta^2v^2dy\right),
\end{eqnarray*}
and, consequently,
\begin{equation}
  \label{2.3}
  \frac{d}{dt}\|\sqrt{\varrho_0}E\|_2^2+\kappa c_v\left\|\frac{\vartheta_y}{\sqrt J}\right\|_2^2
  \leq C\left(\left\|\frac{vv_y}{\sqrt J}\right\|_2^2+\int_{\mathbb R} \frac{\varrho_0^2}{J}\vartheta^2v^2dy\right).
\end{equation}
Test (\ref{EqV}) with $v^3$ and apply the Cauchy-Schwaz inequality to get
\begin{equation}
  \label{2.4}
  \frac{d}{dt}\|\sqrt{\varrho_0}v^2\|_2^2+8\mu\left\|\frac{vv_y}{\sqrt J}\right\|_2^2\leq\frac{9R^2}{\mu}\int_{\mathbb R}\frac{\varrho_0^2\vartheta^2v^2}{J}dy.
\end{equation}
By (\ref{2.2}), (\ref{2.3}), and (\ref{2.4}), one can choose $A_1$ sufficiently large such that
\begin{align}
\frac{d}{dt}(\|\sqrt{\varrho_0}v\|_2^2+\|\sqrt{\varrho_0}E\|_2^2 +A_1\|\sqrt{\varrho_0}v^2\|_2^2)
  &+\mu\left\|\frac{v_y}{\sqrt J}\right\|_2^2\nonumber\\
  +\kappa c_v\left\|\frac{\vartheta_y}{\sqrt J}\right\|_2^2+A_1\mu\left\|\frac{vv_y}{\sqrt J}\right\|_2^2\leq&~~~~ C\int_{\mathbb R}\left(
  \frac{\varrho_0^2\vartheta^2}{J}+\frac{\varrho_0^2\vartheta^2v^2}{J}\right) dy.\label{2.5}
\end{align}

Due to the definition of $T_*$, one has
\begin{eqnarray*}
  \int_{\mathbb R}\left(
  \frac{\varrho_0^2\vartheta^2}{J}+\frac{\varrho_0^2\vartheta^2v^2}{J}\right) dy
  &\leq&C(1+\|\varrho_0\vartheta\|_\infty)\|\sqrt{\varrho_0}E\|_2^2.
\end{eqnarray*}
Note that
\begin{eqnarray}
\|\varrho_0\vartheta\|_\infty^2&\leq&
2\int_\mathbb R(\varrho_0|\varrho_0'|\vartheta^2+\varrho_0^2\vartheta|\vartheta_y|)dy
\nonumber\\
&\leq&C\left(\|\sqrt{\varrho_0}E\|_2^2+\|\sqrt{\varrho_0}E\|_2
\left\|\frac{\vartheta_y}{\sqrt J}\right\|_2\right). \label{Re1}
\end{eqnarray}
Thus,
\begin{eqnarray*}
  \int_{\mathbb R}\left(
  \frac{\varrho_0^2\vartheta^2}{J}+\frac{\varrho_0^2\vartheta^2v^2}{J}\right) dy&\leq&C\left(1+\|\sqrt{\varrho_0}E\|_2+\|\sqrt{\varrho_0}E\|_2^{\frac12}
\left\|\frac{\vartheta_y}{\sqrt J}\right\|_2^{\frac12}\right)
\|\sqrt{\varrho_0}E\|_2^2\\
&\leq&\varepsilon\left\|\frac{\vartheta_y}{\sqrt J}\right\|_2^2+C_\varepsilon
(1+\|\sqrt{\varrho_0}E\|_2^2)^2,
\end{eqnarray*}
for any $t\in[0,T_*]$, and for any $\varepsilon>0$.

Choosing $\varepsilon$ sufficient small, one obtains from this and (\ref{2.5}) that
\begin{align*}
\frac{d}{dt}(\|\sqrt{\varrho_0}v\|_2^2&+\|\sqrt{\varrho_0}E\|_2^2 +A_1\|\sqrt{\varrho_0}v^2\|_2^2)+\mu\left\|\frac{v_y}{\sqrt J}\right\|_2^2\nonumber\\
  & +\frac{\kappa c_v}{2}\left\|\frac{\vartheta_y}{\sqrt J}\right\|_2^2+A_1\mu\left\|\frac{vv_y}{\sqrt J}\right\|_2^2
  \leq C
(1+\|\sqrt{\varrho_0}E\|_2^2)^2,
\end{align*}
for any $t\in[0,T_*]$.
This and solving an ordinary differential
inequality of the type $f'\leq Cf^2$ yield that there is a positive time $T_{\text{ode}}$
such that
\begin{eqnarray*}
  \sup_{0\leq t\leq T_{\text{ode}}^*}\|(\sqrt{\varrho_0}v,\sqrt{\varrho_0}E)\|_2^2
  +\int_0^{T_{\text{ode}}^*}\left\|\left(v_y,vv_y,\vartheta_y\right) \right\|_2^2dt \leq \mathcal E_0',
\end{eqnarray*}
where $T_{\text{ode}}^*:=\min\{T_*, T_{\text{ode}},1\}$. Then, it follows from (\ref{Re1}) that
$\int_0^{T_{\text{ode}}^*}\|\varrho_0\vartheta\|_\infty^2 dt\leq \mathcal E_0''$.
This proves the conclusion.
\end{proof}

\subsection{A priori $H^1$ estimates}
\begin{proposition}
\label{PropEstH1-NV}
  Let $T_{\text{ode}}^*$ be as in Proposition \ref{PropEstL2-NV} and $G$ be given by (\ref{ExG}). Set $G_0=\frac{1}{J_0}(\mu v_0'-R\varrho_0\vartheta_0)$. Then, there is a positive constant $\mathcal E_1$,
  depending only on $\mu, \kappa, c_v, R, \bar\varrho, \underline J, \bar J, \|\varrho_0\|_{W^{1,\infty}}, \|\sqrt{\varrho_0}v_0\|_2, \|\sqrt{\varrho_0}E_0\|_2$, $\|G_0\|_2$, and $\Big\|
  \varrho_0^{\tfrac32}\vartheta_0'\Big\|_2^2$, such that
  \begin{eqnarray*}
\sup_{0\leq t\leq T_{\text{ode}}^*}\|G\|_2^2+\int_0^{T_\text{ode}^*}
    \Bigg(\Bigg\|\frac{G_y}{\sqrt{\varrho_0}}\Bigg\|_2^2
    +\|G\|_\infty^2\Bigg)dt &\leq&\mathcal E_1,\\
\sup_{0\leq t\leq T_{\text{ode}}^*}\Big\|\varrho_0^\frac32
    \vartheta_y\Big\|_2^2+\int_0^{T_\text{ode}^*}
    \Bigg(\|\varrho_0^2\vartheta_t\|_2^2+
    \Bigg\|\varrho_0\bigg(\frac{\vartheta_y}
    {J}\bigg)_y\Bigg\|_2^2+\|v_y\|_\infty^2\Bigg)dt& \leq&\mathcal E_1.
  \end{eqnarray*}
\end{proposition}

\begin{proof}
We start with the estimate on $G$. Testing (\ref{EqG}) with $JG$ yields
\begin{eqnarray*}
  \frac12\frac{d}{dt}\|\sqrt JG\|_2^2+\mu\left\|\frac{G_y}{\sqrt{\varrho_0}}\right\|_2^2
  &=&(2\gamma-1)\int_\mathbb RvGG_ydy-\kappa(\gamma-1)\int_\mathbb R
  \frac{\vartheta_yG_y}{J}dy\\
  &\leq&\frac\mu4\left\|\frac{G_y}{\sqrt{\varrho_0}}\right\|_2^2+
  C\left(\|\sqrt{\varrho_0}v\|_2^2\|G\|_\infty^2+\|\vartheta_y\|_2^2\right).
\end{eqnarray*}
It follows from Proposition \ref{PropEstL2-NV} and the inequality above that
\begin{equation}\label{2.7-1}
  \frac{d}{dt}\|\sqrt JG\|_2^2+\frac{3\mu}{2}\left\|\frac{G_y}{\sqrt{\varrho_0}}\right\|_2^2
  \leq C\left(\|G\|_\infty^2+\|\vartheta_y\|_2^2\right).
\end{equation}
Note that
\begin{equation}
  \|G\|_\infty^2 \leq \int_\mathbb R|\partial_yG^2|dy\leq C\|G\|_2\left\|\frac{G_y}{\sqrt{\varrho_0}}\right\|_2
   \leq \varepsilon\left\|\frac{G_y}{\sqrt{\varrho_0}}\right\|_2^2+
  C_\varepsilon\|\sqrt JG\|_2^2,\label{Re2}
\end{equation}
for any positive $\varepsilon$. Choosing $\varepsilon$ sufficient small, one gets from (\ref{2.7-1}) and (\ref{Re2}) that
\begin{equation*}
  \frac{d}{dt}\|\sqrt JG\|_2^2 +\mu\left\|\frac{G_y}{\sqrt{\varrho_0}}\right\|_2^2
  \leq C\left(\|\sqrt JG\|_2^2+\|\vartheta_y\|_2^2\right).
\end{equation*}
Consequently, Proposition \ref{PropEstL2-NV} and the Gronwall inequality show that
\begin{equation}
  \label{2.8}
  \sup_{0\leq t\leq T_{\text{ode}}^*}\|G\|_2^2+\int_0^{T_{\text{ode}}^*}
  \left(\left\|\frac{G_y}{\sqrt{\varrho_0}}\right\|_2^2+\|G\|_\infty^2\right)dt\leq \mathcal E_1'.
\end{equation}

Next we estimate $\vartheta$. It follows from (\ref{EqTheta}) that
\begin{equation}
  \int_\mathbb R\varrho_0^2
  \left(c_v\varrho_0\vartheta_t-\kappa\left(\frac{\vartheta_y}{J}\right)_y
  \right)^2
  =\int_\mathbb R\varrho_0^2v_y^2 G^2dy. \label{2.9-1}
\end{equation}
Using (\ref{EqJ}),
one deduces
\begin{eqnarray*}
  -2\int_\mathbb R\varrho_0^3\vartheta_t \left(\frac{\vartheta_y}{J}\right)_y dy
   = \frac{d}{dt}\left\|\sqrt{\frac{\varrho_0^3}{J}} \vartheta_y\right\|_2^2
  +\int_\mathbb R\left(v_y\frac{\varrho_0^3}{J^2}\vartheta_y^2
  +6\varrho_0^2\varrho_0'\frac{\vartheta_y\vartheta_t}{J}\right)dy\\
   \geq\ \  \frac{d}{dt}\left\|\sqrt{\frac{\varrho_0^3}{J}}\vartheta_y \right\|_2^2-\frac{c_v}{2\kappa}\|\varrho_0^2\vartheta_t\|_2^2-\left\|\frac{v_y}{J}\right\|_\infty \left\|\sqrt{\frac{\varrho_0^3}{J}}\vartheta_y\right\|_2^2
-\frac{18\kappa}{c_v}  \left\|\frac{\varrho_0'\vartheta_y}{J}\right\|_2^2.
\end{eqnarray*}
This, together with (\ref{2.8}) and (\ref{2.9-1}), yields
\begin{eqnarray}
  &&c_v\kappa\frac{d}{dt}\left\|\sqrt{\frac{\varrho_0^3}{J}}\vartheta_y \right\|_2^2+\frac{c_v^2}{2}\|\varrho_0^2\vartheta_t\|_2^2+
  \kappa^2\left\|\varrho_0\left(\frac{\vartheta_y}{J}\right)_y \right\|_2^2 \nonumber\\
  &\leq& c_v\kappa \left\|\frac{v_y}{J}\right\|_\infty \left\|\sqrt{\frac{\varrho_0^3}{J}}\vartheta_y\right\|_2^2+
  18\kappa^2
  \left\|\frac{\varrho_0'\vartheta_y}{J}\right\|_2^2+\int_\mathbb R\varrho_0^2v_y^2 G^2dy\nonumber\\
  &\leq& C\left[(1+\|v_y\|_\infty^2)\left(\left\|\sqrt{\frac{\varrho_0^3}{J}}
  \vartheta_y\right\|_2^2+\mathcal E_1'\right)
  +\|\vartheta_y\|_2^2\right]. \label{2.9-2}
\end{eqnarray}
Since $v_y=\frac1\mu(JG+R\varrho_0\vartheta)$, by Proposition \ref{PropEstL2-NV} and (\ref{2.8}),
one has
\begin{eqnarray*}
  \int_0^{T_{\text{ode}}^*}\|v_y\|_\infty^2 dt&\leq&C\int_0^{T_{\text{ode}}^*}(\|G\|_\infty^2 +\|\varrho_0\vartheta\|_\infty^2)dt\leq C(\mathcal E_0+\mathcal E_1').
\end{eqnarray*}
This, together with (\ref{2.9-2}) and Proposition \ref{PropEstL2-NV}, yields
\begin{eqnarray*}
  \sup_{0\leq t\leq T_{\text{ode}}^*}\left\|\sqrt{\varrho_0^3}\vartheta_y\right\|_2^2
  +\int_0^{T_{\text{ode}}^*}\left(\|\varrho_0^2\vartheta_t\|_2^2+
  \left\|\varrho_0\left(\frac{\vartheta_y}{J}\right)_y\right\|_2^2
  \right)dt\\
  \leq\ \  e^{C(1+\mathcal E_0+\mathcal E_1')}\left(\left\|
  \sqrt{\frac{\varrho_0^3}{J_0}}\vartheta_0'\right\|_2^2+\mathcal E_1'
  +\mathcal E_0\right)=:\mathcal E_1''.
\end{eqnarray*}
The conclusion follows by setting $\mathcal E_1=\max\{\mathcal E_1', \mathcal E_1''\}$.
\end{proof}

\begin{proposition}
  \label{PropEstH1'-NV}
  Let $T_{\text{ode}}^*$ be as in Proposition \ref{PropEstL2-NV}. Then, there is a positive constant $\mathcal E_2$
  depending only on $\mu, \kappa, c_v, R, \bar\varrho, \underline J, \bar J, \|\varrho_0\|_{W^{1,\infty}}, \|\sqrt{\varrho_0}v_0\|_2, \|\sqrt{\varrho_0}E_0\|_2$, $\|G_0\|_2$, $\left\|
  \varrho_0^\frac32\vartheta_0'\right\|_2^2$, and $\|\sqrt{\varrho_0}J_0'\|_2^2$, such that
  \begin{eqnarray*}
    \sup_{0\leq t\leq T_{\text{ode}}^*}
    \|v_y\|_2^2
    +\int_0^{T_{\text{ode}}^*}
    (\|\sqrt{\varrho_0}v_t\|_2^2+\|\sqrt{\varrho_0}v_{yy}\|_2^2) dt\leq
    \mathcal E_2,\\
    \sup_{0\leq t\leq T_{\text{ode}}^*}
    (\|\sqrt{\varrho_0}J_y\|_2^2+\|J_t\|_2^2)
    +\int_0^{T_{\text{ode}}^*}
    \|\sqrt{\varrho_0}J_{yt}\|_2^2 dt\leq
    \mathcal E_2.
  \end{eqnarray*}
\end{proposition}

\begin{proof}
Since $J_t=v_y$ and $J_{yt}=v_{yy}$, it suffices to prove the first conclusion and the estimate $\sup_{0\leq t\leq
T_{\text{ode}}^*}\|\sqrt{\varrho_0}J_y\|_2^2$. Besides, since $\sqrt{\varrho_0}v_t=\frac{G_y}{\sqrt{\varrho_0}}$,
the estimate for $\int_0^{T_{\text{ode}}^*}\|\sqrt{\varrho_0}v_t\|_2^2dt$ follows from Proposition
\ref{PropEstH1-NV} directly.

Note that
\begin{equation*}
  \sqrt{\varrho_0}J_{yt}=\sqrt{\varrho_0}v_{yy}
=\frac{\sqrt{\varrho_0}}{\mu}(JG_y+J_yG+R\varrho_0'
\vartheta+R\varrho_0\vartheta_y).
\end{equation*}
Multiplying the equation before by $\sqrt{\varrho_0}J_y$ yields
\begin{eqnarray*}
  \frac12\frac{d}{dt}\|\sqrt{\varrho_0}J_y\|_2^2&\leq&\frac{\|G\|_\infty}
  {\mu}\|\sqrt{\varrho_0}J_y\|_2^2+\frac{\|\sqrt{\varrho_0}J_y\|_2}
  {\mu}\|J\sqrt{\varrho_0}G_y +R\varrho_0'\sqrt{\varrho_0}\vartheta+R\varrho_0^{\frac32}
  \vartheta_y\|_2\\
  &\leq&C(\|G\|_\infty^2+1)\|\sqrt{\varrho_0}J_y\|_2^2+C(\|G_y\|_2^2+
  \|\sqrt{\varrho_0}\vartheta\|_2^2+\|\vartheta_y\|_2^2),
\end{eqnarray*}
which, together with Propositions \ref{PropEstL2-NV}--\ref{PropEstH1-NV}, yields
\begin{eqnarray}
  \sup_{0\leq t\leq T_{\text{ode}}^*}\|\sqrt{\varrho_0}J_y\|_2^2
  &\leq& e^{C\int_0^{T_{\text{ode}}^*}(1+\|G\|_\infty^2)dt}
  \left[\|\sqrt{\varrho_0}J_0'\|_2^2+C\int_0^{T_{\text{ode}}^*}\|(G_y,\sqrt{\varrho_0}\vartheta,\vartheta_y)\|_2^2dt\right]\nonumber\\
  &\leq&Ce^{C\mathcal E_1}(\|\sqrt{\varrho_0}J_0'\|_2^2+C\mathcal E_0+C\mathcal E_1)
  =:\mathcal E_2'.\label{2.12}
\end{eqnarray}

It follows from direct calculations, (\ref{2.12}), and
Propositions \ref{PropEstL2-NV}--\ref{PropEstH1-NV} that
\begin{eqnarray*}
  &&\sup_{0\leq t\leq T_{\text{ode}}^*}\|v_y\|_2^2+\int_0^{T_{\text{ode}}^*} \|\sqrt{\varrho_0}
  v_{yy}\|_2^2 dt\\
  &=&\sup_{0\leq t\leq T_{\text{ode}}^*}\left\|\frac1\mu
  (JG+R\varrho_0\vartheta)\right\|_2^2 +\int_0^{T_{\text{ode}}^*}
  \left\|\frac{\sqrt{\varrho_0}}{\mu}(JG_y+J_yG+R\varrho_0'\vartheta
  +R\varrho_0\vartheta_y)\right\|_2^2dt \\
  &\leq&C\sup_{0\leq t\leq T_{\text{ode}}^*}\|(G,\sqrt{\varrho_0}\vartheta)\|_2^2 +C\int_0^{T_{\text{ode}}^*}(\|(G_y,
  \sqrt{\varrho_0}\vartheta,\vartheta_y)\|_2^2+\|\sqrt{\varrho_0}J_y\|_2^2
  \|G\|_\infty^2) dt\\
  &\leq& C(\mathcal E_0+\mathcal E_1+\mathcal E_1\mathcal E_2')=:\mathcal E_2''.
\end{eqnarray*}
Setting $\mathcal E_2=
\mathcal E_2'+\mathcal E_2''$ gives the desired conclusion.
\end{proof}

\subsection{Estimate on the life span and a summary of a priori estimates}
\begin{proposition}
\label{PropEstT}
  Let $T_{\text{ode}}$ and $T_{\text{ode}}^*$ be as in Proposition \ref{PropEstL2-NV}, and $\mathcal E_1$ in Proposition \ref{PropEstH1-NV}. Then, $
    T_{\text{ode}}^*=
    \min\left\{
    T_{\text{ode}},1,\frac{\underline J^2}{4\mathcal E_1}\right\}.$
\end{proposition}

\begin{proof}
Note that $\varrho_0\geq\underline\varrho>0$. Propositions \ref{PropEstL2-NV}--\ref{PropEstH1'-NV} imply
\begin{equation*}
  \sup_{0\leq t\leq T_{\text{ode}}^*}(\|J_y\|_2+\|v\|_{H^1}+\|\vartheta\|_{H^1})<\infty.
\end{equation*}
It follows from the definition of $T_*$ and $T_{\text{ode}}^*=\min\{T_*,T_{\text{ode}}, 1\}$ that
$$
\sup_{0\leq t\leq T_{\text{ode}}^*}\left((\inf_{y\in\mathbb R}J)^{-1}+
\sup_{y\in\mathbb R}J\right)<\infty.
$$
Thus, $T_{\text{ode}}^*<T_{\text{max}}$.

Then, (\ref{EqJ}) and Proposition \ref{PropEstH1-NV} imply
\begin{eqnarray*}
 & J=J_0+\int_0^tv_yd\tau \geq \underline J-\left(\int_0^t\|v_y\|_\infty^2d\tau\right)^{\frac12}t^{\frac12}
   \geq \underline J-\mathcal E_1^{\frac12}t^{\frac12}
  \geq\frac{\underline J}{2},\\
 & J \leq \bar J+t^{\frac12}\left(\int_0^t\|v_y\|_\infty^2d\tau\right)^{\frac12}
  \leq\bar J+\mathcal E_1^{\frac12}t^{\frac12}\leq\frac{3\bar J}{2},
\end{eqnarray*}
for all $t\leq T_{**}$, where
$$
T_{**}:=\min\left\{T_{\text{ode}}^*,\frac{\underline J^2}{4\mathcal E_1},\frac{\bar J^2}{4\mathcal E_1}\right\}=\min\left\{T_*,
T_{\text{ode}},1,\frac{\underline J^2}{4\mathcal E_1}\right\},$$
with $T_{\text{ode}}$ given in Proposition \ref{PropEstL2-NV}.

Note that
$\frac{\underline J}{3}<\frac{\underline J}{2}
\leq J \leq\frac{3\bar J}{2}<3\bar J$ on $\mathbb R
\times[0,T_{**}],$ $H^1(\mathbb R)\hookrightarrow C(\overline{\mathbb R})$, $J-J_0\in C([0,T_\text{max});H^1(\mathbb R))$, and $T_{**}\leq T_{\text{ode}}^*<T_{\text{max}}$.
There is a positive constant $T_{**}^+\in(T_{**}, T_{\text{max}})$, such that
$\frac{\underline J}{3}\leq J\leq3\bar J$ on $\mathbb R\times[0,T_{**}^+].$
Thanks to this and the definition of $T_*$ in (\ref{T_*}), one has
$$
\min\left\{T_*,
T_{\text{ode}},1,\frac{\underline J^2}{4\mathcal E_1}\right\}=T_{**}<T_{**}^+\leq T_*,
$$
which implies $T_*>\min\left\{
T_{\text{ode}},1,\frac{\underline J^2}{4\mathcal E_1}\right\}$. Thus,
$$
T_{\text{ode}}^*=\min\{T_{\text{ode}},1,T_*\}=
\min\left\{
T_{\text{ode}},1,\frac{\underline J^2}{4\mathcal E_1}\right\},
$$
which yields the desired conclusion.
\end{proof}

As a consequence of Propositions \ref{PropLocNovacuum}--\ref{PropEstT}, we have the following:

\begin{theorem}
  \label{ThmLocNV}
  Under the assumption (\ref{ass-nv}), there are two positive constants $\mathcal T$ and $\mathcal E$ depending only on $c_v, R,\mu,\kappa, \bar\varrho, \underline J, \bar J, \|\varrho_0'\|_{\infty}, \|\sqrt{\varrho_0}v_0\|_2, \|\sqrt{\varrho_0}v_0^2\|_2, \|\sqrt{\varrho_0}\vartheta_0\|_2$, $\|v_0'\|_2$, $\|\sqrt{\varrho_0}J_0'\|_2^2$, and $\left\|
  \varrho_0^\frac32\vartheta_0'\right\|_2^2$, but independent of $\underline\varrho$, such that the problem (\ref{EqJ})--(\ref{IC}) has
a unique solution $(J, v,\vartheta)$ on $\mathbb R\times[0,\mathcal T]$, satisfying
\begin{eqnarray*}
 \frac{\underline J}{2}\leq J\leq2\bar J,\quad\vartheta\geq0,\quad\mbox{ on }\mathbb R\times[0,\mathcal T],&&\\
  \sup_{0\leq t\leq\mathcal
  T}\|(J_t,\sqrt{\varrho_0}J_y)\|_2^2
  +\int_0^{\mathcal T}\|\sqrt{\varrho_0}J_{yt}\|_2^2 dt&\leq&\mathcal E,\\
  \displaystyle\sup_{0\leq t\leq\mathcal
  T}\|(\sqrt{\varrho_0}v,\sqrt{\varrho_0}v^2,v_y)\|_2^2
  +\int_0^{\mathcal
  T}\|(vv_y,\sqrt{\varrho_0}v_t,\sqrt{\varrho_0}v_{yy})\|_2^2) dt&\leq&\mathcal E,\\
  \displaystyle\sup_{0\leq t\leq\mathcal
  T}\|(\sqrt{\varrho_0}\vartheta,\varrho_0^{\frac32}\vartheta_y)\|_2^2
  +\int_0^{\mathcal T}\left\|\left(\vartheta_y,\varrho_0^2\vartheta_t,
  \varrho_0\left(\frac{\vartheta_y}{J}\right)_y\right)\right\|_2^2dt &\leq&\mathcal E.
\end{eqnarray*}
\end{theorem}

\section{Local existence in the presence of far field vacuum}
\label{SecLoc-v}
The aim of this section is to establish the local existence of solutions to the problem (\ref{EqJ})--(\ref{IC}), with vacuum at the far field only.

\begin{theorem}
  \label{ThmLocV}
Let $\bar\varrho, \underline J,$ and $\bar J$ be positive constants. Assume that
\begin{equation}\label{ass-v}
  \left\{
    \begin{aligned}
      &0<\varrho_0(y)\leq\bar\varrho,\quad \underline J\leq J_0(y)\leq\bar J, \quad \vartheta_0(y)\geq0, \quad\forall y\in\mathbb R,\\
      &\varrho_0'\in L^\infty(\mathbb
      R), \quad \left(\sqrt{\varrho_0}J_0',\sqrt{\varrho_0}v_0, \sqrt{\varrho_0}v_0^2, v_0', \sqrt{\varrho_0}\vartheta_0, \varrho_0^{\frac32}\vartheta_0'\right)\in L^2(\mathbb R).
    \end{aligned}
  \right.
\end{equation}

Then, there is a positive time $\mathcal T$ depending only on $c_v, R,\mu, \kappa,$ and
\begin{equation*}
  \left\{
\begin{array}{l}
\bar\varrho, \underline J, \bar J, \|\varrho_0'\|_{\infty},\|\sqrt{\varrho_0}v_0\|_2, \|\sqrt{\varrho_0}v_0^2\|_2, \\
\|\sqrt{\varrho_0}\vartheta_0\|_2, \|v_0'\|_2,\|\sqrt{\varrho_0}J_0'\|_2,
\|\varrho_0^\frac32\vartheta_0'\|_2,
\end{array}
\right.
\end{equation*}
such that the problem (\ref{EqJ})--(\ref{IC}), on $\mathbb R\times[0,\mathcal T]$, has
a solution $(J, v,\vartheta)$, satisfying $\frac{\underline J}{2}\leq J\leq2\bar J$ and $\vartheta\geq0$ on $\mathbb R\times[0,\mathcal T],$ and
\begin{eqnarray*}
&J_t,\sqrt{\varrho_0}J_y, \sqrt{\varrho_0}v,\sqrt{\varrho_0}v^2, v_y, \sqrt{\varrho_0}\vartheta, \varrho_0^{\frac32}\vartheta_y\in L^\infty(0,\mathcal T; L^2(\mathbb R)), \\
&\sqrt{\varrho_0}J_{yt}, vv_y,\sqrt{\varrho_0}v_t,\sqrt{\varrho_0}v_{yy}, \vartheta_y, \varrho_0(\frac{\vartheta_y}{J})_y, \varrho_0^2\vartheta_t\in L^2(0,\mathcal T; L^2(\mathbb R)).
\end{eqnarray*}
\end{theorem}

\begin{proof}
We first construct a sequence $\{(\varrho_{0n}, J_{0n}, v_{0n}, \vartheta_{0n})\}_{n=1}^\infty$ approximating $(\varrho_0, J_0, v_0, \vartheta_0)$, so that Theorem \ref{ThmLocNV} applies.

For each integer $n\geq1$, choose $0<\delta_n\leq\frac1n$ sufficiently small such that
\begin{equation}
  \label{3.1}
\delta_n\max\left\{\|v_0\|_{L^2(I_{2n})}^2, \|v_0\|_{L^4(I_{2n})}^4,
\|\vartheta_0\|_{L^2(I_{2n})}^2,\|\vartheta_0'\|_{L^2(I_{2n})}^2\right\}
\leq1,
\end{equation}
where $I_{2n}=(-2n,2n)$. Choose $\varphi\in C_0^\infty((-2,2))$, with $\varphi\equiv1$ on $(-1,1)$, $0\leq\varphi\leq1$ on $(-2,2)$,
and $|\varphi'|\leq2$. Since $v_0', \vartheta_0'\in L^2(\mathbb R)$, it is clear $v_0,\vartheta_0\in C(\mathbb R)$, and
\begin{eqnarray}
  |v_0(y)|
&\leq&|v_0(0)|+\|v_0'\|_2\sqrt{|y|}\leq A_0\sqrt{1+|y|},\label{3.2}
\end{eqnarray}
where $A_0=\max\left\{|v_0(0)|,\|v_0'\|_2\right\}$.

Define $\varrho_{0n}, J_{0n}, v_{0n},$ and $\vartheta_{0n}$ as
\begin{eqnarray*}
&&\varrho_{0n}=\delta_n+\varrho_0,\quad
J_{0n}=J_0,\quad v_{0n}=\varphi\left(\frac{\cdot}{n}\right) v_0, \quad \vartheta_{0n}=
\varphi\left(\frac{\cdot}{n}\right)\vartheta_0.
\end{eqnarray*}
Then
\begin{equation}
  \label{3.4}
0<\delta_n\leq\varrho_{0n}\leq\bar\varrho+1,\quad
\|\varrho_{0n}'\|_\infty=\|\varrho_0'\|_\infty.
\end{equation}
(\ref{3.1}) and (\ref{3.2}) imply that
\begin{eqnarray}
  \|\sqrt{\varrho_{0n}}v_{0n}\|_2^2
&\leq& \|\sqrt{\varrho_0}v_0\|_2^2+ \delta_n\|v_0\|_{L^2((-2n,2n))}^2\leq\|\sqrt{\varrho_0}v_0\|_2^2+1,
\label{3.5}\\
\|\sqrt{\varrho_{0n}}v_{0n}^2\|_2^2
&\leq& \|\sqrt{\varrho_0}v_0^2\|_2^2+ \delta_n\|v_0\|_{L^4((-2n,2n))}^4\leq\|\sqrt{\varrho_0}v_0^2\|_2^2+1,
\nonumber\\
  \|v_{0n}'\|_2^2
&\leq&
2\|v_0'\|_2^2+64A_0^2,\label{3.6}\\
\|\sqrt{\varrho_{0n}}\vartheta_{0n}\|_2^2&\leq&\|\sqrt{\varrho_0}\vartheta_0
\|_2^2+1.  \label{3.7}
\end{eqnarray}
Due to (\ref{3.1}) and $0\leq\delta_n\leq\frac1n$, one can get
\begin{eqnarray}
  \label{3.8}
\|\varrho_{0n}^{\frac32}\vartheta_{0n}'\|_2^2
&\leq&8\int_{-2n}^{2n} (\varrho_0^3+\delta_n^3)
\left(|\vartheta_0'|^2+\frac{4}{n^2}\vartheta_0^2\right)dy\nonumber\\
&\leq&8\left(\|\varrho_0^{\frac32}\vartheta_0'\|_2^2 +\delta_n^3 \|\vartheta_0'\|_2^2+4\bar\varrho^2
\|\sqrt{\varrho_0}\vartheta_0\|_2^2+4\delta_n^3\|\vartheta_0
\|_{L^2((-2n,2n))}^2\right)\nonumber\\
&\leq&8\left(\|\varrho_0^{\frac32}\vartheta_0'\|_2^2 +5+4\bar\varrho^2
\|\sqrt{\varrho_0}\vartheta_0\|_2^2\right).
\end{eqnarray}

Since $(\varrho_{0n}, J_{0n}, v_{0n}, \vartheta_{0n})$ fulfills the assumption (\ref{ass-nv}), with $\underline\varrho$ and $\bar\varrho$ replaced by $\delta_n$ and $\bar\varrho+1$, respectively, by (\ref{3.4})--(\ref{3.8}) and Theorem \ref{ThmLocNV}, there
is a positive time $\mathcal T$ depending only on the quantities stated in Theorem \ref{ThmLocNV}, which in particular is independent of $n$, such that the problem (\ref{EqJ})--(\ref{IC}), has a
unique solution $(J_n, v_n, \vartheta_n)$, satisfying
\begin{align}
&\frac{\underline J}{2}\leq J_n\leq2\bar J,\quad\vartheta_n\geq0,\quad\mbox{ on }\mathbb R\times[0,\mathcal T],\label{AE0}
\\
&\sup_{0\leq t\leq\mathcal T}
\|(\partial_tJ_n,\sqrt{\varrho_{0n}}\partial_yJ_n)\|_2^2
+\int_0^{\mathcal T}\|\sqrt{\varrho_{0n}}\partial_{yt}J_n\|_2^2 dt\leq\mathcal E,\label{AE1}\\
&\sup_{0\leq t\leq\mathcal T}\|(\varrho_{0n}^{\frac12}v_n,\varrho_{0n}^{\frac12}v_n^2,\partial_yv_n)\|_2^2
+\int_0^{\mathcal
  T}\|(v_n\partial_yv_n, \varrho_{0n}^{\frac12}\partial_tv_n,
  \varrho_{0n}^{\frac12}\partial_{yy}v_n) \|_2^2dt\leq\mathcal E,
\label{AE2}\\
&\sup_{0\leq t\leq\mathcal
  T}\left\|\left(\sqrt{\varrho_{0n}}\vartheta_n,
\varrho_{0n}^{\frac32}
  \partial_y\vartheta_n\right)\right\|_2^2+
  \int_0^\mathcal T
\left\|\left(\partial_y\vartheta_n,\varrho_{0n}\left(\frac{\partial_y\vartheta_n}{J_n}\right)_y,\varrho_{0n}^2\partial_t\vartheta_n
\right)\right\|_2^2dt \leq\mathcal E,\label{AE3}
\end{align}
for a positive constant $\mathcal E$ independent of $n$.

Since $\varrho_0'\in L^\infty(\mathbb R)$ and $\varrho_0(y)>0$
for all $y\in\mathbb R$, so $\min_{|y|\leq R}\varrho_0>0$
for any $R\in\mathbb R$. Thus, it follows from (\ref{AE0})--(\ref{AE3}) that
\begin{eqnarray*}
\|(J_n,v_n,\vartheta_n)\|_{L^\infty(0,\mathcal T; H^1(I_k))}, \|v_n\|_{L^2(0,\mathcal T; H^2(I_k))}&\leq&\mathcal E_k,\\
\|\partial_tJ_n\|_{L^2(0,\mathcal T; H^1(I_k))},
\|(\partial_t v_n,\partial_t\vartheta_n)\|_{L^2(0,\mathcal T; L^2(I_k))}&\leq&\mathcal E_k,
\end{eqnarray*}
for any positive integer $k$, where $I_k=(-k,k)$ and $\mathcal E_k$ is a positive constant independent of $n$. Then, by the Cantor's diagonal argument in both $n$ and $k$, there is a subsequence, denoted still by $\{(J_n, v_n, \vartheta_n)\}_{n=1}^\infty$, and $(J, v, \vartheta)$, such that
\begin{eqnarray}
&(J_n,v_n,\vartheta_n)\rightharpoonup^*(J,v,\vartheta)\quad\mbox{in }L^\infty(0,\mathcal T; H^1(I_R)),\label{WC1}
\\
&v_n\rightharpoonup v\quad\mbox{in }L^2(0,\mathcal T;
H^2(I_R)),\label{WC2}\\
&\partial_tJ_n\rightharpoonup^*\partial_tJ\quad\mbox{in
}L^\infty(0,\mathcal T; L^2(I_R)),\label{WC3}\\
&\partial_tJ_n\rightharpoonup\partial_tJ_n\quad\mbox{in }L^2(0,\mathcal
T; H^1(I_R)),\label{WC4}\\
&(\partial_tv_n,\partial_t\vartheta_n)\rightharpoonup
(\partial_tv,\partial_t\vartheta)\quad\mbox{in
}L^2(0,\mathcal T; L^2(I_R)),\label{WC5}
\end{eqnarray}
for any $R\in(0,\infty)$, where $\rightharpoonup$ and
$\rightharpoonup^*$ denote the weak and weak-* convergence,
respectively, in the corresponding spaces, and $I_R=(-R,R)$. Moreover, noticing that
$H^1(I_R)\hookrightarrow\hookrightarrow C(\overline{I_R})$,
by the Aubin-Lions
lemma, and using the Cantor's diagonal argument again (in both $n$ and
$k$), one can get a subsequence of the previous subsequence, denoted
still by $\{(J_n, v_n, \vartheta_n)\}_{n=1}^\infty$, such that
\begin{eqnarray}
 &(J_n,v_n,\vartheta_n) \rightarrow (J,v,\vartheta)\quad \mbox{in }C([0,\mathcal T]; C(\overline{I_R})),\label{SC1}\\
 &v_n\rightarrow v\quad\mbox{in }L^2(0,\mathcal T; H^1(I_R)), \label{SC2}
\end{eqnarray}
for any $R\in(0,\infty)$. These and (\ref{AE0}) imply that
\begin{equation}\label{PT}
\frac{\underline J}{2}\leq J\leq 2\bar J,\quad \vartheta\geq0,\quad\mbox{ on }\mathbb R\times[0,\mathcal T].
\end{equation}
It follows from (\ref{AE0}), (\ref{AE3}), (\ref{WC1}), (\ref{SC1}), and (\ref{PT}) that for any $R\in(0,\infty)$
\begin{equation}\label{VC5}
  \frac{\partial_y\vartheta_n}{J_n} \rightharpoonup\frac{\vartheta_y}{J}\quad\mbox{in }L^2(0,\mathcal T; H^1(I_R)).
\end{equation}

Thanks to the convergences (\ref{WC1})--(\ref{SC2}), and (\ref{VC5}), as well as the a priori estimates
(\ref{AE1})--(\ref{AE3}), one can obtain by the weakly lower semi-continuity of norms that $(J, v, \vartheta)$ has the
regularities stated in the proposition.
Besides, by (\ref{WC1})--(\ref{SC2}) and (\ref{VC5}), one can take the limit, as
$n\rightarrow\infty$, to conclude that $(J, v, \vartheta)$ satisfies equations (\ref{EqJ})--(\ref{EqTheta}), in the sense
of distribution. However, due to the regularities of $(J, v, \vartheta)$ and the positivity of $\varrho_0$
on $\mathbb R$, one can show that
the equations are satisfied a.e.\,
in $\mathbb R\times(0,\mathcal T)$. While the initial condition (\ref{IC}) is guaranteed by (\ref{SC1}) and (\ref{SC2}).
Therefore, $(J, v, \vartheta)$ is the desired solution to the problem (\ref{EqJ})--(\ref{IC}). This completes the proof.
\end{proof}

\section{Global well-posedness in the presence of far field vacuum}
\label{SecGlo-v}
This section is devoted to proving the global existence and uniqueness of solutions in the presence of far field vacuum via  establishing a series of a priori estimates, which are finite up to any finite time.
Throughout this section, we will suppose, in addition to the assumption (\ref{ass-v}), that
\begin{equation}
  \varrho_0\in L^1(\mathbb R),\quad\frac{J_0'}{\sqrt{\varrho_0}},\quad\sqrt{\varrho_0}\vartheta_0'\in L^2(\mathbb R),\label{finitemass}
\end{equation}
and, for some given positive constant $K_1$,
\begin{equation}
|\varrho_0'(y)|\leq K_1\varrho_0^{\frac32}(y),\quad y\in\mathbb R.\label{HSLOW}
\end{equation}

\begin{remark}
\label{WEAKERSLOWDECAY}
It should be noticed that though (\ref{HSLOW}) is assumed throughout this section,
yet it is not needed for some results (say, Propositions \ref{PropBasic}--\ref{PropIdJ} and Corollary \ref{CorJlower}),
while for some others (Proposition \ref{PropEstL2-V} and Proposition \ref{PropEstJy}),
one needs only the following weaker assumption
$$
|\varrho_0'|\leq \tilde K_1\varrho_0, \mbox{ on }\mathbb R,\mbox{
for some positive constant }\tilde K_1.$$

Note that the above weaker assumption can be satisfied even if the initial density decays very fast. It is an interesting problem to see if all the results in this section (and thus the well-posedness) still hold without (\ref{HSLOW}) or under the weaker assumption.
\end{remark}

In the rest of this section, we always assume that $(J,v,\vartheta)$ is
a solution to the problem (\ref{EqJ})--(\ref{IC}), in $\mathbb R\times(0,T)$, for some positive
time $T$, satisfying
\begin{eqnarray*}
&&0<J,J^{-1}\in L^\infty(0,T;L^\infty(\mathbb R)),\quad\vartheta\geq0,\\
&&J_t,\sqrt{\varrho_0}J_y, \sqrt{\varrho_0}v, \sqrt{\varrho_0}v^2, v_y, \sqrt{\varrho_0}\vartheta, \varrho_0^{\frac32}\vartheta_y\in L^\infty(0, T; L^2(\mathbb R)), \\
&&
\sqrt{\varrho_0}J_{yt}, vv_y, \sqrt{\varrho_0}v_t,\sqrt{\varrho_0}v_{yy}, \vartheta_y, \varrho_0^2\vartheta_t\in L^2(0, T; L^2(\mathbb R)).
\end{eqnarray*}

\subsection{Basic estimates and the control of $J$}
\label{ssecBscV}
The basic energy estimates, uniform positive lower bound of $J$, and a control on the upper bound of $J$ are derived
in this subsection. We start with the conservation of the energy.

\begin{proposition}
\label{PropBasic}
Set $\mathscr E_0:=\int_\mathbb R\varrho_0\left(\frac{v_0^2}{2}+c_v\vartheta_0\right)dy$. Then
$$
\left[\int_\mathbb R\varrho_0\left(\frac{v^2}{2}+c_v\vartheta\right)dy\right](t)=\mathscr E_0.
$$
\end{proposition}

\begin{proof} Let $\varphi$ be the cut-off function given in the proof of Theorem \ref{ThmLocV}, and set
$\varphi_r(\cdot)=\varphi\left(\frac \cdot r\right).$ Testing (\ref{EqE}) by $\varphi_r$ yields
\begin{equation}
\int_\mathbb R\varrho_0E\varphi_rdy =
\int_\mathbb R\varrho_0E_0\varphi_rdy
-\int_0^t\int_\mathbb R\frac{\varphi_r'}J\left(\kappa\vartheta_y+\mu vv_y-R\varrho_0\vartheta v\right)dyd\tau.
\label{BASIC.1}
\end{equation}
Direct calculations show that
\begin{eqnarray*}
&&\left|\int_0^t\int_\mathbb R\frac{\varphi_r'}J\left(\kappa\vartheta_y+\mu vv_y-R\varrho_0\vartheta v\right)dyd\tau\right|\nonumber \\
&\leq&\frac Cr\int_0^t\int_{r\leq|y|\leq2r}\left(|\vartheta_y|
+|v||v_y|+\varrho_0\vartheta|v|\right)dy
d\tau\nonumber \\
&\leq&\frac Cr\int_0^t\left[(\|\vartheta_y\|_2+\|vv_y\|_2)\sqrt{r}+\|\sqrt{\varrho_0}\vartheta\|_2\|\sqrt{\varrho_0}
v\|_2\right]d\tau\\
&\leq&\frac{C}{\sqrt{r}}\left(1+\|\left(\vartheta_y,vv_y,\sqrt{\varrho_0}v,
\sqrt{\varrho_0}\vartheta\right)\|_{L^2(\mathbb R\times(0,t))}^2\right),
\end{eqnarray*}
for any $r\geq1$, where $C$ is a positive constant independent of $r$ but may depend on $t$.
Then, taking $r\rightarrow\infty$ in (\ref{BASIC.1}) gives the desired identity.
\end{proof}

The equality for $J$ in the next proposition is in the spirit of Kazhikov-Shelukin
\cite{KAZHIKOV77}, where the mass Lagrangian
coordinate, rather than the flow map, was considered.

\begin{proposition}
\label{PropIdJ}
It holds that for any $(y,t)\in\mathbb R\times(0,T)$
$$
J(y,t)=B(y,t)
\left(J_0(y)
+\frac R\mu\int_0^t
\frac{\varrho_0(y)\vartheta(y,\tau)}{B(y,\tau)} d\tau\right),
$$
where $B(y,t)=\exp\left\{\frac1\mu\int_{-\infty}^y\varrho_0(v-v_0)dy'\right\}.$
\end{proposition}

\begin{proof}
It follows from (\ref{EqJ}) and (\ref{EqV}) that
\begin{equation*}
  \varrho_0v-\varrho_0v_0-\mu\left[(\log J)_y-(\log J_0)'\right]
+R\int_0^t\left(\frac{\varrho_0\vartheta}{J}\right)_yd\tau =0.
\end{equation*}
Integrating the above equation in the spatial variable over $(z,y)$ yields
\begin{eqnarray*}
\int_{-\infty}^y(\varrho_0 v-\varrho_0v_0)dy'-\mu\left(\log J(y,t)-\log J_0(y)\right)
+R\int_0^t\frac{\varrho_0(y)\vartheta(y,\tau)}{J(y,\tau)}d\tau \\
=\int_{-\infty}^z(\varrho_0v-\varrho_0v_0)dy'
-\mu\left(\log J(z,t)-\log J_0(z)\right)
+R\int_0^t\frac{\varrho_0(z)\vartheta(z,\tau)}{J(z,\tau)}d\tau.
\end{eqnarray*}
Therefore, there is a function $f(t)$ independent of $y$ such that
\begin{equation}\label{f(t)}
\int_{-\infty}^y(\varrho_0 v-\varrho_0v_0)dy'-\mu\left(\log J-\log J_0\right)
+R\int_0^t\frac{\varrho_0\vartheta}{J }d\tau=f(t).
\end{equation}

We claim that $f\equiv0$. 
Set $\delta_T:=\inf_{(y,t)\in\mathbb R\times(0,T)}J(y,t)>0$. It follows from (\ref{EqJ}) and $v_y\in L^2(\mathbb R\times(0,T))$ that
\begin{eqnarray*}
\Bigg|\int_{-(k+1)}^{-k}(\log J-\log J_0)dy\Bigg|=\left|\int_{-(k+1)}^{-k}\int_0^t\frac{v_y}{J}d\tau dy\right|&&\\
\leq\delta_T^{-1}\sqrt t\|v_y\|_{L^2((-(k+1),-k)\times(0,t))}
&\rightarrow&0,\quad \mbox{as }k\rightarrow\infty.
\end{eqnarray*}
While $\varrho_0\vartheta\in L^1(0,T; L^1(\mathbb R))$ yields
\begin{eqnarray*}
  \int_{-(k+1)}^{-k}\int_0^t\frac{\varrho_0\vartheta}{J}d\tau dy
\leq\delta_T^{-1}\|\varrho_0\vartheta\|_{L^1((-(k+1),-k)\times(0,t))}
\rightarrow0, \quad\mbox{as }k\rightarrow\infty.
\end{eqnarray*}
Since $\varrho_0\in L^1(\mathbb R)$ and $\sqrt{\varrho_0}v\in L^\infty(0,T; L^2(\mathbb R))$, one has
\begin{equation*}
\left|\int_{-(k+1)}^{-k}\int_{-\infty}^y\varrho_0vdy'dy\right|
 \leq\left(\int_{-\infty}^{-k}\varrho_0 dy\right)^{\frac12}\left(\int_{-\infty}^{-k}\varrho_0v^2\right)^{\frac12}
 \rightarrow0, \quad\mbox{as }k\rightarrow\infty.
\end{equation*}
Hence, $f(t)\equiv0$, and, consequently, (\ref{f(t)}) gives
$$
\int_{-\infty}^y\varrho_0 vdy'-\mu\log J
+R\int_0^t\frac{\varrho_0 \vartheta }{J }d\tau
=\int_{-\infty}^y\varrho_0v_0dy'-\mu\log J_0.
$$
Dividing both sides by $\mu$ and taking the exponential yield
\begin{equation} \frac1J\exp\left\{\frac{R}{\mu}\int_0^t\frac{\varrho_0\vartheta}{J}d\tau
\right\}
=\exp\left\{\frac1\mu\int_{-\infty}^y\varrho_0(v_0-v)dy'\right\}
\frac{1}{J_0}.\label{IdJ.4}
\end{equation}
Multiplying (\ref{IdJ.4}) by $\frac{R}{\mu}\varrho_0\vartheta$ and integrating in $t$ yield
$$
\exp\left\{\frac{R}{\mu}\int_0^t\frac{\varrho_0
\vartheta}{J}d\tau
\right\}
=1+\frac{R\varrho_0}{\mu J_0}\int_0^t\exp\left\{\frac1\mu\int_{-\infty}^y\varrho_0(v_0-v)
dy'\right\}\vartheta d\tau.
$$
Substituting the above into (\ref{IdJ.4}) gives
\begin{eqnarray*}
\frac1J\left(1+\frac{R\varrho_0}{\mu J_0}\int_0^te^{\frac1\mu\int_{-\infty}^y\varrho_0(v_0-v)
dy'}\vartheta d\tau\right)
=e^{\frac1\mu\int_{-\infty}^y\varrho_0(v_0-v)dy'}
\frac{1}{J_0},
\end{eqnarray*}
which yields the desired expression for $J$.
\end{proof}

As a corollary of Propositions \ref{PropBasic} and \ref{PropIdJ}, one can obtain the uniform positive lower bound of $J$ and the upper control of $J$ stated as follows.

\begin{corollary}
  \label{CorJlower}
It holds that
\begin{eqnarray*}
J&\geq&\underline Je^{-\frac{2\sqrt2}{\mu}\sqrt{\|\varrho_0\|_1\mathcal E_0}},\quad\mbox{ and }\\
\|J\|_\infty(t)&\leq& e^{\frac{4\sqrt2}{\mu}\sqrt{\|\varrho_0\|_1\mathcal E_0}}
\left(\bar J+\frac R\mu\int_0^t\|\varrho_0\vartheta\|_\infty d\tau\right).
\end{eqnarray*}
\end{corollary}

\begin{proof}
Proposition \ref{PropBasic} implies
\begin{eqnarray*}
 \left|\int_{-\infty}^y\varrho_0(v-v_0)dy'\right|
 &\leq& \left(\int_\mathbb R\varrho_0dz\right)^{\frac12}
\left[\left(\int_\mathbb R\varrho_0v^2dy\right)^{\frac12}+
\left(\int_\mathbb R\varrho_0v_0^2dy\right)^{\frac12}\right]\\
&\leq&2\sqrt{2\|\varrho_0\|_1\mathscr E_0}.
\end{eqnarray*}
Therefore, it follows from the definition of $B$ in Proposition \ref{PropIdJ} that
\begin{equation*}
\label{EstB}
\exp\left\{-\frac{2\sqrt2}{\mu}\sqrt{\|\varrho_0\|_1\mathscr E_0}\right\}\leq B(y,t)
\leq \exp\left\{\frac{2\sqrt2}{\mu}\sqrt{\|\varrho_0\|_1\mathscr E_0}\right\}.
\end{equation*}
Due to this and $\vartheta\geq0$, the conclusions follow easily from Proposition \ref{PropIdJ}.
\end{proof}


\subsection{$L^2$ estimates}
\label{ssecL2V}
We now turn to derive the $L^\infty(0,T; L^2(\mathbb R))$ a priori estimates on $(J, v,\vartheta)$. We need an elementary lemma.

\begin{lemma}
\label{lem}
Let $\omega$ and $\eta$ be nonnegative and bounded on $\mathbb R$, satisfying
$|\omega'|\leq K|\omega|$ and $\eta>0$ on $\mathbb R$, for some positive constant $K$. Assume that $f$ is a nonnegative measurable function on $\mathbb R$ such that $\sqrt\omega f, \frac{f'}{\sqrt\eta}\in L^2(\mathbb R)$,
and $\omega f\in L^1(\mathbb R)$. Then,
$$
\|\sqrt\omega f\|_\infty^2\leq 2K\|\sqrt\omega f\|_2^2+8\|\omega\|_\infty^{\frac13}\|\omega f\|_1^{\frac23}
\left\|\frac{f'}{\sqrt\eta}\right\|_2^{\frac43}\|\eta\|_\infty^{\frac23}.
$$
\end{lemma}

\begin{proof}
By assumptions and elementary calculations, one deduces
\begin{eqnarray*}
\|\sqrt\omega f\|_\infty^2&\leq& \int_\mathbb R(|\omega'|f^2+2\omega f|f'|)dz\\
&\leq&K\|\sqrt\omega f\|_2^2+2\|\sqrt\omega f\|_\infty^{\frac12}
\|\omega\|_\infty^{\frac14}\|\omega f\|_1^{\frac12}\left\|\frac{f'}{\sqrt\eta}\right\|_2 \|\eta\|_\infty^{\frac12}\\
&\leq&2^{\frac53}\|\omega\|_\infty^{\frac13}\|\omega f\|_1^{\frac23}\left\|\frac{f'}{\sqrt\eta}\right\|_2^{\frac43}
\|\eta\|_\infty^{\frac23} +\frac12\|\sqrt\omega f\|_\infty^2+K\|\sqrt\omega f\|_2^2,
\end{eqnarray*}
which yields the desired conclusion.
\end{proof}

Now we are ready to derive the $L^\infty(0,T; L^2)$ estimates.

\begin{proposition}
  \label{PropEstL2-V}
It holds that
\begin{align*}
  \sup_{0\leq t\leq T}(\|\sqrt{\varrho_0}E\|_2^2+\|J\|_\infty^2)
&+\int_0^T(\|v_y\|_2^2+\||v|v_y\|_2^2+\|\vartheta_y\|_2^2+\|\sqrt{\varrho_0}\vartheta\|_\infty^2) dt\nonumber\\
\leq&\ \  C(1+\|\sqrt{\varrho_0}E_0\|_2^2),
\end{align*}
for a positive constant $C$ depending only on $\mu, \kappa, c_v, R, K_1, \bar\varrho, \underline J, \bar J, \|\varrho_0\|_1, T,$ and $\mathcal E_0$.
\end{proposition}

\begin{proof}
Let $\varphi_r$ be given as in the proof of Proposition \ref{PropBasic}.
Testing (\ref{EqV}) with $v\varphi_r^2$ yields
\begin{eqnarray}
  &&\frac12\frac{d}{dt}\int_\mathbb R\varrho_0v^2\varphi_r^2 dy+\mu\int_\mathbb R\frac{v_y^2}{J}
  \varphi_r^2 dy\nonumber\\
  &=&-2\mu\int_\mathbb R\frac{v_y}{J}v\varphi_r\varphi_r'dy+R\int_\mathbb R\frac{\varrho_0\vartheta}
  {J}(v_y\varphi_r^2+2v\varphi_r\varphi_r')dy\nonumber\\
  &\leq&\frac{C}{r}\left(\int_{r\leq|y|\leq2r}|v v_y|dy +\|\sqrt{\varrho_0}v\|_2
  \|\sqrt{\varrho_0}\vartheta\|_2\right)\nonumber\\
  &&+\frac\mu2\int_\mathbb R\frac{v_y^2}{J}\varphi_r^2dy+C\int_\mathbb R\varrho_0E^2 \varphi_r^2dy,\nonumber
\end{eqnarray}
where Corollary \ref{CorJlower} has been used. Therefore, recalling Proposition \ref{PropBasic}, we have
\begin{eqnarray}
   &&\frac{d}{dt}\int_\mathbb R\varrho_0v^2\varphi_r^2 dy+\mu\int_\mathbb R\frac{v_y^2}{J}
  \varphi_r^2 dy\nonumber\\
   &\leq& \frac{C}{r}\Big(\int_{r\leq|y|\leq2r} |v v_y |dy +
  \|\sqrt{\varrho_0}\vartheta\|_2\Big)+C\int_\mathbb R\varrho_0E^2\varphi_r^2dy.
  \label{L2-1}
\end{eqnarray}

Rewrite equation (\ref{EqE}) as
$\varrho_0E_t-\frac{\kappa}{c_v}\left(\frac{E_y}{J}\right)_y=\left(\mu-\frac{\kappa}{c_v}\right)
  \left(\frac{vv_y}{J}\right)_y-R\left(\frac{\varrho_0\vartheta v}{J}\right)_y$ and test it with $E\varphi_r^2$ to get
\begin{eqnarray}
  &&\frac12\frac{d}{dt}\int_\mathbb R\varrho_0E^2\varphi_r^2dy+\frac{\kappa}{c_v}\int_\mathbb R
  \frac{E_y^2}{J}\varphi_r^2 dy \nonumber\\
  &=&\int_\mathbb R\left[R\varrho_0\vartheta v+\left(\frac{\kappa}{c_v}-\mu\right) vv_y\right]\frac{E_y}{J}\varphi_r^2dy\nonumber\\
  &&+2\int_\mathbb R\left[-\frac{\kappa}{c_v}E_y+R\varrho_0\vartheta v+\left(\frac{\kappa}{c_v}-\mu\right)vv_y\right]\frac{E}{J}\varphi_r\varphi_r'dy\nonumber\\
  &\leq&\int_\mathbb R\left[\frac{\kappa}{2c_v} E_y^2+C\left(v^2
  v_y^2+\varrho_0^2\vartheta^2v^2\right)\right]\frac{\varphi_r^2}{J} dy \nonumber\\
  &&+\frac Cr\int_{r\leq|y|\leq 2r}E(|E_y|+\varrho_0\vartheta|v|+|v||v_y|)dy\nonumber\\
  &\leq&\frac{\kappa}{2c_v}\int_\mathbb R\frac{E_y^2}{J}\varphi_r^2 dy+C\int_\mathbb R\left(\frac{v^2
  v_y^2}{J}+\varrho_0^2\vartheta^2v^2\right) \varphi_r^2 dy
  \nonumber\\
  &&+\frac Cr\int_{r\leq|y|\leq 2r}\left[E(|\vartheta_y|+|v||v_y|)+\varrho_0E^{\frac52}\right]dy,
  \nonumber
\end{eqnarray}
where Corollary \ref{CorJlower} has been used.
Therefore,
\begin{eqnarray}
  && \frac{d}{dt}\int_\mathbb R\varrho_0E^2\varphi_r^2dy+\frac{\kappa}{c_v}\int_\mathbb R
  \frac{E_y^2}{J}\varphi_r^2 dy \nonumber\\
  &\leq&C\int_\mathbb R\left(\frac{v^2
  v_y^2}{J}+\varrho_0^2\vartheta^2v^2\right) \varphi_r^2 dy
  +\frac Cr\int_{r\leq|y|\leq 2r}\left[E(|\vartheta_y|+|v||v_y|)+\varrho_0E^{\frac52}\right]dy.
  \label{L2-2}
\end{eqnarray}
Similarly, taking the inner product of (\ref{EqV}) with $v^3\varphi_r^2$ leads to
\begin{eqnarray}
  && \frac{d}{dt}\int_\mathbb R\varrho_0v^4\varphi_r^2 dy +8\mu\int_\mathbb R
  \frac{v_y^2}{J}v^2\varphi_r^2 dy\nonumber \\
  &\leq&\frac Cr\int_{r\leq|y|\leq 2r}(\varrho_0E^{\frac52}+E|v||v_y|) dy+C\int_\mathbb R \varrho_0^2 v^2\vartheta^2 \varphi_r^2 dy. \label{L2-3}
\end{eqnarray}

Multiplying (\ref{L2-3}) with a sufficiently large positive number $M$ and adding the resultant with (\ref{L2-1}) and (\ref{L2-2}), one obtains
\begin{align*}
\frac{d}{dt}\int_\mathbb R&\varrho_0(v^2+E^2+Mv^4)\varphi_r^2 dy +\int_\mathbb R\frac{1}{J}
  \left(\mu v_y^2+\frac{\kappa}{c_v}E_y^2+\mu Mv^2v_y^2\right)\varphi_r^2dy \nonumber\\
  \leq&\ \ C\int_\mathbb R(\varrho_0E^2+\varrho_0^2\vartheta^2v^2)\varphi_r^2 dy+
  \frac Cr\left(\int_{r
  \leq|y|\leq 2r} |v||v_y| dy+ \|\sqrt{\varrho_0}\vartheta\|_2\right)\nonumber\\
  &\ \ +\frac Cr\int_{r\leq|y|\leq 2r}\big[\varrho_0E^{\frac52}+(|v||v_y|+|\vartheta_y|) E \big]dy.
\end{align*}
Integrating the above inequality in $t$ yields
\begin{align}
  \bigg(\int_\mathbb R\varrho_0E^2\varphi_r^2 dy&\bigg)(t)+\int_0^t\int_\mathbb R\frac{\varphi_r^2}{J}
  \left(v_y^2+E_y^2+v^2v_y^2+\vartheta_y^2\right)dyd\tau \nonumber\\
  \leq&\ \ C\left(1+\|\sqrt{\varrho_0}E_0\|_2^2+\int_0^t\int_\mathbb R(\varrho_0E^2+\varrho_0^2\vartheta^2v^2)\varphi_r^2 dyd\tau\right) \nonumber\\
  &\ \ +\frac Cr\int_0^t\int_{r\leq|y|\leq 2r}\big[\varrho_0E^{\frac52}+(|v||v_y|+|\vartheta_y|) E \big]dyd\tau \nonumber\\
  &\ \ +\frac Cr\int_0^t\left(\int_{r
  \leq|y|\leq 2r} |v||v_y| dy+ \|\sqrt{\varrho_0}\vartheta\|_2\right)d\tau.\label{L2-4}
\end{align}

We claim that the last two terms on the right-hand side of (\ref{L2-4})
tend to zero, as $r\rightarrow\infty$. Since $vv_y\in L^2(0,T; L^2)$ and $\sqrt{\varrho_0}\vartheta\in L^\infty(0,T; L^2)$, one deduces
\begin{eqnarray*}
  I_1&:=&\frac Cr\int_0^t\left(\int_{r
  \leq|y|\leq 2r} |v||v_y| dy+ \|\sqrt{\varrho_0}\vartheta\|_2\right)d\tau\\
  &\leq&\frac{Ct^{\frac12}}{r^{\frac12}}\left(\int_0^t\int_{r\leq|y|\leq 2r}v^2 v_y^2 dyd\tau\right)^{\frac12}+\frac{Ct}{r}\sup_{0\leq\tau\leq t}\|\sqrt{\varrho_0}\vartheta\|_2\rightarrow0, \quad\mbox{as }r\rightarrow\infty.
\end{eqnarray*}
For $t\in(0,T)$, choose $\xi(t)\in(-1,1)$
such that $E^2(\xi(t),t)\leq\frac{2\int_{-1}^1\varrho_0E^2dz}{\int_{-1}^1\varrho_0dz}$. Then,
\begin{equation*}
  |E(y,t)|=\left|E(\xi(t),t)+\int_{\xi(t)}^yE_y(z,t)dz\right|
   \leq C\|\sqrt{\varrho_0}E\|_2 +\|E_y\|_2^{\frac12}(|y|+1)^{\frac12},\quad\forall y\in\mathbb R.
\end{equation*}
Hence, for any $r\geq1$, it holds that
\begin{equation}
  \label{UnifE}
  |E(y,t)|\leq C\left(\|\sqrt{\varrho_0}E\|_2 +\|E_y\|_2^{\frac12}r^{\frac12}\right),\quad
  \forall r\leq|y|\leq2r.
\end{equation}
It follows from (\ref{UnifE}), $\sqrt{\varrho_0}E\in L^\infty(0,T; L^2(\mathbb R))$, $E_y\in L^2(0,T; L^2(\mathbb R))$, and
\begin{eqnarray*}
  I_2&:=&\frac1r\int_0^t\int_{r\leq|y|\leq2r}\varrho_0E^{\frac52}dyd\tau \\
  &\leq&\frac Cr\int_0^t\int_{r\leq|y|\leq2r}\left(\|\sqrt{\varrho_0}E\|_2^\frac12 +\|E_y\|_2^{\frac14}r^{\frac14}\right)\varrho_0E^2dyd\tau \\
  &\leq&\frac {Ct}{r}\sup_{0\leq s\leq t}\|\sqrt{\varrho_0}E\|_2^{\frac52}
  +\frac{Ct^{\frac78}}{r^{\frac34}}\sup_{0\leq s\leq t}\|\sqrt{\varrho_0}E\|_2^2
  \left(\int_0^t\|E_y\|_2^2d\tau\right)^{\frac18},
\end{eqnarray*}
that $I_2\rightarrow0$, as $r\rightarrow\infty$. Similarly, it follows from (\ref{UnifE}), $(\sqrt{\varrho_0}E, vv_y, \vartheta_y, E_y)\in L^2(0,T; L^2(\mathbb R))$, and
\begin{eqnarray*}
  I_3&:=&\frac1r\int_0^t\int_{r\leq|y|\leq2r}(|v||v_y|+|\vartheta_y|)Edyd\tau \\
  &\leq&\frac Cr\int_0^t\int_{r\leq|y|\leq2r}(|v||v_y|+|\vartheta_y|)\left(\|\sqrt{\varrho_0}E\|_2 +\|E_y\|_2^{\frac12}r^{\frac12}\right)dyd\tau\\
  &\leq&\frac Cr\int_0^t\left[\int_{r\leq|y|\leq2r}(v^2v_y^2+\vartheta_y^2)dy\right]^{\frac12}r^{\frac12}
  \left(\|\sqrt{\varrho_0}E\|_2 +\|E_y\|_2^{\frac12}r^{\frac12}\right)d\tau\\
  &\leq&\frac{C}{r^{\frac12}}\left(\int_0^t\|\sqrt{\varrho_0}E\|_2^2d\tau\right)^{\frac12}
  \left[\int_0^t\int_{r\leq|y|\leq 2r}(v^2v_y^2+\vartheta_y^2)dyd\tau\right]^{\frac12}\\
  &&+Ct^{\frac14}\left[\int_0^t\int_{r\leq|y|\leq2r}(v^2v_y^2+\vartheta_y^2)dyd\tau
  \right]^{\frac12}\left(\int_0^t\|E_y\|_2^2d\tau\right)^{\frac14},
\end{eqnarray*}
that $I_3\rightarrow0$, as $r\rightarrow\infty$.

Thus, taking the limit as $r\uparrow\infty$ in (\ref{L2-4}) gives
\begin{eqnarray}
  \bigg(\int_\mathbb R\varrho_0E^2 dy \bigg)(t)+\int_0^t\int_\mathbb R\frac{1}{J}
  \left(v_y^2+E_y^2+v^2v_y^2+\vartheta_y^2\right)dyd\tau \nonumber\\
  \leq C\left(1+\|\sqrt{\varrho_0}E_0\|_2^2+\int_0^t\int_\mathbb R(\varrho_0E^2+\varrho_0^2\vartheta^2v^2) dyd\tau\right).\label{L2-5}
\end{eqnarray}
By Proposition \ref{PropBasic}, one has
\begin{eqnarray*}
\int_0^t\int_\mathbb R \varrho_0^2\vartheta^2v^2dyd\tau
\leq\int_0^t\|\sqrt{\varrho_0}v\|_2^2\|
\sqrt{\varrho_0}\vartheta\|_\infty^2d\tau\leq C\int_0^t\|
\sqrt{\varrho_0}\vartheta\|_\infty^2d\tau.
\end{eqnarray*}
Therefore, it follows from (\ref{L2-5}) that
\begin{eqnarray}
  &&\|\sqrt{\varrho_0}E\|_2^2(t)
  +\int_0^t\int_\mathbb R\frac{1}{J}
  \left(v_y^2+E_y^2+v^2v_y^2+\vartheta_y^2\right)dyd\tau\nonumber\\
  &\leq& C\left(1+\|\sqrt{\varrho_0}E_0\|_2^2
  + \int_0^t\|\sqrt{\varrho_0}E\|_2^2d\tau\right)+A_2\int_0^t\|
\sqrt{\varrho_0}\vartheta\|_\infty^2d\tau,\label{EstL2-7}
\end{eqnarray}
with a positive constant $A_2$.

It remains to estimate $\|\sqrt{\varrho_0}\vartheta\|_\infty^2$ in (\ref{EstL2-7}).
Note that (\ref{HSLOW}) implies $|\varrho_0'|\leq \sqrt{\bar\varrho}K_1\varrho_0$. One can apply Lemma \ref{lem},
with $\omega=\varrho_0, f=\vartheta,$ and $\eta=J$, to obtain
\begin{equation}
  \|\sqrt{\varrho_0}\vartheta\|_\infty^2\leq
2\sqrt{\bar\varrho}K_1\|\sqrt{\varrho_0}\vartheta\|_2^2+8\|\varrho_0\|_\infty^{\frac13}
\|\varrho_0\vartheta\|_1^{\frac23}\left\|\frac{\vartheta_y}{\sqrt J}
\right\|_2^{\frac43}\|J\|_\infty^{\frac23}.\label{USEHSLOW}
\end{equation}
It follows from (\ref{USEHSLOW}), Proposition \ref{PropBasic}, and Corollary \ref{CorJlower} that
\begin{eqnarray}
\|\sqrt{\varrho_0}\vartheta\|_\infty^2(t)&\leq&
C\|\sqrt{\varrho_0}E\|_2^2+C
\left\|\frac{\vartheta_y}{\sqrt J}\right\|_2^{\frac43}\|J\|_\infty^{\frac23}\nonumber\\
&\leq&\frac{1}{4A_2}\left\|\frac{\vartheta_y}{\sqrt J}\right\|_2^2+C(\|J\|_\infty^2+\|\sqrt{\varrho_0}E\|_2^2)
\nonumber\\
&\leq&\frac{1}{4A_2}\left\|\frac{\vartheta_y}{\sqrt J}\right\|_2^2+C\left(1+\|\sqrt{\varrho_0}E\|_2^2+\int_0^t
\|\sqrt{\varrho_0}\vartheta\|_\infty^2d\tau \right),\nonumber
\end{eqnarray}
and, thus,
\begin{equation}
  2A_2\int_0^t\|\sqrt{\varrho_0}\vartheta\|_\infty^2d\tau
 \leq \frac12\int_0^t\left\|\frac{\vartheta_y}{\sqrt J}\right\|_2^2d\tau
  +C+C \int_0^t \left(\|\sqrt{\varrho_0}E\|_2^2 + \int_0^\tau
\|\sqrt{\varrho_0}\vartheta\|_\infty^2ds\right)d\tau.\label{EstL2-8}
\end{equation}

Summing (\ref{EstL2-8}) with (\ref{EstL2-7}) yields
\begin{eqnarray*}
  \|\sqrt{\varrho_0}E\|_2^2(t)
  +\int_0^t\left[\int_\mathbb R\frac{1}{J}
  \left(v_y^2+E_y^2+v^2v_y^2+\vartheta_y^2\right)dy+\|\sqrt{\varrho_0}\vartheta\|_\infty^2\right] d\tau\nonumber\\
  \leq C(1+\|\sqrt{\varrho_0}E_0\|_2^2)
  +C\int_0^t\left(\|\sqrt{\varrho_0}E\|_2^2+\int_0^\tau
\|\sqrt{\varrho_0}\vartheta\|_\infty^2ds\right)d\tau.
\end{eqnarray*}
Thus the Gronwall inequality yields
\begin{align}
  \sup_{0\leq t\leq T}\|\sqrt{\varrho_0}E\|_2^2
&+\int_0^T\left[\int_\mathbb R\frac{1}{J}
  \left(v_y^2+E_y^2+v^2v_y^2+\vartheta_y^2\right)dy+\|\sqrt{\varrho_0}\vartheta\|_\infty^2\right] dt\nonumber\\
\leq&\ \  C(1+\|\sqrt{\varrho_0}E_0\|_2^2).
\label{EstL2.2}
\end{align}
This and Corollary \ref{CorJlower} show that
\begin{equation}
  \sup_{0\leq t\leq T}\|J\|_\infty^2\leq C(1+
\|\sqrt{\varrho_0}E_0\|_2^2).\label{EstL2.3}
\end{equation}
Then the conclusion follows from (\ref{EstL2.2}) and (\ref{EstL2.3}).
\end{proof}


\subsection{$H^1$ estimates}
\label{ssecH1V}
In this subsection, we establish the $L^\infty(0,T; H^1)$ type a priori estimates for $(J, v,\vartheta)$.

We start with the $H^1$ estimate of $J$.

\begin{proposition}
\label{PropEstJy}
It holds that
\begin{eqnarray*}
  \sup_{0\leq t\leq T}\left\|\frac{J_y}{\sqrt{\varrho_0}}\right\|_2 \leq C
  \left(\left\|\frac{J_0'}{\sqrt{\varrho_0}}\right\|_2+\|\sqrt{\varrho_0}E_0\|_2\right),
\end{eqnarray*}
for a positive constant $C$ depending only on $\mu,\kappa,c_v,R, T, K, \bar\varrho, \underline J, \bar J, \|\varrho_0\|_1$, and $\mathscr E_0$.
\end{proposition}

\begin{proof}
Recalling Proposition \ref{PropIdJ} and the following estimate on $B$ obtained in the proof of Corollary \ref{CorJlower}
\begin{equation}\label{bdB}
\exp\left\{-\frac{2\sqrt2}{\mu}\sqrt{\|\varrho_0\|_1\mathcal E_0}\right\}\leq B
\leq \exp\left\{\frac{2\sqrt2}{\mu}\sqrt{\|\varrho_0\|_1\mathcal E_0}\right\},
\end{equation}
one can get from
\begin{eqnarray*}
  J_y=B\left[J_0'+\frac R\mu\int_0^t\left(\frac{\varrho_0'\vartheta+\varrho_0\vartheta_y}{B}
  -\frac{B_y}{B^2}\varrho_0\vartheta\right)d\tau\right]
  +B_y\left(J_0+\frac R\mu\int_0^t\frac{\varrho_0\vartheta}{B}d\tau\right),
\end{eqnarray*}
(\ref{bdB}), and (\ref{HSLOW}) that
\begin{eqnarray*}
  \left\|\frac{J_y}{\sqrt{\varrho_0}}\right\|_2&\leq&
  C\left[\left\|\frac{J_0'}{\sqrt{\varrho_0}}\right\|_2+ \int_0^t (\|\sqrt{\varrho_0}\vartheta\|_2+\|\vartheta_y\|_2+  \|B_y\|_2\|\sqrt{\varrho_0}\vartheta\|_\infty)d\tau\right]\\
  &&+\left\|\frac{B_y}{\sqrt{\varrho_0}}\right\|_2\left(\|J_0\|_\infty\
  +\int_0^t\|\sqrt{\varrho_0}\vartheta\|_\infty d\tau\right)\\
  &\leq&
  C\left[\left\|\frac{J_0'}{\sqrt{\varrho_0}}\right\|_2+ \int_0^t (\|\sqrt{\varrho_0}\vartheta\|_2+\|\vartheta_y\|_2+  \|\sqrt{\varrho_0}(v-v_0)\|_2\|\sqrt{\varrho_0}\vartheta\|_\infty)d\tau\right]\\
  &&+ \|\sqrt{\varrho_0}(v-v_0)\|_2\left(\|J_0\|_\infty\
  +\int_0^t\|\sqrt{\varrho_0}\vartheta\|_\infty d\tau\right).
\end{eqnarray*}
Then the desired conclusion follows by Proposition \ref{PropBasic} and Proposition \ref{PropEstL2-V}.
\end{proof}


Next, we carry out the estimate on the effective viscous flux $G$, which is the key to get
the corresponding $L^\infty(0,T; H^1)$ estimates of $v$ and $\vartheta$.

For simplicity of presentations, the proofs of Proposition \ref{PropEstG} and Proposition \ref{PropEsttheta} in this subsection, as well as
the uniqueness part of Theorem \ref{ThmGloUni-v} in the next one, will be given ``formally". However,
similar to the proof of Proposition \ref{PropEstL2-V}, one can
easily adopt the cut-off procedure there to justify the arguments rigorously.

\begin{proposition}
  \label{PropEstG}
It holds that
$$
  \sup_{0\leq t\leq T}\|G\|_2^2+\int_0^T\left\|\frac{G_y}{\sqrt{\varrho_0}}\right\|_2^2d\tau
\leq C,
$$
for a positive constant $C$ depending only on $\mu, \kappa, c_v, R,$ $T, K_1, \bar\varrho,\underline J,$ $\bar J,$ $\|\varrho_0\|_1$, $\mathcal E_0,$ $\|\sqrt{\varrho_0}E_0\|_2,$ and $\|G_0\|_2$, where $G$ is defined by (\ref{ExG}), and $G_0=\frac{1}{J_0}(v_0'-R\varrho_0\vartheta_0)$.
\end{proposition}

\begin{proof}
 Taking the inner product of (\ref{EqG}) with $JG$ yields
  \begin{equation}
    \frac12\frac{d}{dt}\int_\mathbb RJG^2dy +\mu\int_\mathbb R\frac{G_y^2}{\varrho_0} dy
    \leq\frac\mu2\int_\mathbb R\frac{G_y^2}{\varrho_0} dy +C\int_\mathbb R\left(|v_y|
    G^2 +\varrho_0\vartheta_y^2 \right)dy,\label{ESTG.2}
  \end{equation}
  where Corollary \ref{CorJlower} has been used.
  By the Gagliardo-Nirenberg inequality,
  \begin{eqnarray*}
\int_\mathbb R|v_y|G^2dy&\leq&\|v_y\|_2\|G\|_4^2\leq C\|v_y\|_2\|G\|_2^{\frac32}(\|G\|_2+\|G_y\|_2)^{\frac12}\\
    &\leq&\varepsilon\left\|\frac{G_y}{\sqrt{\varrho_0}}\right\|_2^2+C_\varepsilon(1+\|v_y\|_2^2)
    \|G\|_2^2,
  \end{eqnarray*}
for any $\varepsilon>0$. Thanks to this and choosing $\varepsilon$ suitably small, one obtains from Corollary \ref{CorJlower}, Proposition \ref{PropEstL2-V}, and (\ref{ESTG.2}) that
\begin{equation*}
  \sup_{0\leq t\leq T}\|G\|_2^2+\int_0^T\left\|\frac{G_y}{\sqrt{\varrho_0}}\right\|_2^2d\tau
  \leq Ce^{C\int_0^T(1+\|v_y\|_2^2)d\tau }(1+\|G_0\|_2^2)\leq C. \label{EstG.3}
\end{equation*}
This proves the conclusion.
\end{proof}

Then, we derive the $L^\infty(0,T; H^1)$ estimate on $v$.

\begin{proposition}
  \label{PropEstH1u}
It holds that
  $$\sup_{0\leq t\leq T}\|v_y\|_2^2+\int_0^T\left(\|\sqrt{\varrho_0}v_t\|_2^2
  +\left\|\frac{v_{yy}}{\sqrt{\varrho_0}}\right\|_2^2\right)dt\leq C, $$
for a positive constant $C$ depending only on $c_v$,$R$,$\mu$,$\kappa$,$T$,$\bar\varrho$,$K_1$,$\underline J$,$\bar J$,$\|\varrho_0\|_1$,$\mathcal E_0,$ $\|\sqrt{\varrho_0}E_0\|_2$, $\|G_0\|_2$, and $\left\|\frac{J_0'}{\sqrt{\varrho_0}}\right\|_2$.
\end{proposition}

\begin{proof}
Since $v_y=\frac1\mu(JG+R\varrho_0\vartheta)$ and $\varrho_0v_t=G_y$, it follows from Corollary \ref{CorJlower}, Proposition \ref{PropEstL2-V}, and Proposition \ref{PropEstG} that
\begin{eqnarray*}
  \sup_{0\leq t\leq T}\|v_y\|_2^2+\int_0^T\|\sqrt{\varrho_0}v_t\|_2^2dt
  &=&\frac1\mu\sup_{0\leq t\leq T}\|JG+R\varrho_0\vartheta\|_2^2+\int_0^T\left\|\frac{G_y}{\sqrt{\varrho_0}}\right\|_2^2dt
  \nonumber \\
  &\leq&C\sup_{0\leq t\leq T}(\|G\|_2^2+\|\varrho_0\vartheta\|_2^2)+\int_0^T\left\|\frac{G_y}{\sqrt{\varrho_0}}\right\|_2^2dt
  \leq C.
\end{eqnarray*}
Since
\begin{equation*}
  v_{yy}
  =\frac1\mu(JG_y+J_yG
  +R\varrho_0'\vartheta+R\varrho_0\vartheta_y),
\end{equation*}
it follows from (\ref{HSLOW}), the Sobolev inequality, and
Propositions \ref{PropEstL2-V}--\ref{PropEstG} that
\begin{eqnarray*}
  \int_0^T\left\|\frac{v_{yy}}{\sqrt{\varrho_0}}\right\|_2^2dt
  &\leq&C\int_0^T\left(\left\|\left(\frac{G_y}{\sqrt{\varrho_0}},\varrho_0\vartheta,\sqrt{\varrho_0}\vartheta_y
  \right)\right\|_2^2+\left\|\frac{J_y}{\sqrt{\varrho_0}}
  \right\|_2^2\|G\|_\infty^2\right)
  dt\nonumber\\
  &\leq&C\int_0^T\left(\left\|\left(\frac{G_y}{\sqrt{\varrho_0}},\varrho_0\vartheta,\sqrt{\varrho_0}\vartheta_y
  \right)\right\|_2^2
  +\left\|\frac{J_y}{\sqrt{\varrho_0}}
  \right\|_2^2\|G\|_{H^1}^2\right)
  dt \leq C,
\end{eqnarray*}
which yields the conclusion.
\end{proof}

Finally, we give the corresponding weighted $L^\infty(0,T; H^1)$ estimates on $\vartheta$.

\begin{proposition}
  \label{PropEsttheta}
  The following estimate holds
$$
  \sup_{0\leq t\leq T}\|\sqrt{\varrho_0}\vartheta_y\|_2^2+\int_0^T \left(\|\varrho_0\vartheta_t\|_2^2+\left\|
  \left(\frac{\vartheta_y}{J}\right)_y\right\|_2^2+\|\vartheta_{yy}\|_2^2\right)dt \leq C,
$$
for a positive constant $C$ depending only on $\mu, \kappa, c_v, R, T, K_1, \bar\varrho, \underline J, \bar J,
\|\varrho_0\|_1, \mathcal E_0, \|\sqrt{\varrho_0}E_0\|_2,$ $\|G_0\|_2, \|\sqrt{\varrho_0}\vartheta_0'\|_2$, and $\left\|\frac{J_0'}{\sqrt{\varrho_0}}\right\|_2$.
\end{proposition}

\begin{proof}
Rewrite (\ref{EqTheta}) as
$c_v\varrho_0\vartheta_t-\kappa\left(\frac{\vartheta_y}{J}\right)_y
  =v_y G.$ Then,
\begin{eqnarray}
   -2c_v\kappa\int_\mathbb R\varrho_0\vartheta_t\left(\frac{\vartheta_y}{J}\right)_y
   dy+\int_\mathbb R\left(c_v^2\varrho_0^2\vartheta_t^2+\kappa\left|\left(\frac{\vartheta_y}{J}\right)_y\right|^2
  \right)dy =\int_\mathbb Rv_y^2G^2 dy. \label{Esttheta.1}
\end{eqnarray}
By direct calculations, one can get that
\begin{align}
  c_v\kappa\frac{d}{dt}&\int_\mathbb R\frac{\varrho_0}{J}\vartheta_y^2 dy+
  \int_\mathbb R\left(c_v^2\varrho_0^2\vartheta_t^2+\kappa^2\left|\left(\frac{\vartheta_y}{J}
  \right)_y\right|^2\right)dy\nonumber\\
  =&\ \ -c_v\kappa\int_\mathbb R\left(\frac{ \varrho_0}{J^2}v_y\vartheta_y^2+2\varrho_0'\vartheta_t\frac{\vartheta_y}{J}
    \right)dy +\int_\mathbb R v_y^2G^2 dy\nonumber \\
  \leq&\ \ \frac{c_v^2}{2}\int_\mathbb R\varrho_0^2\vartheta_t^2dy+C\int_\mathbb R(\varrho_0|v_y|\vartheta_y^2 +\varrho_0\vartheta_y^2+v_y^2G^2 )dy,\label{Esttheta.2}
\end{align}
where (\ref{HSLOW}) and Corollary \ref{CorJlower} have been used.
Then, Propositions \ref{PropEstL2-V} and \ref{PropEstG} imply
\begin{eqnarray*}
  \int_\mathbb R\varrho_0|v_y|\vartheta_y^2dy&=&\frac1\mu\int_\mathbb R
  \varrho_0|JG+R\varrho_0\vartheta|\vartheta_y^2dy\leq
 C(\|G\|_{H^1}+\|\varrho_0\vartheta\|_\infty)\|\sqrt{\varrho_0}\vartheta_y\|_2^2,
\end{eqnarray*}
and
\begin{eqnarray*}
  \int_\mathbb Rv_y^2G^2dy&=&\frac{1}{\mu^2}\int_\mathbb R
  (JG+R\varrho_0\vartheta)^2G^2dy \\
  &\leq&C\left[\|G\|_2^3(\|G\|_2+\|G_y\|_2)+\|\varrho_0\vartheta\|_\infty^2\|G\|_2^2\right]\\
  &\leq& C(1+\|G_y\|_2+\|\varrho_0\vartheta\|_\infty^2).
\end{eqnarray*}
Substituting the above two estimates into (\ref{Esttheta.2}) gives
\begin{eqnarray*}
  c_v\kappa\frac{d}{dt} \int_\mathbb R\frac{\varrho_0}{J}\vartheta_y^2 dy+\frac12
  \int_\mathbb R\left(c_v^2\varrho_0^2\vartheta_t^2+\kappa^2\left|\left(\frac{\vartheta_y}{J}
  \right)_y\right|^2\right)dy\\
  \leq C (1+\|\varrho_0\vartheta\|_\infty^2
  +\|G\|_{H^1})(1+\|\sqrt{\varrho_0}\vartheta_y\|_2^2),
\end{eqnarray*}
which, together with Corollary \ref{CorJlower} and Propositions \ref{PropEstL2-V} and \ref{PropEstG}, implies
\begin{eqnarray}
   \sup_{0\leq t\leq T}\|\sqrt{\varrho_0}\vartheta_y\|_2^2+\int_0^T \left(\|\varrho_0\vartheta_t\|_2^2+\left\|
  \left(\frac{\vartheta_y}{J}\right)_y\right\|_2^2\right)dt\nonumber\\
   \leq e^{C\int_0^T (1+\|\varrho_0\vartheta\|_\infty^2
  +\|G\|_{H^1})dt}(1+\|\sqrt{\varrho_0}\vartheta_0'\|_2^2) \leq C. \label{Esttheta.3}
\end{eqnarray}

It remains to estimate $\vartheta_{yy}$. Direct calculations show that
\begin{equation*}
  \label{Esttheta.4}
  \vartheta_{yy}
  =J\left(\frac{\vartheta_y}{J}\right)_y+\frac{J_y}{\sqrt{\varrho_0}}\frac{\sqrt{\varrho_0}}{J}
   \vartheta_y,
\end{equation*}
and
\begin{eqnarray*}
  \left\|\varrho_0\left(\frac{\vartheta_y}{J}\right)^2\right\|_\infty
  &\leq&\int_\mathbb R\left|\partial_y\left(\varrho_0\left(\frac{\vartheta_y}{J}\right)^2\right)\right|dy
  =\int_\mathbb R\left|\varrho_0'\frac{\vartheta_y^2}{J^2}+2\varrho_0\frac{\vartheta_y}{J}
  \left(\frac{\vartheta_y}{J}\right)_y\right|dy\nonumber\\
  &\leq&C\left(\|\sqrt{\varrho_0} \vartheta_y\|_2^2 +
  \|\sqrt{\varrho_0} \vartheta_y\|_2
  \left\|\left(\frac{\vartheta_y}{J}\right)_y\right\|_2\right),
\end{eqnarray*}
which gives
\begin{equation}
  \label{USE1}
  \left\|\sqrt{\varrho_0}\frac{\vartheta_y}{J}\right\|_\infty^2\leq  C\left(1+
  \left\|\left(\frac{\vartheta_y}{J}\right)_y\right\|_2 \right),
\end{equation}
where one has used Corollary \ref{CorJlower}, (\ref{HSLOW}), and (\ref{Esttheta.3}). It follows from
(\ref{Esttheta.3}), (\ref{USE1}), and Proposition \ref{PropEstJy} that
\begin{eqnarray*}
\int_0^T\|\vartheta_{yy}\|_2^2dt&\leq& C\int_0^T
  \left(\left\|\left(\frac{\vartheta_y}{J}\right)_y\right\|_2^2+
  \left\|\frac{J_y}{\sqrt{\varrho_0}}\right\|_2^2\left\|
  \frac{\sqrt{\varrho_0}}{J}\vartheta_y\right\|_\infty^2\right)dt\nonumber\\
  &\leq&C\int_0^T
\left(1+\left\|\left(\frac{\vartheta_y}{J}\right)_y\right\|_2+\left\|\left(\frac{\vartheta_y}{J}
\right)_y\right\|_2^2\right)dt \leq C.
\end{eqnarray*}
Combining this with (\ref{Esttheta.3}) yields the desired conclusion.
\end{proof}

\subsection{Global existence and uniqueness}
\label{ssecGLO} Based on the a priori estimates in the previous subsections, we are now ready to prove the following global well-posedness.

\begin{theorem}
  \label{ThmGloUni-v}
Assume that (\ref{ass-v}), (\ref{finitemass}), and (\ref{HSLOW}) hold. Then, there is a unique global solution
$(J, v, \vartheta)$ to the problem (\ref{EqJ})--(\ref{IC}), such that, for any finite $T$,
\begin{eqnarray}
&&0<J, J^{-1}
\in L^\infty(0,T; L^\infty),\quad\vartheta\geq0,\nonumber\\
&&J_{yt}, vv_y, \frac{v_{yy}}{\sqrt{\varrho_0}}, \sqrt{\varrho_0}v_t, \vartheta_y, \left(\frac{\vartheta_y}{J}\right)_y, \vartheta_{yy}, \varrho_0 \vartheta_t\in L^2(0,T; L^2(\mathbb R)), \label{line2}\\
&&J_t,
\frac{J_y}{\sqrt{\varrho_0}}, \sqrt{\varrho_0}v,\sqrt{\varrho_0}v^2, v_y, \sqrt{\varrho_0}\vartheta, \sqrt{\varrho_0}\vartheta_y\in L^\infty(0,T; L^2(\mathbb R)),\label{line3}\\
&&J-J_0, \sqrt{\varrho_0}v, {\varrho_0}\vartheta\in C([0,T]; L^2). \label{line4}
\end{eqnarray}
\end{theorem}

\begin{proof}
We start with the uniqueness.
Let $(J_1, v_1, \vartheta_1)$ and $(J_2, v_2, \vartheta_2)$ be two solutions to problem
(\ref{EqJ})--(\ref{IC}), satisfying the regularities in the theorem.
Set
$$
(J, v,\theta)=(J_1-J_2, v_1-v_2, \vartheta_1-\vartheta_2).
$$
Then, straightforward calculations yield
\begin{eqnarray}
  J_t=v_y,\label{DJ}\\
  \varrho_0v_t-\mu\left(\frac{v_y}{J_1}\right)_y=\left(\omega_1J+\omega_2\varrho_0\vartheta\right)_y,
  \label{Dv}\\
  c_v\varrho_0\vartheta_t-\kappa\left(\frac{\vartheta_y}{J_1}\right)_y=\left(\varpi_1J\right)_y
  +\varpi_2v_y+\varpi_3J+\varpi_4\varrho_0\vartheta,\label{Dta}
\end{eqnarray}
where
\begin{eqnarray*}
  &\omega_1:=\frac{R\varrho_0\vartheta_2-\mu\partial_yv_2}
  {J_1J_2},\qquad \omega_2:=-\frac{R}{J_1},\\
  &\varpi_1:=-\kappa \frac{\partial_y\vartheta_2}{J_1J_2},\qquad\varpi_2:=\frac{1}{J_1}(\mu\partial_y(v_1+v_2) -R\varrho_0\vartheta_1),\\
  &\varpi_3:=
  \frac{\partial_y v_2}{J_1J_2}(R\varrho_0\vartheta_2-\mu\partial_yv_2),\qquad
  \varpi_4:=-R\frac{\partial_y v_2}{J_1}.
\end{eqnarray*}

Taking the inner product of (\ref{DJ}) with $J$ yields
\begin{equation}
  \label{U.1}
  \frac{d}{dt}\int_\mathbb RJ^2 dy=\int_\mathbb Rv_yJ dy \leq
  \varepsilon\int_\mathbb R\frac{v_y^2}{J_1} dy +C_\varepsilon\int_\mathbb RJ_1J^2 dy,
\end{equation}
for any positive $\varepsilon>0$. Taking the inner product of (\ref{Dv}) with $v$ leads to
\begin{eqnarray*}
  \frac12\frac{d}{dt}\int_\mathbb R\varrho_0v^2 dy +\mu\int_\mathbb R\frac{v_y^2}{J_1} dy
  &\leq& C\int_\mathbb R(|\omega_1||J|+|\omega_2|\varrho_0|\vartheta|) |v_y| dy\\
  &\leq&\frac\mu2\int_\mathbb R\frac{v_y^2}{J_1} dy+C\int_\mathbb R  (\omega_1^2J^2+\omega_2^2\varrho_0^2\vartheta^2) .
\end{eqnarray*}
Therefore,
\begin{equation}
  \frac{d}{dt}\int_\mathbb R\varrho_0v^2 dy +\mu\int_\mathbb R\frac{v_y^2}{J_1} dy
  \leq C\int_\mathbb R ( \omega_1^2+\omega_2^2)( J^2+\varrho_0^2\vartheta^2) dy.
  \label{U.2}
\end{equation}
Taking the inner product of (\ref{Dta}) with $\varrho_0\vartheta$ and using (\ref{HSLOW}), one can get
\begin{eqnarray*}
  &&\frac{c_v}{2}\frac{d}{dt}\int_\mathbb R\varrho_0^2\vartheta^2 dy +\kappa\int_\mathbb R
  \varrho_0\frac{\vartheta_y^2}{J_1} dy \\
  &\leq&C\int_\mathbb R\frac{|\vartheta_y|}{J_1} \varrho_0^{\frac32}|\vartheta| dy +C\int_\mathbb R|\varpi_1||J|
  (\varrho_0^{\frac32}|\vartheta| +\varrho_0|\vartheta_y|)dy \\
  &&+C\int_\mathbb R(|\varpi_2||v_y|+|\varpi_3||J|
  +|\varpi_4|\varrho_0|\vartheta|)\varrho_0|\vartheta| dy\\
  &\leq&\frac\kappa2\int_\mathbb R\varrho_0\frac{\vartheta_y^2}{J_1} dy+C\int_\mathbb R (\varrho_0^2\vartheta^2+|\varpi_1|^2\varrho_0J^2) dy+\varepsilon\int_\mathbb R\frac{v_y^2}{J_1} dy\\
  && +C_\varepsilon\int_\mathbb R\varpi_2^2
  \varrho_0^2\vartheta^2 dy+C\int_\mathbb R(|\varpi_3|+|\varpi_4|)(J^2+\varrho_0^2\vartheta^2) dy,
\end{eqnarray*}
which yields
\begin{eqnarray}
  c_v\frac{d}{dt}\int_\mathbb R\varrho_0^2\vartheta^2 dy
   +\kappa\int_\mathbb R\varrho_0\frac{\vartheta_y^2}{J_1} dy
  &\leq&2\varepsilon\int_\mathbb R\frac{v_y^2}{J_1} dy+C_\varepsilon\int_\mathbb R
  (1+ \varrho_0 \varpi_1 ^2+\varpi_2^2\nonumber\\
  &&+|\varpi_3|+|\varpi_4|) (J^2+\varrho_0^2\vartheta^2)dy.\label{U.3}
\end{eqnarray}

It follows from (\ref{U.1})--(\ref{U.3}) and choosing $\varepsilon$ sufficiently small that
\begin{eqnarray}
  &&\frac{d}{dt}\int_\mathbb R\left(J^2+\varrho_0v^2+c_v\varrho_0^2\vartheta^2\right) dy
  +\int_\mathbb R\frac{1}{J_1}\left(\frac\mu2 v_y^2+\kappa\varrho_0\vartheta_y^2\right)dy \nonumber\\
  &\leq&C(1+\|(\omega_1,\omega_2,\sqrt{\varrho_0}\varpi_1,
  \varpi_2)\|_\infty^2+\|(\varpi_3,\varpi_4)\|_\infty)\|(J,\sqrt{\varrho_0}v,\varrho_0\vartheta)\|_2^2.\nonumber
\end{eqnarray}
Thanks to this and that
$$
\omega_1, \omega_2, \sqrt{\varrho_0}\varpi_1, \varpi_2\in L^2(0,T; L^\infty(\mathbb R))\quad\mbox{and}\quad
\varpi_3, \varpi_4\in L^1(0,T; L^\infty(\mathbb R)),
$$
which can be easily verified by the regularities of $(J_i,v_i,\vartheta_i)$, $i=1,2$,
the uniqueness follows by the Gronwall inequality.

Next we prove the global existence. The local existence of solutions in the class stated in the theorem follows from Theorem \ref{ThmLocV} and Propositions \ref{PropEstJy}, \ref{PropEstH1u}, and \ref{PropEsttheta}. Note that the regularities $J_t\in L^\infty(0,T; L^2)$ and $J_{yt}\in L^2(0,T; L^2)$ follow directly from Proposition \ref{PropEstH1u} and equation (\ref{EqJ}), while the regularities in (\ref{line4}) follow from those in (\ref{line2})--(\ref{line3}). The global existence is then the corollary of the local existence and uniqueness and the a priori estimates obtained in Propositions \ref{PropBasic}--\ref{PropEsttheta}.
This completes the proof of Theorem \ref{ThmGloUni-v}.
\end{proof}

\section{Uniform lower bound of the entropy}
\label{SecLowEty}
In this section, we establish the uniform lower bound for the entropy. This is
proved by a De Giorgi type iteration which will be carried out for a suitably modified entropy equation.
To this end, we assume that (\ref{ass-v}), (\ref{finitemass}), and (\ref{HSLOW}) hold, and the initial entropy is bounded from below. Furthermore, we require that
\begin{equation}
  |\varrho_0''|\leq K_2\varrho_0^2,\quad\mbox{ on }\mathbb R,
  \label{HSLOW2}
\end{equation}
with any given positive constant $K_2$. Let $(\varrho, v, \vartheta)$ be the unique global solution guaranteed by Theorem \ref{ThmGloUni-v} (for this section and the next one).

Set
\begin{eqnarray}
&\underline\ell_0:=\log\left(\frac{\min\left\{1,\frac ARe^{\frac{\underline s_0}{c_v}}\right\}}{\max\left\{1,2^{\gamma-2}\right\}}\right), \label{lell0}\\
  &\underline J_T:=\inf_{(y,t)\in\mathbb R\times(0,T)}J(y,t),\quad\bar J_T:=\sup_{(y,t)\in\mathbb R\times(0,T)}J(y,t),\label{JT}\\
  &\mathcal Z_J(T):=\sup_{0\leq t\leq T}\left(\left\|\varrho_0^{-\frac12}J_y\right\|_2^2+\|\sqrt{\varrho_0}\vartheta\|_2^2
  \right),\label{ZJT}
\end{eqnarray}
where $\underline s_0:=\inf_{y\in\mathbb R}s_0(y).$

Due to (\ref{Entropy}) and that $J$ is uniformly positive, to get a uniform lower bound for $s$, it suffices to
obtain that for $\log\vartheta-(\gamma-1)\log\varrho_0$. For $\varepsilon\in(0,1)$, set
\begin{equation*}
  S_\varepsilon:=\log\vartheta_\varepsilon-(\gamma-1)\log\tilde\varrho_\varepsilon,\quad\mbox{with }
  \vartheta_\varepsilon=\vartheta+\varepsilon\mbox{ and } \tilde\varrho_\varepsilon=\varrho_0+\varepsilon^{\frac{1}{\gamma-1}}. \label{Sepsilon'}
\end{equation*}
Then, by direct calculations,
\begin{equation}
c_v\varrho_0\partial_t S_\varepsilon-\kappa\partial_y\left(\frac{\partial_yS_\varepsilon}{J}\right)
=\kappa(\gamma-1)\left(\frac1J\left(\frac{\varrho_0'}{\tilde\varrho_\varepsilon}\right)'
  -\frac{\varrho_0'J_y}{\tilde\varrho_\varepsilon J^2}\right)
-\frac{R^2}{4\mu}\frac{\varrho_0^2\vartheta^2}{J\vartheta_\varepsilon}+H_\varepsilon.
\label{EqSepsilon-1}
\end{equation}
where $H_\varepsilon=\frac{\mu}{J\vartheta_\varepsilon}\left(v_y-\frac{R}{2\mu}\varrho_0\vartheta\right)^2
+\kappa\frac{|\partial_y\vartheta_\varepsilon|^2}
{J\vartheta_\varepsilon^2}.$
Define
\begin{equation}
  \label{Sepsilon}
  s_\varepsilon:=S_\varepsilon+\underline M_Tt
\end{equation}
with
\begin{equation}
  \label{LMT}
  \underline M_T:=\frac{\kappa(\gamma-1)}{c_v\underline J_T}
  \left(K_1^2+K_2\right).
\end{equation}
Then, it follows from (\ref{EqSepsilon-1}) that
\begin{equation}
  c_v\varrho_0\partial_ts_\varepsilon-\kappa\partial_y\left(\frac{\partial_ys_\varepsilon}{J}\right)
  = -\kappa(\gamma-1)\frac{\varrho_0'J_y}{\tilde\varrho_\varepsilon J^2}-\frac{R^2}{4\mu}\frac{\varrho_0^2\vartheta^2}{J\vartheta_\varepsilon}+\widetilde H_\varepsilon, \label{EqSepsilon}
\end{equation}
where
$\widetilde H_\varepsilon=H_\varepsilon+c_v\underline M_T\varrho_0+\kappa(\gamma-1) \frac1J\left(\frac{\varrho_0'}{\tilde\varrho_\varepsilon}\right)'\geq 0.$
The nonnegativity of $\widetilde H_\varepsilon$
can be verified easily. Indeed, since $\tilde\varrho_\varepsilon>\varrho_0$, it follows from (\ref{HSLOW}) and (\ref{HSLOW2})
that
\begin{eqnarray*}
\left|\frac1J\left(\frac{\varrho_0'}{\tilde\varrho_\varepsilon}
  \right)'\right|
  \leq \frac{\varrho_0}{\underline J_T} \left(\left|\frac{\varrho_0''}{\varrho_0^2}
  \right|+\left|\frac{\varrho_0'}{\varrho_0^{\frac32}}\right|^2\right)
  \leq \frac{\varrho_0}{\underline J_T} (K_1^2+K_2)
  =\frac{c_v\underline M_T}{\kappa(\gamma-1)}\varrho_0.
\end{eqnarray*}
Thus, $\kappa(\gamma-1)\frac1J\left(\frac{\varrho_0'}{\tilde\varrho_\varepsilon}\right)'
  +c_v\underline M_T\varrho_0\geq0.$
This and $H_\varepsilon\geq0$ imply that $\widetilde H_\varepsilon\geq0$.

Now, we are going to derive an uniform lower bound for $s_\varepsilon$, independent of $\varepsilon$, which will be achieved by using a De Giorgi type iteration. To this end, as a preparation, we state the following iterative lemma whose proof is given in the Appendix.

\begin{lemma}
  \label{lemiteration}
Let $m_0\in[0,\infty)$ be given and $f$ be a nonnegative non-increasing function on $[m_0,\infty)$ satisfying
$$
f(\ell)\leq\frac{M_0(\ell+1)^\alpha}{(\ell-m)^\beta}f^\sigma(m),\quad\forall\ell>m\geq m_0,
$$
for some nonnegative constants $M_0, \alpha, \beta,$ and $\sigma$, with $0\leq\alpha<\beta$ and $\sigma>1$.
Then,
$$f(m_0+d)=0,$$
where
$$d=\left[2 f^\sigma(m_0) (m_0+M_0+2)^{\frac{2\alpha+2\beta+1}{\sigma-1}+\frac{\beta}{(\sigma-1)^2}+2\alpha+\beta+1}\right]^{\frac{1}{\beta-\alpha}}+2.
$$
\end{lemma}

\subsection{$L^2$ estimate on $s_\varepsilon$}
\label{ssecL2s}
The following $L^2$ energy inequality holds for $s_\varepsilon$.

\begin{proposition}
 \label{PropEstL2S}
Let $s_\varepsilon$ be defined as (\ref{Sepsilon}). Then, it holds that
\begin{align*}
 \sup_{0\leq t\leq T}&\|(s_\varepsilon-\ell)_-\|_2^2+\int_0^T\left\|\frac{\partial_y (s_\varepsilon-\ell)_-}{\sqrt{\varrho_0}}
\right\|_2^2dt\\
 \leq&\ \   C\int_0^T\int_\mathbb R\left(\left|\frac{J_y}{\sqrt{\varrho_0}}\right|^2+\varrho_0^2\vartheta^2
\right)\Bigg|_{\{s_\varepsilon<\ell\}}dydt,
\end{align*}
for any $\ell\leq\underline\ell_0$, where
$\underline\ell_0$ is given by (\ref{lell0}),
and $C$ is a positive constant depending only on $R,\gamma,\kappa,\mu,\underline J_T,\bar J_T, T$, and $K_1$.
\end{proposition}

\begin{proof}
For $\delta>0$, set $\varrho_\delta=\varrho_0+\delta$.
Testing (\ref{EqSepsilon}) with $-\frac{\varrho_0}{\varrho_\delta^2}(s_\varepsilon-\ell)_-\varphi_r^2$ and recalling $\widetilde H_\varepsilon\geq0$, one obtains
\begin{eqnarray}
  \frac{c_v}{2}\frac{d}{dt}\int_\mathbb R\frac{\varrho_0^{2}}{\varrho_\delta^2}|(s_\varepsilon
  -\ell)_-|^2\varphi_r^2dy+\kappa\int_\mathbb R\partial_y\left(\frac{\partial_ys_\varepsilon}{J}\right)
  \frac{\varrho_0}{\varrho_\delta^2}(s_\varepsilon-\ell)_-\varphi_r^2dy \nonumber\\
  \leq\int_\mathbb R\left(\kappa(\gamma-1)\frac{\varrho_0'J_y}{\tilde\varrho_\varepsilon J^2}+\frac{R^2}{4\mu}\frac{\varrho_0^2\vartheta^2}{J\vartheta_\varepsilon}\right)
  \frac{\varrho_0}{\varrho_\delta^2}(s_\varepsilon-\ell)_-\varphi_r^2dy.
\label{EstSe1}
\end{eqnarray}
Integration by parts and using the Cauchy inequality yield
\begin{eqnarray}
  &&\int_\mathbb R\partial_y\left(\frac{\partial_ys_\varepsilon}{J}\right)\frac{\varrho_0}{\varrho_\delta^2}
  (s_\varepsilon-\ell)_-\varphi_r^2dy \nonumber\\
  &=&\int_\mathbb R\frac{\varrho_0}{J\varrho_\delta^2}|\partial_y(s_\varepsilon-\ell)_-|^2\varphi_r^2 dy
  +2\int_\mathbb R\frac{\varrho_0}{J\varrho_\delta^2}\partial_y(s_\varepsilon-\ell)_-
  (s_\varepsilon-\ell)_-\varphi_r\varphi_r'dy \nonumber\\
  &&+\int_\mathbb R\frac{\partial_y(s_\varepsilon-\ell)_-}{J}\frac{\varrho_0 \varrho_0'}{
  \varrho_\delta^ 2}\left(\frac{1}{\varrho_0}-\frac{ 2}{\varrho_\delta}\right)(s_\varepsilon-
  \ell)_-\varphi_r^2dy\nonumber \\
  &\geq&
  \frac34\int_\mathbb R\frac{\varrho_0 }{J\varrho_\delta^ 2}|\partial_y(s_\varepsilon-\ell)_-|^2\varphi_r^2 dy
  -C\int_\mathbb R\frac{\varrho_0 }{\varrho_\delta^ 2}|(s_\varepsilon-\ell)_-|^2|\varphi_r'|^2 dy\nonumber\\
  &&-C\int_\mathbb R\frac{\varrho_0^{2}}{\varrho_\delta^ 2}|
  (s_\varepsilon-\ell)_-|^2\varphi_r^2 dy,
  \label{EstSe1-1}
\end{eqnarray}
where (\ref{HSLOW}) has been used.
Note that
$\tilde\varrho_\varepsilon>\varrho_0$ and $\frac{\vartheta}{\vartheta_\varepsilon}\leq1$.
It follows from (\ref{HSLOW}) that
\begin{eqnarray}
&&\int_\mathbb R
  \left(\kappa(\gamma-1)\frac{\varrho_0'J_y}{\tilde\varrho_\varepsilon J^2}+\frac{R^2}{4\mu}\frac{\varrho_0^2\vartheta^2}{J\vartheta_\varepsilon}\right)
  \frac{\varrho_0 }{\varrho_\delta^ 2}(s_\varepsilon-\ell)_-\varphi_r^2dy\nonumber\\
&\leq&C\int_\mathbb R
  \left(\left|\frac{J_y}{\sqrt{\varrho_0}}\right|+\varrho_0\vartheta\right)
  \frac{\varrho_0^{2}}{\varrho_\delta^ 2}(s_\varepsilon-\ell)_-\varphi_r^2dy\nonumber\\
&\leq&C\int_\mathbb R\frac{\varrho_0^{2}}{ \varrho_\delta^ 2}
\left[\left(\left|\frac{J_y}{\sqrt{\varrho_0}}\right|^2+\varrho_0^2\vartheta^2\right)
\bigg|_{\{s_\varepsilon<\ell\}}+|(s_\varepsilon-\ell)_-|^2\right]{\varphi_r^2}dy.
\label{EstSe1-3}
\end{eqnarray}
Substituting (\ref{EstSe1-1}) and (\ref{EstSe1-3}) into (\ref{EstSe1}) and applying the Gronwall inequality yield
\begin{eqnarray}
  &&\left(\int_\mathbb R\frac{\varrho_0^{2}}{\varrho_\delta^ 2}
  |(s_\varepsilon-\ell)_-|^2\varphi_r^2dy\right)(t)+\int_0^t\int_\mathbb R\frac{\varrho_0 }{\varrho_\delta^ 2}|\partial_y(s_\varepsilon-\ell)_-|^2\varphi_r^2dy \nonumber\\
  &\leq&Ce^{Ct}\Bigg(\int_\mathbb R\frac{\varrho_0^{2}}{\varrho_\delta^ 2}
  |(s_\varepsilon-\ell)_-|^2\varphi_r^2dy\bigg|_{t=0}+\int_0^t\int_\mathbb R
  \frac{\varrho_0 }{\varrho_\delta^ 2}|(s_\varepsilon-\ell)_-|^2|\varphi_r'|^2
  dyd\tau\Bigg)\nonumber\\
  &&+Ce^{Ct}\int_0^t\int_\mathbb R
  \frac{\varrho_0^{2}}{ \varrho_\delta^ 2}\left(\left.\left|\frac{J_y}{\sqrt{\varrho_0}}\right|^2+\varrho_0^2\vartheta^2
  \right)\right|_{\{s_\varepsilon<\ell\}} {\varphi_r^2}dyd\tau.
  \label{EstSe3}
\end{eqnarray}

Due to the definition of $s_\varepsilon$, it holds that
\begin{eqnarray*}
  s_\varepsilon
  &\geq&\log\varepsilon-(\gamma-1)\log\left(\|\varrho_0\|_\infty
  +\varepsilon^{\frac{1}{\gamma-1}}\right),
\end{eqnarray*}
and, thus,
\begin{equation*}
  0\leq(s_\varepsilon-\ell)_-\leq\max\left\{0,\ell-\log\varepsilon+(\gamma-1)\log\left(\|\varrho_0\|_\infty
  +\varepsilon^{\frac{1}{\gamma-1}}\right)\right\}:=A_{\ell,\varepsilon}.
\end{equation*}
Therefore,
\begin{eqnarray}
  \int_0^t\int_\mathbb R\frac{\varrho_0 }
  {\varrho_\delta^ 2}|(s_\varepsilon-\ell)_-|^2|\varphi_r'|^2dyd\tau
   \leq  CA_{\ell,\varepsilon}^2\delta^{- 2}t\int_{r\leq|y|\leq 2r}\frac{\varrho_0 }{r^2}dy \nonumber\\
    \leq  CA_{\ell,\varepsilon}^2\delta^{- 2}\|\varrho_0\|_\infty  tr^{-1}\rightarrow0, \quad\mbox{as }r\rightarrow\infty.\label{EstSe4}
\end{eqnarray}
Thanks to (\ref{EstSe4}), one can take the limits $r\uparrow\infty$ first and then $\delta\downarrow0$ in (\ref{EstSe3})
to get
\begin{align}
  &\bigg(\int_\mathbb R|(s_\varepsilon-\ell)_-|^2dy\bigg)(t)+\int_0^t\int_\mathbb R\frac{|\partial_y(s_\varepsilon-\ell)_-|^2}{\varrho_0}dy \nonumber\\
  \leq&\ \ e^{Ct}\left[\int_0^t\int_\mathbb R\left(\left.\left|\frac{J_y}{\sqrt{\varrho_0}}\right|^2+\varrho_0^2\vartheta^2
  \right)\right|_{\{s_\varepsilon<\ell\}}dyd\tau
  +\int_\mathbb R
  |(s_\varepsilon-\ell)_-|^2dy\bigg|_{t=0}\right],\label{EstSe5}
\end{align}
where the monotone convergence theorem has been used.

Using the elementary inequalities that for any $a,b>0$, $(a+b)^\sigma\leq 2^{\sigma-1}(a^\sigma+b^\sigma)$,
if $\sigma\geq1$, and $(a+b)^\sigma\leq(a^{\sigma}+b^{\sigma})$, if $0<\sigma<1$, one can deduce easily
\begin{eqnarray*}
  \left(\varrho_0+\varepsilon^{\frac{1}{\gamma-1}}
  \right)^{\gamma-1}\leq\max\left\{1,2^{\gamma-2}\right\}(\varrho_0^{\gamma-1}+\varepsilon).
\end{eqnarray*}
On the other hand,
\begin{equation*}
  \vartheta_0+\varepsilon = \frac ARe^{\frac{s_0}{c_v}}\varrho_0^{\gamma-1}+\varepsilon\geq
  \frac ARe^{\frac{\underline s_0}{c_v}}\varrho_0^{\gamma-1}+\varepsilon
  \geq \min\left\{1,\frac ARe^{\frac{\underline s_0}{c_v}}\right\}\left(\varrho_0^{\gamma-1}+\varepsilon\right).
\end{equation*}
Therefore, recalling (\ref{lell0}), one has
\begin{eqnarray*}
  s_\varepsilon\Big|_{t=0}
  &=&\log\left(\frac{\vartheta_0+\varepsilon}{\left(\varrho_0+\varepsilon^{\frac{1}{\gamma-1}}\right)
  ^{\gamma-1}}\right)\geq\log\left(\frac{\min\left\{1,\frac ARe^{\frac{\underline s_0}{c_v}}\right\}}{\max\left\{1,2^{\gamma-2}\right\}}\right)=\underline\ell_0,
\end{eqnarray*}
and, consequently,
\begin{equation}
  (s_\varepsilon-\ell)_-\Big|_{t=0}\equiv0, \quad\forall \ell\leq\underline\ell_0. \label{Se=0}
\end{equation}

Combining (\ref{EstSe5}) with (\ref{Se=0}) yields the conclusion.
\end{proof}

As a straightforward corollary of Proposition \ref{PropEstL2S}, we have the following:

\begin{corollary}
  \label{CorEstSL2}
Let $\underline\ell_0, \mathcal Z_J$ and $s_\varepsilon$ be defined by (\ref{lell0}), (\ref{ZJT}), and (\ref{Sepsilon}),  respectively. Then, for any $\ell\leq\underline\ell_0$, it holds that
\begin{equation*}
  \sup_{0\leq t\leq T}\|(s_\varepsilon-\ell)_-\|_2^2+\int_0^T\left\|\frac{\partial_y(s_\varepsilon-\ell)_-}{\sqrt{\varrho_0}}
\right\|_2^2dt
  \leq C\mathcal Z_J(T),
\end{equation*}
where $C$ is a positive constant depending only on $R,\gamma,\kappa,\mu,\underline J_T, \bar J_T, T$, and $K_1$.
\end{corollary}

\subsection{The De Giorgi iteration for $s_\varepsilon$}
\label{ssecDGs}
The De Giorgi iteration for $s_\varepsilon$ is stated in the following proposition.

\begin{proposition}
\label{PropDGs}
Let $\underline\ell_0, \mathcal Z_J$ and $s_\varepsilon$ be defined by (\ref{lell0}), (\ref{ZJT}), and (\ref{Sepsilon}), respectively,
and denote
  \begin{eqnarray*}
q_\ell&=& \sup_{0\leq t\leq T}\|(s_\varepsilon-\ell)_-\|_2^2+\int_0^T\left\|
  \frac{\partial_y(s_\varepsilon-\ell)_-}{\sqrt{\varrho_0}}
\right\|_2^2dt.
\end{eqnarray*}
Then, it holds that
\begin{equation*}
  q_\ell\leq\frac{C\mathcal Z_J(T)}{(m-\ell)^4}q_m^2,
  \qquad \mbox{for any }-\infty<\ell<m\leq\underline\ell_0,
\end{equation*}
with a positive constant $C$ depending only on $R,\gamma,\kappa,\mu,\bar\varrho,\underline J_T,\bar J_T,T$, and $K_1$.
\end{proposition}

\begin{proof}
For any $\ell\leq\underline\ell_0$, Corollary \ref{CorEstSL2} implies that
$$
(s_\varepsilon-\ell)_-\in L^\infty(0,T; L^2(\mathbb R))\cap L^2(0,T; H^1(\mathbb R)),
$$
and Proposition \ref{PropEstL2S} shows that
\begin{eqnarray}
q_\ell
\leq C\int_0^T\int_\mathbb R\left(\left.\left|\frac{J_y}{\sqrt{\varrho_0}}\right|^2+\varrho_0^2\vartheta^2
\right)\right|_{\{s_\varepsilon<\ell\}}dydt. \label{EstSe6}
\end{eqnarray}

Let $-\infty<\ell<m\leq\underline\ell_0$. Then, it is clear that
$$
1<\frac{(s_\varepsilon(y,t)-m)_-}{m-\ell},\quad\mbox{ for any }(y,t)\text{ such that }s_\varepsilon(y,t)<\ell .
$$
It follows from this, (\ref{HSLOW}), (\ref{ZJT}), and the Gagliardo-Nirenberg inequality that
\begin{eqnarray}
  &&\int_0^T\int_\mathbb R\left(\left.\left|\frac{J_y}{\sqrt{\varrho_0}}\right|^2+\varrho_0^2\vartheta^2
\right)\right|_{\{s_\varepsilon<\ell\}}dydt\nonumber\\
  &\leq&\frac{1}{(m-\ell)^4}\int_0^T\int_\mathbb R
  \left(\left|\frac{J_y}{\sqrt{\varrho_0}}\right|^2+\varrho_0^2\vartheta^2
\right)\left|(s_\varepsilon-m)_-\right|^{4}dydt\nonumber\\
  &\leq&\frac{C}{(m-\ell)^4}\int_0^T\left(\left\|\frac{J_y}{\sqrt{\varrho_0}} \right\|_2^2+\|\sqrt{\varrho_0}\vartheta\|_2^2
  \right) \|(s_\varepsilon-m)_-\|_\infty^4 dt\nonumber\\
  &\leq&\frac{C \mathcal Z_J(T)}{(m-\ell)^4}
\int_0^T\|(s_\varepsilon-m)_-\|_2^2\|\partial_y(s_\varepsilon-m)_-\|_2^2 dt\nonumber\\
  &\leq&\frac{C \mathcal Z_J(T)}{(m-\ell)^4}
  \sup_{0\leq t\leq T}\|(s_\varepsilon-m)_-\|_2^2\int_0^T\left\|\frac{\partial_y(s_\varepsilon-m)_-}
  {\sqrt{\varrho_0}}
  \right\|_2^2 dt. \label{EstSe7}
\end{eqnarray}

Combining (\ref{EstSe6}) and (\ref{EstSe7}) yields the conclusion.
\end{proof}

\subsection{Lower bound of the entropy}
\label{ssecSlow}
As a corollary of Proposition \ref{PropDGs} and Lemma \ref{lemiteration},
we have the following uniform lower bound of the entropy.

\begin{theorem}
  \label{ThmSlowBd}
Assume that (\ref{ass-v}), (\ref{finitemass}), (\ref{HSLOW}), and (\ref{HSLOW2}) hold, and that the initial entropy is bonded from below. Let $\underline\ell_0$, $\underline J_T, \mathcal Z_J(T),$ and $\underline M_T$ be given by (\ref{lell0}), (\ref{JT}), (\ref{ZJT}), and (\ref{LMT}), respectively. Then, the unique global solution obtained in Theorem \ref{ThmGloUni-v} satisfies
\begin{eqnarray*}
\inf_{(y,t)\in\mathbb R\times(0,T)}s&\geq&c_v\left[\log\frac RA+\underline\ell_0+(\gamma-1)\log \underline J_T-\underline M_TT-C\Big(\mathcal Z_J(T)+1-\underline \ell_0\Big)^5\right],
\end{eqnarray*}
for any positive time $T$, with a positive constant $C$ depending only on $R$, $\gamma$, $\kappa$, $\mu$, $\underline J_T$, $\bar J_T$, $T$, and $K_1$.
\end{theorem}

\begin{proof}
Set $m_0=-\underline\ell_0\geq0$,
and define $f(\ell):=q_{-\ell}$, for $\ell\geq m_0$, with $q_\ell$ given in Proposition \ref{PropDGs}.
Then, $f$ is nonnegative and non-increasing on $[m_0,\infty)$. It follows from Proposition \ref{PropDGs} that
$$
f(\ell)=q_{-\ell}\leq \frac{C \mathcal Z_J(T) }{(\ell-m)^4}f^2(m),\quad \forall\ell>m\geq m_0.
$$
Applying Lemma \ref{lemiteration}, with $M_0=C \mathcal Z_J(T), \alpha=0, \beta=4,$ and $\sigma=2,$ one can get
\begin{equation}
f(m_0+d_0)=q_{-(m_0+d_0)}=q_{\underline\ell_0-d_0}=0, \label{SLOW1}
\end{equation}
where $d_0=\left[2q_{\underline\ell_0}^2\big(-\ell_0+C \mathcal Z_J(T) +2\big)^{18}\right]^{\frac14}+2.$
Thus,
$$
(s_\varepsilon-(\underline\ell_0-d_0))_-=0, \quad \mbox{ on }\mathbb R\times(0,T),
$$
which, due to the definition of $s_\varepsilon$, implies that
$$
\vartheta+\varepsilon\geq e^{\ell_0-d_0-\underline M_TT}\left(\varrho_0+\varepsilon^{\frac1{\gamma-1}}\right)^{\gamma-1}.
$$
This, passing limit $\varepsilon\rightarrow0$, shows that
$\vartheta \geq e^{\ell_0-d_0-\underline M_TT} \varrho_0 ^{\gamma-1}.$
Therefore,
\begin{eqnarray}
s
&=&c_v\left(\log\frac RA+\log\vartheta-(\gamma-1)\log\varrho_0+(\gamma-1)\log J\right)\nonumber\\
&\geq&c_v\left(\log\frac RA+\underline\ell_0-d_0-\underline M_TT+(\gamma-1)\log \underline J_T\right),\label{20181221-1}
\end{eqnarray}
for any $(y,t)\in\mathbb R\times(0,T)$.
Corollary \ref{CorEstSL2} and the expression of $d_0$ imply that $d_0\leq C(\mathcal Z_J(T)+1-\ell_0)^5,$ which, together with (\ref{20181221-1}), leads to the conclusion.
\end{proof}

\section{Uniform upper bound of the entropy}
\label{SecUppEty}
This section is devoted to deriving the uniform upper bound for the entropy. Due to the degeneracy of equations (\ref{EqV})--(\ref{EqTheta}) at the far fields, some
singular type estimates on $(v,\vartheta, G)$ will be needed, which require some additional compatibility conditions on the initial data. Indeed,
in addition to  (\ref{ass-v}), (\ref{finitemass}), (\ref{HSLOW}), and (\ref{HSLOW2}), used in Theorem \ref{ThmSlowBd}, we assume further that the initial entropy is bounded from above, and
\begin{equation}
  \varrho_0^{\frac{1-\gamma}{2}}v_0, \varrho_0^{1-\frac\gamma2}\vartheta_0, \varrho_0^{-\frac\gamma2}G_0\in L^2(\mathbb R),\label{ass-singular}
\end{equation}
where $G_0=\mu\frac{v_0'}{J_0}-R\frac{\varrho_0}{J_0}\vartheta_0$.

All the notations in Section \ref{SecLowEty} will be adopted in this section. Furthermore, set
\begin{equation*}
\bar\ell_0:=\frac ARe^{\frac{\bar s_0}{c_v}}, \label{uell0}
\end{equation*}
where
$\bar s_0:=\sup_{y\in\mathbb R}s_0(y),$ and, for any positive time $T$,
\begin{eqnarray}
  \mathcal Z_\vartheta(T)&:=&\sup_{0\leq t\leq T}\|\varrho_0^{1-\frac\gamma2}\vartheta\|_2^2+
  \int_0^T\|\varrho_0^{\frac{1-\gamma}{2}}\vartheta\|_2^2dt,\label{ZTT}\\
  \mathcal Z_G(T)&:=&\sup_{0\leq t\leq T}\|\varrho_0^{-\frac\gamma2}G\|_2^2+
  \int_0^T\|\varrho_0^{-\frac{\gamma+1}{2}}G\|_2^2dt.\label{ZGT}
\end{eqnarray}

The following lemma holds.

\begin{lemma}\label{lemwineq} Let $\sigma\not=0$ and (\ref{HSLOW}) hold. Then, it holds that
  \begin{equation*}
  \|\varrho_0^\sigma f\|_q\leq C\|\varrho_0\|_\infty^{\frac14-\frac{1}{2q}}\left(\|\varrho_0^{\sigma} f\|_2+\|\varrho_0^\sigma f\|_2^{\frac12+\frac1q}\|\varrho_0^{\sigma-\frac12} \partial_yf\|_2^{\frac12-\frac1q} \right),
  \quad2\leq q\leq\infty,
  \end{equation*}
for any $f$ with $\varrho_0^\sigma f\in H^1(\mathbb R)$, where positive constant $C$ depends only on $\sigma, q,$ and $K_1$.
\end{lemma}

\begin{proof}
It follows from the Gagliardo-Nirenberg inequality that
  \begin{eqnarray*}
  \|\varrho_0^\sigma f\|_q
  &\leq&  C\|\varrho_0^\sigma f\|_2^{\frac12+\frac1q}\left(\|\varrho_0^\sigma \partial_yf\|_2+\|\varrho_0^{\sigma-1}\varrho_0' f\|_2\right)^{\frac12-\frac1q}\\
  &\leq&  C\|\varrho_0^\sigma f\|_2^{\frac12+\frac1q}\left(\|\varrho_0^\sigma \partial_yf\|_2+\|\varrho_0^{\sigma+\frac12} f\|_2\right)^{\frac12-\frac1q}\\
  &\leq&  C\|\varrho_0\|_\infty^{\frac14-\frac{1}{2q}}\left(\|\varrho_0^{\sigma} f\|_2+\|\varrho_0^\sigma f\|_2^{\frac12+\frac1q}\|\varrho_0^{\sigma-\frac12} \partial_yf\|_2^{\frac12-\frac1q} \right),
  \end{eqnarray*}
which yields the conclusion.
\end{proof}

As mentioned already in the Introduction, the uniform upper bound for $s$ is achieved by applying
a modified De Giorgi iteration
to the temperature equation rather than to the entropy equation itself. As preparations, a series of singular
energy estimates will be carried out in the following three subsections. These estimates will
be proven in a brief way to make the ideas clear. However, as indicated in the proof of Proposition \ref{PropEstL2S}, one can adopt similar cut-off and approximations there to justify the arguments rigorously.
In particular, one can choose $\frac{\varrho_0}{\varrho_\delta^{\gamma+1}}v\varphi_r^2$ and $\frac{\varrho_0^2}{\varrho_\delta^{\gamma+1}}\vartheta\varphi_r^2$, $\frac{\varrho_0}{\varrho_\delta^{\gamma+1}}J G\varphi_r^2$, and $\frac{\varrho_0}{\varrho_\delta^{2\gamma}}
(\vartheta_\ell)_+\varphi_r^2$, respectively, as testing functions in Propositions \ref{PropWestvtheta}, \ref{PropWestG}, and \ref{PropWL2EstThetaell}, and pass the limits $r\uparrow\infty$ and $\delta\downarrow0$ to give the rigorous proofs.

\subsection{Singular weighted estimates on $(v, \vartheta)$}
\label{subsecsinvt}
\begin{proposition}
\label{PropWestvtheta}
It holds that
  \begin{eqnarray*}
  \sup_{0\leq t\leq T}\bigg(\big\|\varrho_0^{\frac{1-\gamma}{2}}v\big\|_2^2+\big\|
  \varrho_0^{1-\frac{\gamma}{2}}\vartheta\big\|_2^2\bigg) +\int_0^T\left(
  \big\|\varrho_0^{-\frac{\gamma}{2}}v_y\big\|_2^2+\big\|\varrho_0^{\frac{1-\gamma}{2}}
  \vartheta_y\big\|_2^2\right)dt \nonumber\\
  \leq C\left(\big\|\varrho_0^{\frac{1-\gamma}{2}}v_0\big\|_2^2+\big\|
  \varrho_0^{1-\frac{\gamma}{2}}\vartheta_0\big\|_2^2\right)
  e^{C\int_0^T\|v_y\|_2^4dt},
\end{eqnarray*}
for a positive constant $C$ depending only on $\mu, \kappa, \gamma, R, \bar\varrho, K_1, T, \underline J_T,$ and $\bar J_T$.
\end{proposition}

\begin{proof}
Taking the inner product of (\ref{EqV}) with $\frac{v}{\varrho_0^\gamma}$ leads to
\begin{equation}
  \frac12\frac{d}{dt}\|\varrho_0^\frac{1-\gamma}{2}v\|_2^2+\mu\int_\mathbb R\frac{v_y}{J}\partial_y\left(\frac{v}{\varrho_0^\gamma}\right)dy=R\int_\mathbb R\frac{\varrho_0\vartheta}{J}\partial_y\left(\frac{v}{\varrho_0^\gamma}\right)dy. \label{Westv0}
\end{equation}
Direct estimates give
\begin{equation}
  \int_\mathbb R\frac{v_y}{J}\partial_y\left(\frac{v}{\varrho_0^\gamma}\right)dy
  \geq\frac34\int_\mathbb R\frac{v_y^2}{J\varrho_0^\gamma}dy
  -C\int_\mathbb R\frac{v^2}{J\varrho_0^\gamma}\left|\frac{\varrho_0'}{\varrho_0}\right|^2dy,
  \label{Westvdiff}
\end{equation}
and
\begin{eqnarray}
  \int_\mathbb R\frac{\varrho_0\vartheta}{J}\partial_y\left(\frac{v}{\varrho_0^\gamma}\right)dy
  &\leq&\frac{\mu}{4R}\int_\mathbb R\frac{v_y^2}{J\varrho_0^\gamma}dy +C\int_\mathbb R\frac{1}{J\varrho_0^\gamma}\left(\varrho_0^2\vartheta^2+v^2\left|\frac{\varrho_0'}{\varrho_0}\right|^2\right)dy.
  \label{Westvother}
\end{eqnarray}
It follows from (\ref{HSLOW}) and (\ref{Westv0})--(\ref{Westvdiff}) that
\begin{eqnarray}
  &&\frac{d}{dt}\|\varrho_0^\frac{1-\gamma}{2}v\|_2^2+\frac{\mu}{\bar J_T}\|\varrho_0^{-\frac\gamma2}v_y\|_2^2\leq C(\|\varrho_0^{\frac{1-\gamma}{2}}v\|_2^2+\|\varrho_0^{1-\frac\gamma2}\vartheta\|_2^2).\label{Westv}
\end{eqnarray}
Next, taking the inner product of (\ref{EqTheta}) with $\frac{\vartheta}{\varrho_0^{\gamma-1}}$ and estimating as for (\ref{Westvdiff}), one can get from (\ref{HSLOW}) that
\begin{equation}
  c_v\frac{d}{dt} \|\varrho_0^{1-\frac\gamma2}\vartheta\|_2^2+\frac{\kappa}{\bar J_T}\|\varrho_0^{\frac{1-\gamma}{2}}\vartheta_y\|_2^2
  \leq C\int_\mathbb R\left(\frac{\vartheta^2}{J\varrho_0^{\gamma-2}}+\frac{|v_y|\vartheta^2}{J\varrho_0^{\gamma-2}}
  +\frac{v_y^2\vartheta}{J\varrho_0^{\gamma-1}}\right)dy.\label{Westtheta}
\end{equation}
Summing (\ref{Westv}) with (\ref{Westtheta}) leads to
\begin{eqnarray}
   &&\frac{d}{dt}(\|\varrho_0^\frac{1-\gamma}{2}v\|_2^2+c_v \|\varrho_0^{1-\frac\gamma2}\vartheta\|_2^2)
  +\frac{1}{\bar J_T}(\mu\|\varrho_0^{-\frac\gamma2}v_y\|_2^2+\kappa\|\varrho_0^{\frac{1-\gamma}{2}}\vartheta_y\|_2^2) \nonumber\\
   &\leq&C(\|\varrho_0^{\frac{1-\gamma}{2}}v\|_2^2+\|\varrho_0^{1-\frac\gamma2}\vartheta\|_2^2)
  +C\int_\mathbb R\left(\varrho_0^{2-\gamma}|v_y|\vartheta^2
  +\varrho_0^{1-\gamma}v_y^2\vartheta\right)dy.
  \label{Westvtheta0}
\end{eqnarray}

It follows from Lemma \ref{lemwineq} that
\begin{eqnarray}
  \int_\mathbb R\varrho_0^{2-\gamma}|v_y|\vartheta^2dy& \leq&
  \|v_y\|_2\|\varrho_0^{1-\frac\gamma2}\vartheta\|_4^2\nonumber\\
  &\leq& C\|v_y\|_2\left(\|\varrho_0^{1-\frac\gamma2}\vartheta\|_2^2+
  \|\varrho_0^{1-\frac\gamma2}\vartheta\|_2^{\frac32}\|\varrho_0^{\frac{1-\gamma}{2}}\vartheta_y\|_2^{\frac12}\right) \nonumber\\
  &\leq&\eta\|\varrho_0^{\frac{1-\gamma}{2}}\vartheta_y\|_2^2+C_\eta
  \left(\|v_y\|_2+\|v_y\|_2^{\frac43}\right)\|\varrho_0^{1-\frac\gamma2}\vartheta\|_2^2,
  \label{Westvtheta1-1}
\end{eqnarray}
and
\begin{eqnarray}
\int_\mathbb R\varrho_0^{1-\gamma}v_y^2\vartheta dy
&\leq&  \|v_y\|_2\|\varrho_0^{-\frac\gamma2}
  v_y\|_2\|\varrho_0^{1-\frac{\gamma}{2}}
  \vartheta\|_\infty\nonumber\\
  &\leq& C\|v_y\|_2\|\varrho_0^{-\frac\gamma2}
  v_y\|_2\left(\|\varrho_0^{1-\frac{\gamma}{2}}\vartheta\|_2+
  \|\varrho_0^{1-\frac{\gamma}{2}}\vartheta\|_2^{\frac12}\|\varrho_0^{\frac{1-\gamma}{2}}\vartheta_y\|_2^{\frac12} \right)\nonumber\\
  &\leq&\eta\big\|(\varrho_0^{-\frac{\gamma}{2}}
  v_y,\varrho_0^{\frac{1-\gamma}{2}}
  \vartheta_y)\big\|_2^2+C_\eta
  (\|v_y\|_2^2+\|v_y\|_2^4)\|\varrho_0^{1-\frac{\gamma}{2}}
  \vartheta\|_2^2.\label{Westvtheta1-2}
\end{eqnarray}
Substituting (\ref{Westvtheta1-1}) and (\ref{Westvtheta1-2}) into (\ref{Westvtheta0})
and choosing $\eta$ sufficiently small yield
\begin{eqnarray*}
  2\frac{d}{dt}(\|\varrho_0^\frac{1-\gamma}{2}v\|_2^2+c_v \|\varrho_0^{1-\frac\gamma2}\vartheta\|_2^2)
  +\frac{1}{\bar J_T}(\mu\|\varrho_0^{-\frac\gamma2}v_y\|_2^2+\kappa\|\varrho_0^{\frac{1-\gamma}{2}}\vartheta_y\|_2^2) \\ \leq C
  (1+\|v_y\|_2^4)(\|\varrho_0^\frac{1-\gamma}{2}v\|_2^2+ \|\varrho_0^{1-\frac\gamma2}\vartheta\|_2^2),
\end{eqnarray*}
from which, by the Gronwall inequality, the conclusion follows.
\end{proof}

\subsection{A singular weighted estimate on $G$}
\label{subsecsinG}
Based on Proposition \ref{PropWestvtheta}, one can derive the corresponding weighted a priori estimates
on $G$.

\begin{proposition}
\label{PropWestG}
It holds that
$$
\sup_{0\leq t\leq T}\big\|\varrho_0^{-\frac\gamma2}G\big\|_2^2(t)
  +\int_0^T\big\|\varrho_0^{-\frac{1+\gamma}{2}}G_y\big\|_2^2 dt
   \leq Ce^{\int_0^T\|v_y\|_2^2dt}\big\|(\varrho_0^{\frac{1-\gamma}{2}}v_0,
  \varrho_0^{1-\frac{\gamma}{2}}\vartheta_0,\varrho_0^{-\frac\gamma2}
  G_0)\big\|_2^2,
$$
for a positive constant $C$ depending only on $\mu,\kappa,\gamma,K_1,T,\underline J_T,$ and $\bar J_T$.
\end{proposition}

\begin{proof}
Taking the inner product of (\ref{EqG}) with $\frac{JG}{\varrho_0^{\gamma}}$ and using (\ref{HSLOW}), one deduces
\begin{eqnarray*}
  &&\frac12\frac{d}{dt}\int_\mathbb R\frac{JG^2}{\varrho_0^{\gamma}} dy+\mu\|\varrho_0^{-\frac{\gamma+1}{2}}G_y\|_2^2 \nonumber\\
  &\leq&C\int_\mathbb R\left[|\vartheta_y|\left(\varrho_0^{-\gamma}|G_y|+\varrho_0^{\frac12-\gamma}|G|
  \right)
  +|v_y|\varrho_0^{-\gamma}G^2+\varrho_0^{-\frac12-\gamma}|G||G_y|\right]dy\\
  &\leq&\frac \mu8\|\varrho_0^{-\frac{\gamma+1}{2}}G_y\|_2^2+C\left[\|\varrho_0^{\frac{1-\gamma}{2}}\vartheta_y\|_2^2
  +\int_\mathbb R(1+|v_y|)\varrho_0^{-\gamma}G^2dy\right].
\end{eqnarray*}
Similar to (\ref{Westvtheta1-1}), one can get
\begin{eqnarray*}
  \int_\mathbb R|v_y|\varrho_0^{-\gamma}G^2dyd\tau
  &\leq&\eta \|\varrho_0^{-\frac{\gamma+1}{2}}G_y\|_2^2+C_\eta
  \left(\|v_y\|_2+\|v_y\|_2^{\frac43}\right)\|\varrho_0^{-\frac\gamma2}G\|_2^2\nonumber\\
  &\leq&\eta \|\varrho_0^{-\frac{\gamma+1}{2}}G_y\|_2^2+C_\eta
  \left(1+\|v_y\|_2^2\right)\|\varrho_0^{-\frac\gamma2}G\|_2^2.
\end{eqnarray*}
Combining the two inequalities above and choosing $\eta$ sufficiently small yield
\begin{equation*}
  \frac{d}{dt}\int_\mathbb R\frac{JG^2}{\varrho_0^{\gamma}} dy+\mu\|\varrho_0^{-\frac{\gamma+1}{2}}G_y\|_2^2 \leq C\left[\|\varrho_0^{\frac{1-\gamma}{2}}\vartheta_y\|_2^2
  +\left(1+\|v_y\|_2^2\right)\|\varrho_0^{-\frac\gamma2}G\|_2^2\right],
\end{equation*}
which leads to the conclusion by the Gronwall inequality and Proposition \ref{PropWestvtheta}.
\end{proof}

\subsection{Higher singular weighted estimates on $\vartheta$}
\label{subsechigher}
In this subsection, we derive some estimates of $\vartheta$ with weights which are more singular
than those in Section \ref{subsecsinvt}.

Denote
\begin{equation}
  \label{Thetaell}
  \vartheta_\ell:=\vartheta-\ell\varrho_0^{\gamma-1}e^{\overline M_Tt},\quad \ell\geq\bar\ell_0,
\end{equation}
where
\begin{equation}
  \label{UMT}
  \overline M_T:=\frac{\kappa(\gamma-1)}{c_v\underline J_T}(|\gamma-2|K_1^2+K_2).
\end{equation}
Then,
\begin{equation}
\label{EqThetaell}
c_v\varrho_0\partial_t\vartheta_\ell-\kappa\partial_y\left(\frac{\partial_y\vartheta_\ell}{J}\right) =v_yG-\ell \kappa (\gamma-1)e^{\overline M_Tt}\varrho_0^{\gamma-2}\varrho_0'J^{-2}J_y
+N_\ell,
\end{equation}
where $N_\ell:=\ell e^{\overline M_Tt}(\frac{\kappa}{J}(\varrho_0^{\gamma-1})''-c_v\overline M_T\varrho_0^\gamma).$ Note
that $N_\ell\leq0$.
Indeed, since $\tfrac{\varrho_0''}{\varrho_0^2}
\leq K_2$ and $\Big|\tfrac{\varrho_0'}{\varrho_0^{3/2}}\Big|^2\leq K_1^2$, it follows from (\ref{UMT}) and direct calculations that
\begin{eqnarray*}
 N_\ell
 &=&\ell\kappa(\gamma-1) e^{\overline M_Tt}\left[\frac{1}{J}\left( \frac{\varrho_0''}{\varrho_0^2}
 +(\gamma-2)\left|\frac{\varrho_0'}{\varrho_0^{\frac32}}\right|^2\right)\varrho_0^\gamma-\frac{1}{\underline J_T}(|\gamma-2|K_1^2+K_2)\varrho_0^\gamma\right]\\
 &\leq&\ell\kappa(\gamma-1) e^{\overline M_Tt}\left[ \frac{1}{J}\left( K_2
 +|\gamma-2|K_1^2\right)\varrho_0^\gamma-\frac{1}{\underline J_T}(|\gamma-2|K_1^2+K_2)\varrho_0^\gamma\right]\\
 &=&\ell e^{\overline M_Tt}\kappa(\gamma-1) \left( K_2
 +|\gamma-2|K_1^2\right)\varrho_0^\gamma\frac{\underline J_T-J}{J\underline J_T}\leq0.
\end{eqnarray*}

The main singularly weighted estimates on $\vartheta_\ell$ are stated as follows:

\begin{proposition}
  \label{PropWL2EstThetaell}
There exists a positive constant $C$ depending only on $c_v,\kappa,\gamma,\bar\varrho,K_1,K_2,T$, $\underline J_T,$ and $\bar J_T$, such that, for any $\ell\geq\bar\ell_0$,
\begin{eqnarray*}
&&\sup_{0\leq t\leq T}\big\|\varrho_0^{1-\gamma}(\vartheta_\ell)_+\big\|_2^2
+\int_0^T\big\|\varrho_0^{\frac12-\gamma}
\partial_y(\vartheta_\ell)_+\big\|_2^2dt\\
&\leq&C\int_0^T\int_\mathbb R  \big(|\varrho_0^{-\frac\gamma2}G|^2+|\varrho_0^{1-
\frac\gamma2}\vartheta|^2+\ell|\varrho_0^{-\frac12}J_y| \big)  \varrho_0^{1-\gamma}(\vartheta_\ell)_+ dydt
\end{eqnarray*}
and
$$
 \sup_{0\leq t\leq T}\big\|\varrho_0^{1-\gamma}(\vartheta_\ell)_+\big\|_2^2
+\int_0^T\big\|\varrho_0^{\frac12-\gamma}
\partial_y(\vartheta_\ell)_+\big\|_2^2dt
 \leq C(\ell^2\mathcal Z_J+\mathcal Z_\vartheta^2+\mathcal Z_G^2).
$$
\end{proposition}

\begin{proof}
Testing (\ref{EqThetaell}) with $\varrho_0^{1-2\gamma}(\vartheta_\ell)_+$ and recalling $N_\ell\leq0$, one obtains
\begin{eqnarray}
&&\frac{c_v}{2}\frac{d}{dt}\|\varrho_0^{1-\gamma}
(\vartheta_\ell)_+\|_2^2
+\kappa\int_\mathbb R  \frac{\partial_y\vartheta_\ell}{J}  \partial_y\left(\varrho_0^{1-2\gamma}(\vartheta_\ell)_+\right)dy \nonumber\\
&\leq& \int_\mathbb R \left(v_yG-\ell \kappa (\gamma-1)e^{\overline M_Tt}\varrho_0^{\gamma-2}\varrho_0'J^{-2}J_y\right)\varrho_0^{1-2\gamma}(\vartheta_\ell)_+dy=:I. \label{WestS1}
\end{eqnarray}
Similar to (\ref{Westvdiff}), one can get by using (\ref{HSLOW}) that
\begin{eqnarray}
\int_\mathbb R  \frac{\partial_y\vartheta_\ell}{J}  \partial_y\left(\frac{(\vartheta_\ell)_+}{\varrho_0^{2\gamma-1}}\right)dy
  &\geq&\frac34\int_\mathbb R\frac{|\partial_y(\vartheta_\ell)_+|^2}{J\varrho_0^{2\gamma-1}}dy
  -C\int_\mathbb R\frac{ |(\vartheta_\ell)_+|^2}{J\varrho_0^{2\gamma-1}}
  \left|\frac{\varrho_0'}{\varrho_0} \right|^2dy\nonumber\\
  &\geq&\frac{3}{4\bar J_T}\|\varrho_0^{\frac12-\gamma}\partial_y(\vartheta_\ell)_+\|_2^2-C\|\varrho_0^{1-\gamma}
(\vartheta_\ell)_+\|_2^2.
  \label{WestS1-2}
\end{eqnarray}
Due to (\ref{HSLOW}) and $|v_y|\leq C(|G|+\varrho_0\vartheta)$,
\begin{eqnarray}
  I&\leq&C\int_\mathbb R\left[ (|G|+\varrho_0\vartheta)|G|+\ell\varrho_0^{\gamma-\frac12}|J_y|\right]
  \varrho_0^{1-2\gamma}(\vartheta_\ell)_+ dy \nonumber\\
  &\leq&C\int_\mathbb R
  (G^2+\varrho_0^2\vartheta^2+\ell\varrho_0^{\gamma-\frac12}|J_y|)
  \varrho_0^{1-2\gamma}(\vartheta_\ell)_+ dy.\label{WestS1-3}
\end{eqnarray}
Substituting (\ref{WestS1-2}) and (\ref{WestS1-3}) into (\ref{WestS1}) and noticing that
$(\vartheta_0-\ell\varrho_0^{\gamma-1})_+\equiv0$ for any $\ell\geq\bar\ell_0$, one obtains the first
conclusion by the Gronwall inequality.

By the Cauchy inequality, one can derive easily from the first conclusion that
\begin{eqnarray*}
&&\sup_{0\leq t\leq T}\big\|\varrho_0^{1-\gamma}(\vartheta_\ell)_+\big\|_2^2
+\int_0^T\big\|\varrho_0^{\frac12-\gamma}
\partial_y(\vartheta_\ell)_+\big\|_2^2dt\\
&\leq& C\ell^2 \mathcal Z_J(T)+
C\int_0^T \left(\left\|\varrho_0^{-\frac\gamma2}G\right\|_4^4 +\left\|\varrho_0^{1-
\frac\gamma2}\vartheta\right\|_4^4\right)dt,
\end{eqnarray*}
for any $\ell\geq\bar\ell_0$. Next, it follows from Lemma \ref{lemwineq} that
\begin{eqnarray*}
  \int_0^T\left\|\varrho_0^{-\frac\gamma2}G\right\|_4^4dt
  &\leq& C\int_0^T
  \left(\left\|\varrho_0^{-\frac\gamma2}G\right\|_2^4+\left\|\varrho_0^{-\frac\gamma2}G \right\|_2^3
  \left\|\varrho_0^{-\frac{\gamma+1}{2}}G_y\right\|_2\right)dt\\
  &\leq& C\left(\sup_{0\leq t\leq T} \|\varrho_0^{-\frac\gamma2}G\|_2^2
+\int_0^T \|\varrho_0^{-\frac{\gamma+1}{2}}G_y\|_2^2
dt\right)^2= C\mathcal Z_G^2(T).
\end{eqnarray*}
Similarly,
\begin{equation*}
\int_0^T\left\|\varrho_0^{1-
\frac\gamma2}\vartheta\right\|_4^4dt\leq C\left(\sup_{0\leq t\leq T}\|\varrho_0^{1-\frac\gamma2}\vartheta\|_2^2
+\int_0^T\|\varrho_0^{\frac{1-\gamma}{2}}\vartheta_y\|_2^2
dt\right)^2=C\mathcal Z_\vartheta^2(T).
\end{equation*}
Therefore, the second conclusion holds.
\end{proof}

\subsection{The De Giorgi iteration}
\label{subsecDGth}
In this subsection, we derive the estimates for $\vartheta_\ell$ by the De Giorgi iteration.

\begin{proposition}
\label{PropDGThetea}
Set
\begin{equation*}
  Q_\ell:=\sup_{0\leq t\leq T}\|\varrho_0^{1-\gamma}(\vartheta_\ell)_+\|_2^2 +\int_0^T\|\varrho_0^{\frac12-\gamma}\partial_y(\vartheta_\ell)_+\|_2^2dt.
\end{equation*}
Then, it holds that
\begin{equation*}
  Q_\ell\leq \frac{C(1+\ell)^2}{(\ell-m)^3} (\mathcal Z_J^\frac12(T)+\mathcal Z_\vartheta(T)+\mathcal Z_G(T))Q_m^2 \qquad\forall\ell>m\geq\bar\ell_0,
\end{equation*}
where $C$ is a positive constant depending only on $\kappa,\gamma,c_v,\bar\varrho,K_1,K_2,T,\underline J_T,$ and $\bar J_T$.
\end{proposition}

\begin{proof}
By Proposition \ref{PropWL2EstThetaell}, one has that, for any $\ell\geq\bar\ell_0$,
\begin{eqnarray}
Q_\ell&\leq&C\int_0^T\int_\mathbb R  \left(\left|\varrho_0^{-\frac\gamma2}G\right|^2+\left|\varrho_0^{1-
\frac\gamma2}\vartheta\right|^2+\ell
\left|\varrho_0^{-\frac12}J_y\right| \right)
  \varrho_0^{1-\gamma}(\vartheta_\ell)_+ dydt. \label{DGTHTA1}
\end{eqnarray}
For any $(y,t)\in\left\{(y,t)\big|\vartheta_\ell>0\right\}$ and $m<\ell$, it is clear that
\begin{eqnarray*}
  (\vartheta_m)_+(y,t)
  &\geq&(\ell-m)\varrho_0^{\gamma-1}(y)e^{\overline M_Tt}\geq (\ell-m)\varrho_0^{\gamma-1}(y) ,
\end{eqnarray*}
and, thus,
\begin{equation}
\label{ml}
  1\leq\frac{\varrho_0^{1-\gamma}}{\ell-m}(\vartheta_m)_+, \qquad\mbox{on } \left\{(y,t)\big|\vartheta_\ell>0\right\} \quad\forall m<\ell.
\end{equation}

Using (\ref{ml}) and noticing that $(\vartheta_\ell)_+\leq(\vartheta_m)_+$, for $m<\ell$, one can get
\begin{eqnarray}
&&\int_0^T\int_\mathbb R   \left|\varrho_0^{-\frac\gamma2}G\right|^2
\varrho_0^{1-\gamma}(\vartheta_\ell)_+ dydt\nonumber\\
&\leq &\int_0^T\int_\mathbb R   \left|\varrho_0^{-\frac\gamma2}G\right|^2
\varrho_0^{1-\gamma}(\vartheta_m)_+
\left|\frac{\varrho_0^{1-\gamma}(\vartheta_m)_+}{\ell-m}\right|^3dydt\nonumber\\
&\leq& \frac{1}{(\ell-m)^3}\left(\int_0^T\left\|\varrho_0^{-\frac\gamma2}G\right\|_6^6dt\right)^{\frac13}
\left(\int_0^T\left\|\varrho_0^{1-\gamma}(\vartheta_m)_+\right\|_6^6dt\right)^{\frac23}.
\label{DGTHTA1-1'}
\end{eqnarray}
Lemma \ref{lemwineq} implies that
\begin{equation}
  \int_0^T\left\|\varrho_0^{-\frac\gamma2}G\right\|_6^6dt
  \leq C\int_0^T\left(\left\|\varrho_0^{-\frac\gamma2}G\right\|_2^6+
  \left\|\varrho_0^{-\frac\gamma2}G\right\|_2^4\left\|\varrho_0^{-\frac{\gamma+1}{2}}G_y \right\|_2^2\right)dt\leq C\mathcal Z_G^3,\label{DGTHTA1-1.1}\\
\end{equation}
and, similarly,
\begin{eqnarray}
   \int_0^T\left\|\varrho_0^{1-\gamma}(\vartheta_m)_+\right\|_6^6dt
    \leq  CQ_m^3. \label{DGTHTA1-1.2}
\end{eqnarray}
Substituting
(\ref{DGTHTA1-1.1})--(\ref{DGTHTA1-1.2}) into (\ref{DGTHTA1-1'}) yields
\begin{eqnarray}
  &&\int_0^T\int_\mathbb R   \left|\varrho_0^{-\frac\gamma2}G\right|^2
\varrho_0^{1-\gamma}(\vartheta_\ell)_+ dydt\leq \frac{C\mathcal Z_G Q_m^2}{(\ell-m)^3},\quad\ell>m\geq\bar\ell_0.\label{DGTHTA1-1}
\end{eqnarray}
Similarly, one can show that
\begin{eqnarray}
  &&\int_0^T\int_\mathbb R   \left|\varrho_0^{1-\frac\gamma2}\vartheta\right|^2
\varrho_0^{1-\gamma}(\vartheta_\ell)_+ dydt\leq  \frac{C\mathcal Z_\vartheta Q_m^2}{(\ell-m)^3},\quad\ell>m\geq\bar\ell_0.\label{DGTHTA1-2}
\end{eqnarray}

Next, it follows from (\ref{ml}) and the fact that $(\vartheta_\ell)_+\leq(\vartheta_m)_+$, for $\ell>m$,
that
\begin{eqnarray}
  \int_0^T\int_\mathbb R\left|\frac{J_y}{\sqrt{\varrho_0}}\right|\varrho_0^{1-\gamma}
  (\vartheta_\ell)_+dydt
  &\leq&\int_0^T\int_\mathbb R\left|\frac{J_y}{\sqrt{\varrho_0}}\right|\varrho_0^{1-\gamma}
  (\vartheta_m)_+\left|\frac{\varrho_0^{1-\gamma}(\vartheta_m)_+}{\ell-m}
  \right|^3dydt \nonumber\\
  &\leq&\frac{1}{(\ell-m)^3}\sup_{0\leq t\leq T}\left\|\frac{J_y}{\sqrt{\varrho_0}}\right\|_2
  \int_0^T \|\varrho_0^{1-\gamma}(\vartheta_m)_+\|_8^4dt\nonumber\\
  &\leq&\frac{\mathcal Z_J^{\frac12}}{(\ell-m)^3}
  \int_0^T \|\varrho_0^{1-\gamma}(\vartheta_m)_+\|_8^4dt.\label{DGTHTA1-3'}
\end{eqnarray}
By Lemma \ref{lemwineq}, it holds that
\begin{eqnarray}
  &&\int_0^T\|\varrho_0^{1-\gamma}(\vartheta_m)_+\|_8^4dt\nonumber\\
  &\leq&C\int_0^T\left(\|\varrho_0^{1-\gamma}(\vartheta_m)_+\|_2^4
  +\|\varrho_0^{1-\gamma}(\vartheta_m)_+\|_2^{\frac52}
  \|\varrho_0^{\frac12-\gamma}\partial_y(\vartheta_m)_+\|_2^{\frac32}\right)  dt\nonumber\\
  &\leq&C\left(\sup_{0\leq t\leq T}\|\varrho_0^{1-\gamma}(\vartheta_m)_+\|_2^2+
  \int_0^T \|\varrho_0^{\frac12-\gamma}\partial_y(\vartheta_m)_+\|_2^2dt\right)^2=CQ_m^2.
  \label{DGTHTA1-3.1}
\end{eqnarray}
Combining (\ref{DGTHTA1-3'}) with (\ref{DGTHTA1-3.1}) leads to
\begin{eqnarray}
   \int_0^T\int_\mathbb R\left|\frac{J_y}{\sqrt{\varrho_0}}\right|\varrho_0^{1-\gamma}
  (\vartheta_\ell)_+dydt
  &\leq&\frac{C\mathcal Z_J^{\frac12}Q_m^2}{(\ell-m)^3}.\label{DGTHTA1-3}
\end{eqnarray}

Substituting (\ref{DGTHTA1-1}), (\ref{DGTHTA1-2}), and (\ref{DGTHTA1-3}) into (\ref{DGTHTA1}) yields the conclusion.
\end{proof}

\subsection{Upper bound of the entropy}
\label{ssecSup}
We are now ready to establish the uniform upper bound for the entropy.
\begin{theorem}
  \label{ThmSuppBd}
Assume that (\ref{ass-v}), (\ref{finitemass}), (\ref{HSLOW}), (\ref{HSLOW2}), and (\ref{ass-singular}) hold, and the initial entropy is bounded from above. Then, the unique global solution obtained in Theorem \ref{ThmGloUni-v} satisfies
\begin{equation*}
 \sup_{(y,t)\in\mathbb R\times(0,T)} s\leq C\log(2+\bar\ell_0+\mathcal Z(T)),
\end{equation*}
for any positive time $T$, where
$\mathcal Z(T)=\mathcal Z_J^\frac12(T)+\mathcal Z_\vartheta(T)+\mathcal Z_G(T)$
and $C$ is a positive constant depending only on $c_v$, $\kappa$, $\gamma$, $\bar\varrho$, $K_1$, $K_2$, $T$, $\underline J_T$, and $\bar J_T$.
\end{theorem}

\begin{proof}
It follows from Proposition \ref{PropDGThetea} that
\begin{equation*}
  Q_\ell\leq \frac{C(1+\ell)^2}{(\ell-m)^3}\mathcal Z(T) Q_m^2 ,\qquad\forall\ell>m\geq\bar\ell_0.
\end{equation*}
One can check easily that $Q_\ell$ is non-increasing in $\ell$. Therefore, Lemma \ref{lemiteration} implies
$Q_{\bar\ell_0+d}=0$, with
$d=2+2\big(2+\bar\ell_0+C\mathcal Z(T)\big)^{22}Q_{\bar\ell_0}^2.$
Hence, $(\vartheta_{\bar\ell_0+d})_+\equiv0$, which gives
\begin{eqnarray*}
  \vartheta\leq\left(\bar\ell_0+d\right)\varrho_0^{\gamma-1}e^{\overline M_Tt}
  \leq\left(\bar\ell_0+d\right)\varrho_0^{\gamma-1}e^{\overline M_TT},
\end{eqnarray*}
and, consequently,
\begin{eqnarray}
  s&=&c_v\left(\log\frac{R}{A}+(\gamma-1)\log J+\log\vartheta-(\gamma-1)\log\varrho_0\right)\nonumber\\
  &\leq&c_v\left(\log\frac{R}{A}+(\gamma-1)\log\bar J_T+\log(\bar\ell_0+d)+\overline M_TT\right).
  \label{ups1}
\end{eqnarray}

Proposition \ref{PropWL2EstThetaell} shows that $Q_{\bar\ell_0}\leq C(1+\bar\ell_0^2)\mathcal Z^2(T)$, and, thus,
$d \leq C (2+\bar\ell_0+\mathcal Z(T))^{30}.$
This and (\ref{ups1}) give the desired conclusion.
\end{proof}

\section{Appendix}
In this appendix, we prove Lemma \ref{lemiteration}.

\begin{proof}[Proof of Lemma \ref{lemiteration}]
It follows from the assumption that
\begin{equation}\label{ite1}
f(\ell)\leq\frac{2^\alpha M_0\ell^\alpha}{(\ell-m)^\beta}f^\sigma(m),\quad\forall\ell>m\geq m_0+1.
\end{equation}
Let $d_0\geq1$ be a positive number to be determined later, and set
$$
\ell_k=m_0+1+\left(1-\frac{1}{2^k}\right) d_0, \quad k=0,1,2,\cdots.
$$
Then, choosing $\ell=\ell_{k+1}$ and $m=\ell_k$ in (\ref{ite1}), and noticing that $\ell_{k+1}\leq m_0+1+d_0$, one deduces that
\begin{eqnarray*}
  f(\ell_{k+1})&\leq&M_02^\alpha\ell_{k+1}^\alpha(\ell_{k+1}-\ell_k)^{-\beta}f^\sigma(\ell_k)\\
  &\leq&M_02^\alpha(m_0+1+d_0)^\alpha(2^{-(k+1)}d_0)^{-\beta}f^\sigma(\ell_k)\\
  &=&M_02^{k\beta+\alpha+\beta}\left(\frac{m_0+1}{d_0^{\beta/\alpha}}+\frac{1}
  {d_0^{\beta/\alpha-1}}\right)^\alpha f^\sigma(\ell_k),
\end{eqnarray*}
from which, recalling that $d_0\geq1$ and noticing that $\frac\beta\alpha>1$, one obtains
$$
 f(\ell_{k+1})
  \leq M^{k\beta+2\alpha+\beta+1}f^\sigma(\ell_k),\quad k=0,1,2,\cdots,
$$
with $M=M_0+m_0+2$, which can be written equivalently as
\begin{equation}\label{ite2}
  M^{a(k+1)+b}f(\ell_{k+1})\leq [M^{ak+b}f(\ell_k)]^\sigma,\quad k=0,1,2,\cdots,
\end{equation}
where $a=\frac{\beta}{\sigma-1}$ and $b=\frac{2\alpha+\beta+1}{\sigma-1}+\frac{\beta}{(\sigma-1)^2}.$
It follows from (\ref{ite2}) that
$$
M^{a(k+1)+b}f(\ell_{k+1})\leq (M^{a+b}f(\ell_1))^{\sigma^k},
$$
which implies, due to $M\geq2, a>0,$ and $b>0$, that
\begin{equation}
  \label{ite3}
  f(\ell_{k+1})\leq (M^{a+b}f(\ell_1))^{\sigma^k},\quad k=1,2,\cdots.
\end{equation}

Choosing $\ell=\ell_1$ and $m=\ell_0$ in (\ref{ite1}) leads to
\begin{eqnarray*}
  f(\ell_1)&\leq&2^\alpha M_0\ell_1^\alpha(\ell_1-\ell_0)^{-\beta} f^\sigma(\ell_0)
   \leq \frac{2^{\alpha+\beta}M_0}{d_0^{\beta-\alpha}}\left(m_0+2\right)^\alpha f^\sigma(m_0+1).
\end{eqnarray*}
It follows from this and the monotonicity of $f$ that
\begin{eqnarray*}
  f(\ell_1)&\leq&\frac{M^{2\alpha+\beta+1}}{d_0^{\beta-\alpha}}f^\sigma(m_0).
\end{eqnarray*}
Therefore,
\begin{eqnarray*}
  M^{a+b}f(\ell_1)=M^{\frac{2\alpha+2\beta+1}{\sigma-1}+\frac{\beta}{(\sigma-1)^2}}f(\ell_1)
  \leq M^{\frac{2\alpha+2\beta+1}{\sigma-1}+\frac{\beta}{(\sigma-1)^2}+2\alpha+\beta+1}
  \frac{f^\sigma(m_0)}{d_0^{\beta-\alpha}}\leq\frac12,
\end{eqnarray*}
provided that
$d_0=\left(2 f_0^\sigma M^{\frac{2\alpha+2\beta+1}{\sigma-1}+\frac{\beta}{(\sigma-1)^2}+2\alpha+\beta+1}
  \right)^{\frac{1}{\beta-\alpha}}+1$.

It follows from (\ref{ite3}) that
$$
f(m_0+1+d_0)\leq f(\ell_{k+1})\leq (M^{a+b}f(\ell_1))^{\sigma^k}
\leq\left(\frac12\right)^{\sigma^k},\quad k=0, 1,2,\cdots.
$$
Passing $k\rightarrow\infty$ in the above yields $f(m_0+1+d_0)=0$, so the conclusion follows.
\end{proof}

\section*{Acknowledgments}
J.L. was supported in part by the National Natural Science Foundation of China grants 11971009, 11871005, and 11771156, by the Natural Science Foundation of Guangdong Province
grant 2019A1515011621, and by the South China Normal University start-up grant 550-8S0315. Z.X. was supported in part by the
Zheng Ge Ru Foundation and by Hong Kong RGC Earmarked Research Grants
CUHK 14305315, CUHK 14302819, CUHK 14300917, and CUHK 14302917.

\end{document}